\documentclass[12pt,a4paper]{article}
\usepackage{geometry}
\usepackage[utf8]{inputenc}
\usepackage[english]{babel}
\usepackage{amssymb, amsthm, amsmath} 
\usepackage{verbatim}
\usepackage{makeidx}
\usepackage{graphicx}
\usepackage{enotez}
\usepackage{quoting}
\usepackage{url}
\usepackage{eucal}
\usepackage[color]{changebar}
\cbcolor{red}

\usepackage{sectsty}
\usepackage{lipsum}
\subsubsectionfont{\small}
\makeatletter
\def\blfootnote{\xdef\@thefnmark{}\@footnotetext}
\makeatother

\usepackage{xcolor}

\theoremstyle{plain}
\newtheorem{lem}{Lemma}[section]
\newtheorem{coro}[lem]{Corollary}
\newtheorem{teo}[lem]{Theorem}

\newtheorem{propo}[lem]{Proposition}
\theoremstyle{definition}
\newtheorem{exa}[lem]{Example}
\newtheorem{cons}[lem]{Construction}
\newtheorem{rema}[lem]{Remark}
\newtheorem{defi}[lem]{Definition}

\usepackage{tikz}
\usepackage{tikz-cd}

\DeclareMathAlphabet{\mathbbe}{U}{bbold}{m}{n}
\newcommand{\simplexcategory}{\mathbbe{\Delta}}
 \newcommand{\name}[1]{\ulcorner #1\urcorner}
\usepackage{hyperref}
\hypersetup{
    colorlinks=true,
    linkcolor=blue,
    filecolor=magenta,      
    urlcolor=cyan,
}

\usepackage{stmaryrd}
\newcommand{\actto}{\rightarrow\Mapsfromchar}
\tikzset{  act /.tip = >|}

\tikzset{pullback/.style={minimum size=1.2ex,path picture={
\draw[opacity=1,black,-,#1] (-0.5ex,-0.5ex) -- (0.5ex,-0.5ex) -- (0.5ex,0.5ex);%
}}}

 \usepackage[all]{xy}
 \newcommand{\isopil}{\stackrel{\raisebox{0.1ex}[0ex][0ex]{\(\sim\)}}%
			{\raisebox{-0.15ex}[0.28ex]{\(\rightarrow\)}}}

\newcommand{\drpullbacko}[1][dr]{\save*!/#1-4ex/#1:(-1,1)@^{|-}\restore}

\newcommand{\drpullback}{\arrow[phantom]{dr}[very near start,description]{\lrcorner}}
\newcommand{\dlpullback}{\arrow[phantom]{dl}[very near start,description]{\llcorner}}

\newcommand{\Map}{\operatorname{map}}
\newcommand{\Obj}{\operatorname{Obj}}
\newcommand{\op}{\operatorname{op}}

\newcommand{\CULF}{\operatorname{culf}}

\newcommand{\identity}{\operatorname{id}}
\newcommand{\Dec}{\operatorname{Dec}}
\newcommand{\longt}{\operatorname{long}}
\newcommand{\Fib}{\operatorname{Fib}}
\newcommand{\hFib}{\operatorname{hFib}}

\newcommand{\mfunctor}{\operatorname{M}}
\newcommand{\ledge}{\operatorname{\varpi}}
\newcommand{\Ar}{\operatorname{Ar}}

\newcommand{\cart}{\operatorname{cart}}
\newcommand{\dominio}{\operatorname{dom}}

\newcommand{\fffunctor}{\operatorname{f.f.}}

\newcommand{\act}{\operatorname{act}}

\newcommand{\pil}{\pi_{\operatorname{last}}}
\newcommand{\pilong}{\pi_{\operatorname{long}}}
\newcommand{\pif}{\pi_{\operatorname{first}}}
\newcommand{\piend}{\pi_{\operatorname{endvertex}}}
\newcommand{\strs}{\operatorname{str}}

\newcommand{\tildeb}[1]{\stackrel{\sim}{\smash{#1}\rule{0pt}{1.1ex}}}

\newcommand{\kat}[1]{\mathbf{#1}}
\newcommand{\Set}{\kat{Set}}

\newcommand{\Grpd}{\kat{Grpd}}

\newcommand{\cDcmp}{\kat{cDcmp}}

\newcommand{\Interval}{\kat{aInt}}

  \makeatletter
\renewcommand{\tableofcontents}{%
   \begin{center}
\begin{minipage}{115mm}
   \begin{center}
                \bf{\contentsname}
        \end{center}
   \footnotesize
   \begin{center}
                \@starttoc{toc}
        \end{center}
\end{minipage}
        \end{center}
        \addvspace{3em \@plus\p@}
}

\makeatother

\newtheorem{taller}[lem]{$\!\!$}

\newenvironment{blanko}[1]%
{\begin{taller}{\normalfont\bfseries  #1}\normalfont}%
{\end{taller}}

\newenvironment{blanko*}[1]{\begin{list}{\bf {#1} }%
{\setlength{\labelsep}{0mm}\setlength{\leftmargin}{0mm}%
\setlength{\labelwidth}{0mm}\setlength{\listparindent}{\parindent}%
\setlength{\parsep}{\parskip}\setlength{\partopsep}{0mm}}%
\item%
}{\end{list}}

\begin{document}

\title{\Large\textbf{THE GÁLVEZ--KOCK--TONKS CONJECTURE FOR LOCALLY DISCRETE DECOMPOSITION SPACES}\footnote{Previously circulated under the title THE GÁLVEZ--KOCK--TONKS CONJECTURE FOR DISCRETE DECOMPOSITION SPACES}} 
\author{\normalsize{WILSON FORERO}} 
\date{}
\maketitle

\begin{abstract}
G\'alvez-Carrillo, Kock, and Tonks~\cite{GTK3} constructed a decomposition space $U$ of all M\"obius intervals, as a recipient of Lawvere's interval construction for M\"obius categories, and conjectured that $U$ enjoys a certain universal property: for every M\"obius decomposition space $X$, the space of culf functors from $X$ to $U$ is contractible. In this paper, we work at the level of homotopy 1-types to prove the first case of the conjecture, namely for locally discrete decomposition spaces. This provides also the first substantial evidence for the general conjecture.  
 
This case is general enough to cover all locally finite posets, Cartier--Foata monoids, M\"obius categories and strict (directed) restriction species. The proof is 2-categorical. First, we construct a local strict model of $U$, which is then used to show by hand that the Lawvere interval construction, considered as a natural transformation, does not admit other self-modifications than the identity.  
\end{abstract}

\tableofcontents

\section*{Introduction}
\label{sec:intro}
\addcontentsline{toc}{section}{Introduction}
Incidence algebras and M\"obius inversion form a cornerstone of combinatorics. It has important applications in many areas of
mathematics. Beyond the original applications in number theory (see Hardy and Wright~\cite{Hardy-Wright}) and group theory
(\cite{Weisner} and \cite{Hall}), one can cite applictions in probability theory~\cite{Nica-Speicher} and algebraic topology~\cite{Hopkins-Kuhn-Ravenel}, and it is also closely related to Hopf-algebraic renormalisation~\cite{Kock:Mrenorm} in quantum field theory.

Since Rota formalised the theory~\cite{Rota1987}, \cite{Rota} (on the grounds of previous contributions~\cite{Weisner}, \cite{Hall})
the standard framework for the theory has been that of posets, but the theory has also been
used in the context of monoids~\cite{Cartier}, and in the more
general framework of certain categories called M\"obius categories,
introduced by Leroux~\cite{Lero}. The uniform appearance of
the M\"obius inversion formula across all application areas led
Lawvere~\cite{Lawnotes} in the 1980s to discover that
there is a universal M\"obius function which induces all other
M\"obius functions. It is an `arithmetic function' on a certain Hopf
algebra of M\"obius intervals. A category is an interval if it has
an initial and a terminal object~\cite{LawInterval}, and the M\"obius condition is a certain finiteness condition. This Hopf algebra has the property
that it receives a canonical coalgebra homomorphism from every
incidence coalgebra of a M\"obius category. This includes all
locally finite posets and all the monoids considered in
\cite{Cartier}. Lawvere's work remained unpublished for some
decades, but it is cited in influential texts from that time, such
as Joyal~\cite{Joyal:1981} and Joni--Rota~\cite{Rota}.
Independently, Ehrenborg~\cite{REhre} constructed a closely
related Hopf algebra, but less universal. It only accounts for
intervals in posets. In both cases, the universal object can be
interpreted as the colimit of all incidence coalgebras of intervals.
The possibility of this is closely related to the local nature of
coalgebras, expressed for example in the well-known fact that every
coalgebra is the colimit of its finite-dimensional subcoalgebras,
see Sweedler~\cite{Sweedler}.

Lawvere's discovery did not appear in print until
Lawvere--Menni~\cite{Law} in 2010. In that work the
authors took an important step towards explaining the universal
property by lifting the construction of the Hopf algebra of M\"obius
intervals to the objective level. This means that its
comultiplication is realised as something called a pro-comonoidal
structure on certain extensive categories. The original Hopf algebra
is exhibited as being only a numerical shadow of this
categorical construction. There are at least two precursors to the
idea of a more objective approach to incidence algebras. One is
given by Joyal~\cite{Joyal:1981}. In his foundational paper on
species, there is a final section where he considers certain
decomposition structures on categories (that final section has
little to do with species). Another is in the work of
D\"ur~\cite{Dur} who constructed incidence coalgebras of
certain categorical and simplicial structures.

However,
\begin{quote}
{\em many coalgebras, bialgebras and Hopf algebras in
combinatorics are not of incidence type},
\end{quote}
meaning that they cannot arise directly as the incidence coalgebra of any
M\"obius category. In fact the Lawvere--Menni Hopf algebra is not
of incidence type. This gives the somewhat unsatisfactory situation that the universal object is not of the same type as the objects it is universal for.

A solution to this problem was found by
G\'alvez, Kock, and Tonks~\cite{GTK1,GTK2,GTK3}. They discovered
that the incidence coalgebra construction and M\"obius inversion
make sense for objects more general than M\"obius categories (recall
that M\"obius categories include locally finite posets and
Cartier--Foata monoids). These are completely new objects in this context which they
call {\em decomposition spaces}. They are certain simplicial objects
subject to an axiom that expresses decomposition, in the same way as
the Segal condition (which characterises categories among simplicial
sets) expresses composition. Decomposition spaces are the same thing
as the $2$-Segal spaces of Dyckerhoff and Kapranov~\cite{DK} (see
\cite{unital} for the last piece of this equivalence). It seems
likely that all combinatorial coalgebras, bialgebras and Hopf
algebras (with nonnegative section coefficients) arise from the
incidence coalgebra construction of decomposition spaces. This has
been shown for most of Schmitt's examples~\cite{schmitt_1993} (restriction species in
G\'alvez--Kock--Tonks~\cite{GKT:restr} and hereditary species in
Carlier~\cite{carlier2019hereditary}). G\'alvez--Kock--Tonks (as also
Dyckerhoff--Kapranov) work in the fully homotopical setting of
simplicial $\infty$-groupoids, but already the discrete case of the
notion is very rich, as exemplified by work of Bergner et
al.~\cite{BOORS} and Kock--Spivak~\cite{Kock-Spivak}, who relate
the notion to constructions in algebraic topology and category
theory.

G\'alvez, Kock and Tonks~\cite{GTK3} showed that the Lawvere--Menni
Hopf algebra is the incidence coalgebra of a decomposition space
$U$. With this discovery the universal property could be stated,
showing its nature as a moduli space: {\em For any decomposition
space $X$ the mapping space $\Map(X,U)$ is contractible.} This
statement is the G\'alvez--Kock--Tonks conjecture, which is the
objective of the present paper. The mapping space is the space of
culf maps, as detailed further below. Culf maps were identified to
play a key role already in the work of Lawvere and
Menni~\cite{Law}. 

Lawvere's original work (suitably upgraded to the new
context) shows that $\Map(X,U)$ is not empty: it contains $I \colon X \to
U$, which is essentially Lawvere's interval construction. G\'alvez, Kock and Tonks~\cite{GTK3} were able to establish one further
ingredient of the conjecture, namely that $\Map(X,U)$ is connected,
meaning that every map is homotopy equivalent to $I$. The finer
property of being contractible is the full homotopy uniqueness
statement, that not only is every map equivalent to $I$: it is so
uniquely (in a coherent homotopy sense).

The homotopy content was one of the reasons for G\'alvez, Kock and
Tonks to develop the whole theory in a homotopy setting:
decomposition spaces are defined to be certain simplicial
$\infty$-groupoids, and everything is fully homotopy invariant. It
is an important insight of higher category theory (see for example
Lurie~\cite{HTT}) that a universal object cannot exist in any
truncated situation. Most famous is the fact that the topos of sets
($0$-types) contains a classifier for monomorphisms ($(-1)$-types)
but cannot contain a classifier for sets ($0$-types), and that for
these to be classified one needs the $2$-topos of groupoids
($1$-types), and to classify $1$-types one needs to $3$-topos of
$2$-types, and so on. Only in the limit is it possible to find a
classifier for general homotopy types ($\infty$-groupoids) in the
$\infty$-topos of $\infty$-groupoids.

At the moment, the technical difficulties of the general
G\'alvez--Kock--Tonks conjecture are too big.

In the present paper, the first case of the conjecture is proved.
Working at the level of $1$-types, we define the simplicial groupoid
$U$ of discrete intervals (i.e.~intervals that are simplicial sets rather than simplicial spaces), and show that:

{\bf Theorem.} (Theorem~\ref{corocontractibleU}.) {\em $\Map(X,U)$ is a contractible
$1$-groupoid for every $1$-truncated locally discrete decomposition space $X$.}

This is the first substantial evidence for the full conjecture. According to the discussion above, the expected generality for $X$ is that of discrete decomposition spaces, but in fact (as kindly pointed out by the anonymous referee) the proofs work the same for a broader class of decomposition spaces, namely those $1$-truncated decomposition spaces with the property that all their intervals are discrete. Clearly, discrete decomposition spaces have this property, so the level of generality already covers all the classical theory of incidence algebras and M\"obius
inversion in combinatorics, since locally finite posets,
Cartier--Foata monoids, M\"obius categories, and Schmitt's examples
are all $0$-truncated simplicial spaces. In particular it gives
finally a firm formalisation of Lawvere's intuition that the
interval construction should be universal in some sense.  As a
particular case it establishes also the universal property of the
Ehrenborg Hopf algebra.

The idea of the proof is the following. The theorem, namely the contractibility of the $1$-groupoids $\Map(X,U)$, is a
$2$-categorical statement. The proof we give is based on
$2$-category theory. However, a direct verification of the statement
seems intractable, due to coherence problems. The difficulty is that
$U \colon \simplexcategory^{\op} \rightarrow \Grpd$ is only a pseudo-simplicial groupoid. Jardine~\cite{Jardine} has identified all the $2$-cell
structure and the $17$ coherence conditions for pseudo-simplicial groupoids. The definition of modification in this context requires
compatibility with all that. The strategy to overcome this difficulty
is to build a local strict model, a kind of neighbourhood $U_X
\subset U$ around the intervals of a given locally discrete decomposition
space $X$. The bulk of the paper is concerned with setting up this
local model and show that it is strict. To construct this, we
introduce a stricter algebraic notion of interval, where the initial
and terminal objects are not just given as properties of a discrete
decomposition space, but are carried around as data, in the notion
of chosen initial and terminal objects. This focus is inspired by
the work in another context of Batanin and Markl on operadic
categories~\cite{BM}. This is quite technical, but the
benefit is to achieve a strict local model $U_X$ which is shown to
be a strict simplicial groupoid and a complete decomposition groupoid, and to receive a strict version of
the interval construction. With this strict local model in place,
the local version of the contractibility of $\Map(X,U_X)$ can be
established with $2$-category theory by showing that $I \colon X \to U_X$,
interpreted as a natural transformation, does not admit other
self-modifications than the identity modification. In the end this
check is not so difficult.

At this point is it natural to ask whether the techniques
developed here can be applied or refined to prove the conjecture in
full generality. Unfortunately this is not very likely, or it would
require new conceptual simplifications and new technical tools. The
point is that the proof relies on explicit strictification through
strict models constructed through explicit data standing in for
universal properties (initial and terminal objects, at the level of
the objects involved, and strict simplicial objects at the higher
level). The level of $2$-categories is in practice the highest level
where this kind of technique can be applied, and already at this
level it is quite tricky to find the balance between properties and
property-like structures. The next level, which would be the
contractibility of the $2$-groupoid $\Map(X,U)$ for $X$ a
$1$-truncated decomposition space, and $U$ the universal simplicial
$2$-groupoid of $1$-intervals, seems out of reach. It seems more
promising to pass directly to the homotopical setting of the full
conjecture, aiming at using the theory of $(\infty,2)$-categories.
But it seems quite daunting to carry over the explicit
strictification strategies to that setting.

In conclusion, the present contribution may be seen as only a small
step towards the full conjecture, but it is nevertheless an
important step (and actually the first step ever carried out), and
enough to cover all the cases envisioned by Lawvere, which includes
essentially all the examples from classical combinatorics. It is also already a striking
example of how higher category theory (in this case $2$-categories)
serves to solve problems even in discrete mathematics.

\subsection*{Organisation of the paper}

We begin in Section \ref{sec:preli} with a brief review of basic notions and some results  on homotopy pullbacks of groupoids. In \ref{sec:decspace}, we recall from~\cite{GTK1} some basic notions and results of the theory of decomposition groupoids. In \ref{subsec:incidencecoalgebra} we review the notion of incidence coalgebra of a decomposition space. In \ref{subsec:culf}, we briefly explain the culf condition for a simplicial map. In \ref{subsec:decalage}, we review the notion of decalage. 

In Section \ref{subsec:slices},  we introduce some necessary material relating to the notion of slice and coslice of decomposition groupoids. Furthermore, we give the definition of interval (\ref{definition: thenicalinterval}). 

In Section \ref{sec: rigid ds}, we identify the level of generality. In principle what we need to impose is that all the intervals of $X$ are discrete, but for technical reasons we also impose some strictness. To be more precise we will work with strict simplicial groupoids such that all active-inert squares are strict pullbacks and such that $d_1$ is a discrete isofibration. (It follows that all the strict pullbacks are also homotopy pullbacks.) For short we shall call such decomposition groupoids \emph{rigid} (\ref{defi: discrete intervals dc}). Furthermore, we explain the concept of chosen initial and chosen terminal object (\ref{defi: choseterminalobject and initial}). Also, the notion of discrete algebraic interval (\ref{defi: interval}) and some results for discrete algebraic intervals are given, in particular a lifting property (\ref{existssimplecesinterval} and \ref{3existencecell}).

In Section \ref{sec:Factorisation systems}, we construct  the stretched-culf factorisation system in the category of discrete algebraic intervals. Furthermore, we introduce important working tools (\ref{lemma: culf are like mono} and  \ref{ortoghonalfactsystem}) that will be useful in next sections. 

In Section \ref{sec:U}, we define the decomposition groupoid of all discrete algebraic intervals $U$~\cite{GTK3}. In \ref{subsec:The complete decomposition groupoid $U_X$}, we construct a strict simplicial groupoid $U_X$ (\ref{proofsimplicialsetux}) that only contains the information about the discrete algebraic intervals of a fixed rigid decomposition groupoid $X$ and prove that $U_X$ is a complete decomposition groupoid (\ref{proofdspaceux} and \ref{UXcomplete}). Furthermore, we define a simplicial map $I \colon X \rightarrow U_X$ and prove that $I$ is culf (\ref{propoculfix}).  In \ref{subsec: interval cons of interval}, we explain the interval construction of and interval.  Furthermore, we compare $U_X$ with a strictification $\tildeb{U}$ of $U$ suggested by the referee in \ref{subsec:strictification}.  

In Section \ref{sec:GKTconjecture}, we come to the Gálvez--Kock--Tonks conjecture formulated in~\cite{GTK3}, and we prove a partial result (\ref{conecctedconjecture}) about the connectedness of the mapping space $\Map_{\cDcmp}(X,U)$ in the case of rigid decomposition groupoids.  In \ref{subsec:modif}, we use the concept of modification (\ref{defimodififorconj}) to prove a truncated version of the conjecture, the case of rigid decomposition groupoids. We first show that $\Map_{\cDcmp}(X,U_X)$ is contractible (\ref{teocontractiblefordesset}) and from this we deduce that the groupoid $\Map_{\cDcmp}(X,U)$ is contractible (\ref{corocontractibleU}). This is the version of the G\'alvez--Kock--Tonks conjecture that is the main theorem of this paper.

\subsection*{Acknowledgements}

The author would like to thank Joachim Kock for his advice and support all along the project. I am very thankful to the anonymous referee for a very competent and helpful report. It led to many mathematical improvements. In particular, the referee observed that with few adjustments, the theory could be upgraded from discrete decomposition spaces to locally discrete ones. The referee also suggested the elegant description of the strictification $\tildeb{U}$ in \ref{subsec:strictification}. The author was supported by grants number MTM2016-80439-P and PID2020-116481GB-I00(AEI/FEDER, UE) of Spain.

\section{Preliminaries}\label{sec:preli}

\noindent For the convenience of the reader, this section recalls a few background facts and establishes notation. These results are not new. 

Homotopy pullbacks are important to the theory of decomposition spaces. They are examples of homotopy limits, and as such are defined only up to equivalence. The most used result for homotopy pullbacks is the prism lemma:
\begin{lem}\label{lemma: prism old}
Consider a diagram
\[\begin{tikzcd}
	\cdot & \cdot & \cdot \\
	\cdot & \cdot & \cdot
	\arrow[from=1-1, to=1-2]
	\arrow[from=1-2, to=1-3]
	\arrow[from=1-1, to=2-1]
	\arrow[from=2-1, to=2-2]
	\arrow[from=2-2, to=2-3]
	\arrow[from=1-3, to=2-3]
	\arrow[from=1-2, to=2-2]
	\arrow["\lrcorner"{anchor=center, pos=0.125}, draw=none, from=1-2, to=2-3]
\end{tikzcd}\]
where the right square is a homotopy pullback. Then the left square is a homotopy pullback if and only if the outer diagram is a homotopy pullback.
\end{lem}

A particular case of homotopy pullbacks is given by the homotopy fibres. Given a map of groupoids $p \colon X \rightarrow S$ and an object $s \in S$, the \emph{homotopy fibre} $X_s$ of $p$ over $s$ is the homotopy pullback
\begin{center}
\begin{tikzcd}
X_s \drpullback \arrow[r, " " ] \arrow[d, " "]& X \arrow[d, "p"] \\
1 \arrow[r, "\ulcorner s \urcorner"']& S.
\end{tikzcd}
\end{center}

We use the following standard lemma many times.

\begin{lem}\cite{LK}\label{lemmapullbackfibres}
A square of groupoids 
\[
\begin{tikzcd}
P \drpullback \arrow[r, "u" ] \arrow[d, " "]& Y \arrow[d, ""] \\
X \arrow[r, "f"']& S
\end{tikzcd}
\]
is a homotopy pullback if and only if for each $x \in X$ the induced comparison map $u_x \colon P_x \rightarrow Y_{fx}$ is an equivalence. 
\end{lem}

Since homotopy pullback is defined up to equivalence, for some calculations it is important to work with a specific model. Let's look at one of the models to be used: the homotopy fibre product of a pair of functors $f \colon A \rightarrow C$ and $g \colon B \rightarrow C$ between groupoids is the groupoid $H$ whose objects are triples $(a, \theta, b)$ consisting of objects $a \in A$, $b \in B$ and an isomorphism $\theta \colon f(a) \rightarrow g(b)$ in $C$, and whose arrows $(\alpha, \beta) \colon (a, \theta, b) \rightarrow (a', \theta', b')$ consist of arrows $\alpha \colon a \rightarrow a' \in A$ and $\beta \colon b \rightarrow b' \in B$ such that $g(\beta) \circ \theta = \theta' \circ f(\alpha)$. The groupoid $H$ fits into a homotopy commutative square
\[\begin{tikzcd}
	H & A \\
	B & C,
	\arrow["{\pi_A}", from=1-1, to=1-2]
	\arrow["{\pi_B}"', from=1-1, to=2-1]
	\arrow["f", from=1-2, to=2-2]
	\arrow["g"', from=2-1, to=2-2]
	\arrow["\lrcorner"{anchor=center, pos=0.125}, draw=none, from=1-1, to=2-2]
\end{tikzcd}\]
where $\pi_A \colon H \rightarrow A$ and $\pi_B \colon H \rightarrow B$ are the canonical projections, and the components of the natural isomorphism is given by $\theta$ itself. Note that the projections always are isofibrations \cite{LK}. 

\noindent Another model is possible when one of the two legs $f$ and $g$ is an isofibration. In that case, the strict pullback is also a homotopy pullback \cite[Theorem 1]{Joyal1993}.

We will use the following variation of the prism lemma in Section \ref{subsec:slices}.
\begin{lem}\cite{LK} \label{lemma: prism general}
Consider a diagram
\[\begin{tikzcd}
	\cdot & \cdot & \cdot \\
	\cdot & \cdot & \cdot
	\arrow[from=1-1, to=1-2]
	\arrow[from=1-2, to=1-3]
	\arrow[from=1-1, to=2-1]
	\arrow[from=2-1, to=2-2]
	\arrow[from=2-2, to=2-3]
	\arrow[from=1-3, to=2-3]
	\arrow[from=1-2, to=2-2]
	\arrow["\lrcorner"{anchor=center, pos=0.125}, draw=none, from=1-2, to=2-3]
\end{tikzcd}\]
where the right square is a homotopy fibre product. Then the left square is a strict pullback if and only if the outer diagram is a homotopy fibre product.
\end{lem}

A map of groupoids $f \colon X \rightarrow Y$ is a \emph{monomorphism} when it is fully faithful. Equivalently, its homotopy fibres are ($-1$)-groupoids, that is, are either empty or contractible.

\begin{blanko} 
{Decomposition groupoids}\label{sec:decspace}
\end{blanko}


\noindent This paper is concerned with a truncated case of the G\'alvez--Kock--Tonks conjecture. For that we only have to deal with simplicial groupoids rather than simplicial spaces. For this reason we prefer to use the word decomposition groupoid rather than decomposition space  in all the paper.

The \textit{simplex category} $\simplexcategory$ is the category whose objects are the nonempty finite ordinals and whose morphisms are the monotone maps. These are generated by \emph{coface} maps $d^i \colon [n-1] \rightarrow [n]$, which are the monotone injective functions for which $i \in [n]$ is not in the image, and \emph{codegeneracy} maps $s^i \colon [n+1] \rightarrow [n]$, which are monotone surjective functions for which $i \in [n]$ has a double preimage. We write $d^\bot := d^0$ and $d^\top := d^n$ for the outer coface maps. 

An arrow of  $\simplexcategory$  is termed \emph{active},
and written $g \colon [m] \actto [n]$, if it preserves end-points, $g(0)$ = 0 and $g(m) = n$. An arrow
is termed \emph{inert}, and written $f \colon [m] \rightarrowtail [n]$, if it is distance preserving,
$f(i+1) = f(i)+1$ for $0 \leq i < m$.

\begin{defi}\cite[Definition 3.1]{GTK1}\label{definitiondecompositionspace}
A \emph{decomposition groupoid} is a simplicial groupoid 
$$X \colon \simplexcategory^{\op} \rightarrow \Grpd$$
such that the image of any pushout diagram in $\simplexcategory$ of an active map $g$ along
an inert map $f$ is a homotopy pullback of groupoids,

$$ X\!\left(\!\! \vcenter{\xymatrix{
   [p]  \drpullbacko{}  &  [m]\ar@{->|}[l]_{g'}  \\
   [q]\ar@{>->}[u]^{f'}   & \ar@{->|}[l]^{g} [n] \ar@{>->}[u]_{f}
  }}\right) \qquad=\qquad \vcenter{\xymatrix{
   X_{p}\ar[d]_{{f'}^*} \ar[r]^{{g'}^*} \drpullbacko{}  & X_m\ar[d]^{{f}^*}  \\
   X_q\ar[r]_{{g}^*}   &  X_n .
  }}
  $$

This is equivalent \cite[Proposition 3.5]{GTK1} to requiring that the following diagrams are homotopy pullbacks for all $0 < i < n$:

\begin{center}
\begin{tikzcd}
X_{n+1} \drpullback \arrow[r, "d_{i+1}"]\arrow[d, "d_\bot"']& X_n \arrow[d, "d_\bot"] \\
X_n \arrow[r, "d_i"']& X_{n-1}
\end{tikzcd} \; \; \;
\begin{tikzcd}
X_{n+1} \drpullback \arrow[r, "d_{i}"]\arrow[d, "d_\top"']& X_n \arrow[d, "d_\top"] \\
X_n \arrow[r, "d_i"']& X_{n-1}.
\end{tikzcd}
\end{center}
\end{defi}
\begin{defi}\cite[\S 2.9]{GTK1}\label{definitionsegal}
A simplicial groupoid $X \colon \simplexcategory^{\op} \rightarrow \Grpd$ is called a  \emph{Segal groupoid} if it satisfies the Segal condition,   
\begin{center}
\begin{tikzcd}
X_n \arrow[r, "\simeq"] & X_1 \times_{X_0} \cdots \times_{X_0} X_1  \mbox{ for all } n \geq 0.
\end{tikzcd}
\end{center}
This is equivalent \cite[Lemma 2.10]{GTK1} to requiring that for each $n > 0$  the following diagram is a homotopy pullback
\begin{center}
\begin{tikzcd}
X_{n+1} \drpullback \arrow[r, "d_\top"]\arrow[d, "d_\bot"']& X_n \arrow[d, "d_\bot"] \\
X_n \arrow[r, "d_\top"']& X_{n-1}.
\end{tikzcd}
\end{center}
\end{defi}

\begin{propo}\cite[Proposition 3.7]{GTK1} \label{segaldecomp}
 Any Segal groupoid is a decomposition groupoid.
\end{propo}

\begin{exa}\label{exatrees}
The decomposition groupoid of rooted trees \textbf{RT} is defined as follows \cite{GTK1}. Recall that a \emph{forest} is a disjoint union of rooted trees.  An \emph{admissible cut} of a rooted tree is a splitting of the set of nodes into two subsets such that the second forms a subtree containing the root node or is the empty forest. $\textbf{RT}_1$ denotes the groupoid of isoclasses of forests, and $\textbf{RT}_2$ denotes the groupoid of isoclasses of forests with an admissible cut. More generally, $\textbf{RT}_0$ is defined to be a point, and $\textbf{RT}_k$ is the groupoid of isoclasses of forests with $k - 1$ compatible admissible cuts. These form a simplicial  groupoid in which the inner face maps forget a cut, and the outer face maps project away stuff: $d_\perp$ deletes the crown and $d_\top$ deletes the bottom layer. It is readily  seen that $\textbf{RT}$ is not a Segal groupoid: a tree with a cut cannot be reconstructed from its crown and its bottom tree, which is to say that $\textbf{RT}_2$ is not equivalent to $\textbf{RT}_1 \times_{\textbf{RT}_0} \textbf{RT}_1$. It is straightforward to check that it is a decomposition groupoid~\cite{GTK1}.
\end{exa}

\noindent Recall that a simplicial map $F \colon X \rightarrow Y$ is \emph{cartesian} on an arrow $[n] \rightarrow [k]$ in $\simplexcategory$, if the naturality square for $F$ with respect to this arrow is a homotopy pullback. A simplicial map $F \colon X \rightarrow Y$ is called a \emph{right fibration} if it is cartesian on all bottom coface maps $d_{\perp}$. Similarly, $F$ is called a \emph{left fibration} if it is cartesian on $d_{\top}$.

\begin{lem}  \label{propobesegalright}
Let $Y$ be a Segal groupoid and let $F \colon X \rightarrow Y$ be a simplicial map that is a left or a right fibration, then also $X$ is a Segal groupoid.
\end{lem}
 
Certain pullbacks in $\simplexcategory^{\op}$ are preserved by general decomposition groupoids, which is the content of the following result. 

\begin{lem}\cite[Lemma 3.10]{GTK1} \label{lemma310}
Let $X$ be a decomposition groupoid. For all $0<i<j<n$, the following squares of active face and degeneracy maps are homotopy pullbacks

\begin{center}
\begin{tikzcd}
X_{n+1} \drpullback \arrow[r, "d_{i}"]\arrow[d, "d_{j+1}"']& X_n \arrow[d, "d_j"] \\
X_n \arrow[r, "d_i"']& X_{n-1}
\end{tikzcd} \; \; \;
\begin{tikzcd}
X_{n-3} \drpullback \arrow[r, "s_{i-1}"]\arrow[d, "s_{j-2}"']& X_{n-2} \arrow[d, "s_{j-1}"] \\
X_{n-2} \arrow[r, "s_{i-1}"']& X_{n-1}.
\end{tikzcd}
\end{center}
\end{lem}

 A decomposition groupoid $X$ is \emph{complete} when $s_0 \colon X_0 \rightarrow X_1$ is a monomorphism (i.e.~is ($-1$)-truncated). It follows from the decomposition groupoid axiom that in this case all degeneracy maps are monomorphisms \cite[Lemma 2.5]{GTK2}.

\begin{blanko} 
{The incidence coalgebra of a decomposition groupoid} \label{subsec:incidencecoalgebra}
\end{blanko}

\noindent The span 
\begin{center}
\begin{tikzcd} [cramped, sep=small]
X_1 & X_2 \arrow[l, "d_1"'] \arrow[rr, "(d_2 \, d_0)"]& & X_1 \times_{X_0} X_1
\end{tikzcd} 
\end{center}
\noindent defines a linear functor, the comultiplication
$$\Delta \colon  \Grpd_{/X_1} \rightarrow \Grpd_{/X_1 \times_{X_0} X_1}$$
$$\; \; \; \; \;  \; \;  \; \;  f \mapsto (d_2, d_0)_! \circ  d_1^* (f).$$
 Likewise, the span
\begin{center}
\begin{tikzcd} [cramped, sep=small]
X_1 & X_0\arrow[l, "s_0"'] \arrow[r, "t"]& 1
\end{tikzcd}
\end{center}
 defines a linear functor, the counit
$$\delta \colon  \Grpd_{/X_1} \rightarrow \Grpd$$
$$\; \; \;\ \; \; \; \;  \; \;  \; \;  f  \mapsto t_! \circ  s_0^* (f).$$
The decomposition groupoid axioms serve to ensure that $\Delta$ is coassociative with counit $\delta$, up to coherent homotopy \cite[\S 5.3]{GTK1}. This coalgebra $(\Grpd_{/X_1} , \Delta, \delta)$  is called the \emph{incidence coalgebra}.  The classical notion of incidence coalgebras in vector spaces is obtained by taking homotopy cardinality; see \cite{GTK2}. A monoidal structure on $X$ gives furthermore a bialgebra structure. This is not needed in this paper, except that it will be mentioned in some examples.

\begin{exa}
The Butcher--Connes--Kreimer Hopf algebra of rooted trees is the free commutative algebra on the set of isomorphism classes of rooted trees, with comultiplication defined by summing over certain admissible cuts $c$:
 $$\Delta(T) = \sum_{c \in \mbox{\scriptsize{admi.cuts}}(T)} P_c \otimes R_c.$$
Recall that an admissible cut $c$ is a splitting of the set of nodes into two subsets, such that the second forms a subtree $R_c$ containing the root node (or is the empty forest); the first subset, the complement crown, then forms a subforest $P_c$. The Butcher--Connes--Kreimer Hopf algebra is in fact the incidence bialgebra of the decomposition groupoid of rooted trees of Example \ref{exatrees} \cite{GTK1}.  
\end{exa}

\begin{exa}\cite[\S 5.1]{GTK1}
If $X$ is the nerve of a category (for example, a poset) then $X_2$ is the set of all composable pairs of arrows. The comultiplication is then defined by:
$$\Delta(f) = \sum_{b \circ a = f} a \otimes b,$$
and the counit sends identity arrows to 1 and other arrows to 0.

\end{exa}
\begin{blanko}
{Culf maps} \label{subsec:culf}
\end{blanko}

\begin{defi}\cite[\S 4]{GTK1}
A simplicial map $F \colon X \rightarrow Y$ is called \emph{culf} if $F$ is cartesian on each active map.
\end{defi}

Culf stands for `conservative' and `unique lifting factorisations' where \emph{conservative} means cartesian on all codegeneracy maps, and \emph{unique lifting factorisations} means cartesian on all coface maps. The culf condition can be seen as  an abstraction of coalgebra homomorphism: the conservative condition corresponds to counit preservation, and ulf corresponds to comultiplicativity.

\begin{propo}\cite[Lemma 4.3]{GTK1} \label{culfcondition}
A simplicial map between decomposition groupoids is culf if and only if it is cartesian on $d^1 \colon [1] \rightarrow [2]$.
\end{propo}


\begin{blanko}
{Decalage}\label{decalage}\label{subsec:decalage}
\end{blanko}
\noindent Given a simplicial groupoid $X$, the \emph{lower dec} $\Dec_{\perp} X$ is a new simplicial groupoid obtained by deleting $X_0$ and shifting everything one place down, deleting also all $d_0$ face maps and all $s_0$ degeneracy maps. It comes equipped with a simplicial map, called the lower dec map, $d_\perp \colon \Dec_{\perp} X \rightarrow X$ given by the original $d_0$. 
 Similarly, the \emph{upper dec} $\Dec_{\top} X$ is obtained by instead deleting, in each degree, the top face map $d_\top$ and the top degeneracy map $s_\top$. The deleted top face maps becomes the upper dec map $d_\top \colon \Dec_{\top} X \rightarrow X$.  
\begin{propo}\cite[Proposition 4.9]{GTK1} \label{decareculfs}
If $X$ is a decomposition groupoid then the dec maps $d_\top \colon \Dec_\top X \rightarrow X$ and $d_\perp \colon \Dec_\perp X \rightarrow X$ are culf.
\end{propo}
 The decomposition property can be characterised in terms of decalage:
 
\begin{teo}\cite{DK, unital,GTK1}\label{theoremdecalage}
For a simplicial groupoid $X \colon \simplexcategory^{\op} \rightarrow \Grpd$, the following are equivalent
\begin{enumerate}
\item X is a decomposition groupoid
\item both $\Dec_{\perp} X$ and $\Dec_{\top} X$ are Segal groupoids.
\end{enumerate}
\end{teo}

Throughout we write $\delta A \colon \simplexcategory^{\op} \rightarrow \Grpd$ for the constant simplicial groupoid on a groupoid $A$. We have a natural transformation $\pil \colon \Dec_\top X \rightarrow \delta (X_0)$ defined as follows: the map $\pil \colon \Dec_\top X \rightarrow \delta (X_0)$ sends an $n$-simplex $\lambda$ in $\Dec_\top X$ to $d_\perp^{n+1}(\lambda)$ in $X_0$ and an arrow $\alpha \colon \lambda \rightarrow \eta$ in $(\Dec_\top X)_n$ to $d_\perp^{n+1}(\alpha)$ in $X_0$. 
We denote by ${\simplexcategory^{t}}$ the category whose objects are finite linear orders with a top element, and whose arrows are the maps that preserve the order and the top element. Since $[0]$ is terminal in $({\simplexcategory^{t}})^{\op}$, the map $\pil$ is a simplicial map. 

\begin{lem}\label{lemma: pi top is cartesian}
The natural transformation $\pil$ is cartesian on right fibrations. That is, given a right fibration $p \colon X \rightarrow Y$, the square
\[\begin{tikzcd}
	{\Dec_\top X} & {\delta (X_0)} \\
	{\Dec_\top Y} & {\delta (Y_0)}
	\arrow["{\pil}"', from=2-1, to=2-2]
	\arrow["{\Dec_\top p}"', from=1-1, to=2-1]
	\arrow["{\pil}", from=1-1, to=1-2]
	\arrow["{\delta (p_0)}", from=1-2, to=2-2]
	\arrow["{(1)}"{description}, draw=none, from=1-1, to=2-2]
\end{tikzcd}\]
is a homotopy pullback.
\end{lem}

\begin{proof}
The pullback property can be checked level-wise. Note that $(\Dec_\top X)_n = X_{n+1}$. In level  $n \geq 0$, the square $(1)$ is
\[\begin{tikzcd}
	{X_{n+1}} & {X_0} \\
	{Y_{n+1}} & {Y_0},
	\arrow["{d_\perp^{n+1}}", from=1-1, to=1-2]
	\arrow["{p_{n+1}}"', from=1-1, to=2-1]
	\arrow["{d_\perp^{n+1}}"', from=2-1, to=2-2]
	\arrow["{p_0}", from=1-2, to=2-2]
\end{tikzcd}\]
which is a homotopy pullback since $p$ is a right fibration.
\end{proof}

We also have a natural transformation $\pif \colon \Dec_\perp X \rightarrow \delta (X_0)$ defined as follows: the simplicial map $\pif \colon \Dec_\perp X \rightarrow \delta (X_0)$ sends an $n$-simplex $\lambda$ in $\Dec_\perp X$ to $d_\top^{n+1}(\lambda)$ in $X_0$ and an arrow $\alpha \colon \lambda \rightarrow \eta$ in $(\Dec_\perp X)_n$ to $d_\top^n(\alpha)$ in $X_0$. The proof of the following result is analogous to that of Lemma \ref{lemma: pi top is cartesian}.

\begin{lem}\label{lemma: pi lower is cartesian}
The natural transformation $\pif$ is cartesian on left fibrations. 
\end{lem}

\section{Slices and intervals}\label{subsec:slices}

In this section, we introduce some constructions with slice and coslice of decomposition groupoids required to introduce the concept of interval.

\begin{lem}\cite[Proposition 2.1]{unital}\label{unitallemafor2segalupper}
Let $X$ be a decomposition groupoid. For all $0 \leq i \leq n$ the following squares are homotopy pullbacks:
\begin{center}
\begin{tikzcd}
X_{n+1} \drpullback \arrow[r, "s_{i+1}"]\arrow[d, "d_{0}"']& X_{n+2} \arrow[d, "d_{0}"] \\
X_n \arrow[r, "s_i"']& X_{n+1}.
\end{tikzcd} \; \;
\begin{tikzcd}
X_{n+1} \drpullback \arrow[r, "s_{i}"]\arrow[d, "d_{n+1}"']& X_{n+2} \arrow[d, "d_{n+2}"] \\
X_n \arrow[r, "s_i"']& X_{n+1}.
\end{tikzcd}
\end{center} 
\end{lem}

The pullbacks of Lemma \ref{unitallemafor2segalupper} are called the upper and lower unital condition.
 

\begin{defi}\label{defislicegroupoid}
Let $X$ be a decomposition groupoid. For an object $y$ in $X_0$, the \emph{slice} $X_{/y}$ is defined as the homotopy pullback 
\[\begin{tikzcd}
	{X_{/y}} & 1 \\
	{\Dec_\top X} & {\delta (X_0)}.
	\arrow["{\pil}"', from=2-1, to=2-2]
	\arrow["u"', from=1-1, to=2-1]
	\arrow[from=1-1, to=1-2]
	\arrow["{\ulcorner y \urcorner}", from=1-2, to=2-2]
	\arrow[draw=none, from=1-1, to=2-2]
	\arrow["\lrcorner"{anchor=center, pos=0.125}, draw=none, from=1-1, to=2-2]
\end{tikzcd}\]
\end{defi}


\begin{rema}
Taking the upper decalage construction of $X$ gives a simplicial object starting in $X_1$, but equipped with an augmentation $d_0 \colon X_1 \rightarrow X_0$. Pulling back this simplicial object along $\ulcorner y \urcorner \colon 1 \rightarrow X_0$, yields a new simplicial object which is $X_{/y}$. The map $u$ is cartesian since $1 \rightarrow \delta (X_0)$ is cartesian and cartesian maps are stable under pullback. Therefore, $u$ is a right fibration, and as a consequence, $X_{/y}$ is Segal by \ref{propobesegalright}. 
\end{rema}

\begin{defi}\label{terminalobject}
Let $X$ be a Segal groupoid. An object $b \in X_0$ is called \emph{terminal} if the map $d_{\top} \circ u$ $X_{/b} \rightarrow X$ is a levelwise equivalence, where $d_\top \colon \Dec_{\top} X \to X$ is the upper dec map.
\end{defi}

\begin{propo}\label{existenceterminalwithdecalage}
Let $X$ be a decomposition groupoid. Then for an object $y$ in $X_0$, the object $s_0(y)$ is terminal in $X_{/y}$. 
\end{propo}

\begin{proof}
In the diagram
\[\begin{tikzcd}
	{(X_{/y})_{/s_0(y)}} & 1 \\
	{\Dec_\top X_{/y}} & {\delta ((X_{/y})_0)} \\
	{\Dec_\top \Dec_\top X} & {\delta ((\Dec_\top X)_0)}.
	\arrow[from=1-1, to=1-2]
	\arrow["{u'}"', from=1-1, to=2-1]
	\arrow["{\pil}"', from=2-1, to=2-2]
	\arrow["{\ulcorner s_0y \urcorner}", from=1-2, to=2-2]
	\arrow["{\Dec_{\top} u}"', from=2-1, to=3-1]
	\arrow["{\pil}"', from=3-1, to=3-2]
	\arrow["{\delta (u_0)}", from=2-2, to=3-2]
	\arrow["{(1)}"{description}, draw=none, from=1-1, to=2-2]
	\arrow["{(2)}"{description}, draw=none, from=2-1, to=3-2]
\end{tikzcd}\]
the square $(1)$ is a homotopy pullback by definition of ${(X_{/y})_{/s_0(y)}}$. Since $u \colon X_{/y} \rightarrow \Dec_{\top} X$ is a right fibration, we have that $(2)$ is a homotopy pullback by Lemma \ref{lemma: pi top is cartesian}. Therefore, the outer diagram is a homotopy pullback. Furthermore, note that $\delta (u_0)(s_0(y)) = s_0(y)$. This means that $(X_{/y})_{/{s_0(y)}}$ is the homotopy pullback of $\pil$ along $\ulcorner s_0y \urcorner \colon 1 \rightarrow \delta ((\Dec_\top X)_0)$. Note that in the diagram
\[\begin{tikzcd}
	{X_{/y}} & 1 \\
	{\Dec_\top X} & {\delta (X_0)} \\
	{\Dec_\top \Dec_\top X} & {\delta ((\Dec_\top X)_0)}
	\arrow[from=1-1, to=1-2]
	\arrow["u"', from=1-1, to=2-1]
	\arrow["{\pil}"', from=2-1, to=2-2]
	\arrow["{\ulcorner y \urcorner}", from=1-2, to=2-2]
	\arrow["{H}"', from=2-1, to=3-1]
	\arrow["{\pil}"', from=3-1, to=3-2]
	\arrow["{\delta (s_0)}", from=2-2, to=3-2]
	\arrow["{(3)}"{description}, draw=none, from=1-1, to=2-2]
	\arrow["{(4)}"{description}, draw=none, from=2-1, to=3-2]
\end{tikzcd}\]
the square $(3)$ is a homotopy pullback by definition of $X_{/y}$.  The map $H \colon \Dec_\top X \to \Dec_\top \Dec_\top X$ is defined by $H((\Dec_\top X)_n) = s_{n+1}(X_{n+1})$. The square $(4)$ is a pullback as a consequence of Lemma \ref{unitallemafor2segalupper} and the definition of $H$. 


Combining $(3)$ and $(4)$, the outer diagram is a homotopy pullback. Furthermore, note that $\delta (s_0)(y) = s_0(y)$. This means that $X_{/y}$ is the homotopy pullback of $\pil$ along $\ulcorner s_0y \urcorner \colon 1 \rightarrow \delta ((\Dec_\top X)_0)$. Since $X_{/y}$ and $(X_{/y})_{/s_0(y)}$ are homotopy pullbacks over the same diagram, we get a canonical identification $(X_{/y})_{/s_0(y)} \cong X_{/y}$.  Furthermore, this identification is given by the canonical projection map $\operatorname{pr} \colon (X_{/y})_{/s_0(y)} \to X_{/y}$ since $H \circ u \circ \operatorname{pr} = u' \circ \Dec_\top u$. 
\end{proof}

\begin{defi}\label{defilowercoslicegroupoid}
Let $X$ be a decomposition groupoid. For an object $x$ in $X_0$, the \emph{coslice} $X_{x/}$ is defined as the homotopy pullback
\[\begin{tikzcd}
	{X_{x/}} & 1 \\
	{\Dec_\perp X} & {\delta (X_0)}.
	\arrow["{\pif}"', from=2-1, to=2-2]
	\arrow["v"', from=1-1, to=2-1]
	\arrow[from=1-1, to=1-2]
	\arrow["{\ulcorner x \urcorner}", from=1-2, to=2-2]
	\arrow[draw=none, from=1-1, to=2-2]
	\arrow["\lrcorner"{anchor=center, pos=0.125}, draw=none, from=1-1, to=2-2]
\end{tikzcd}\]
\end{defi}

\noindent We write $v \colon X_{x/} \rightarrow \Dec_\perp X$ for the canonical map of Definition \ref{defilowercoslicegroupoid}. Note that for each $x$ in $X_0$, the coslice $X_{x/}$ is Segal. Indeed, $\Dec_\perp X$ is Segal and $X_{x/}$ is a left fibration over $\Dec_\bot X$, and is therefore Segal too by Lemma \ref{propobesegalright}.


\begin{defi}\label{initialobject}
Let $X$ be a Segal groupoid. An object $a \in X_0$ is called \emph{initial} if the map $d_\perp \circ v \colon X_{a/} \rightarrow X$ is a levelwise equivalence, where $d_\perp \colon \Dec_{\perp} X \to X$ is the lower dec map.
\end{defi}


\begin{propo}\label{existenceinitialwithdecalage}
Let $X$ be a decomposition groupoid. For an object $x$ in $X_0$, the object $s_0(x)$ is an initial object in $X_{x/}$.
\end{propo}

\begin{proof}
The proof is analogous to that of Proposition \ref{existenceterminalwithdecalage}.
\end{proof}

\begin{lem}\label{hereditaryinitialobject}
Let $\mathcal{C}$ be a Segal groupoid with an initial object $\perp_{\mathcal{C}}$. Then for each object $y$ in $\mathcal{C}$, the slice $\mathcal{C}_{/y}$ has an initial object.
\end{lem}

\begin{proof}
Since $\perp_{\mathcal{C}}$ is an initial object, we have a map $f_{\perp_\mathcal{C}} \colon \perp_{\mathcal{C}} \rightarrow y$. This map can be regarded as an object in $\mathcal{C}_{/y}$ or in $\mathcal{C}_{\bot /}$, and after two pullbacks of $\Dec_\top \Dec_\perp = \Dec_\perp \Dec_\top$ we get the natural identification $(\mathcal{C}_{/y})_{f_{\perp_{\mathcal{C}} /}} \cong (\mathcal{C}_{\perp_{\mathcal{C}}/})_{/ f_{\perp_{\mathcal{C}}}}$. Furthermore, in the diagram
\[\begin{tikzcd}
	{(\mathcal{C}_{\perp_{\mathcal{C}}/})_{/ f_{\perp_{\mathcal{C}}}}} & 1 \\
	{\Dec_\top \mathcal{C}_{\perp_{\mathcal{C}}/}} & {\delta ((\mathcal{C}_{\perp_{\mathcal{C}}/})_0)} \\
	{\Dec_\top \mathcal{C}} & {\delta (\mathcal{C}_0)}
	\arrow[from=1-1, to=1-2]
	\arrow["u"', from=1-1, to=2-1]
	\arrow["{\pil}"', from=2-1, to=2-2]
	\arrow["{\ulcorner f_{\perp_{\mathcal{C}}} \urcorner}", from=1-2, to=2-2]
	\arrow["{\Dec_{\top} d_\perp}"', from=2-1, to=3-1]
	\arrow["{\pil}"', from=3-1, to=3-2]
	\arrow["{\delta (d_\perp)}", from=2-2, to=3-2]
	\arrow["{(1)}"{description}, draw=none, from=1-1, to=2-2]
	\arrow["{(2)}"{description}, draw=none, from=2-1, to=3-2]
\end{tikzcd}\]
the square $(1)$ is a homotopy pullback by definition of ${(\mathcal{C}_{\perp_{\mathcal{C}}/})_{/ f_{\perp_{\mathcal{C}}}}}$. Since $\perp_{\mathcal{C}}$ is an initial object, we have that $d_\perp \colon \mathcal{C}_{\perp /} \rightarrow \mathcal{C}$ is a levelwise equivalence. This implies that $(2)$ is a homotopy pullback. Combining $(1)$ and $(2)$, we have that the outer diagram is a homotopy pullback. 
Furthermore, note that $d_\perp(f_{\perp_{\mathcal{C}}}) = y$. This means that ${(\mathcal{C}_{\perp_{\mathcal{C}}/})_{/ f_{\perp_{\mathcal{C}}}}}$ is the homotopy pullback of $\pil$ along $\ulcorner y \urcorner \colon 1 \rightarrow \delta (\mathcal{C}_0)$. But this is precisely the definition of $\mathcal{C}_{/y}$. This implies that ${(\mathcal{C}_{\perp_{\mathcal{C}}/})_{/ f_{\perp_{\mathcal{C}}}}} \cong \mathcal{C}_{/y}$ and therefore $(\mathcal{C}_{/y})_{f_{\perp_{\mathcal{C}} /}} \cong \mathcal{C}_{/y}$. Furthermore, this isomorphism is given by the canonical projection map $\operatorname{pr} \colon(\mathcal{C}_{/y})_{f_{\perp_{\mathcal{C}} /}}  \to \mathcal{C}_{/y}$ since $u'' \circ \operatorname{pr} = \Dec_{\top} d_{\perp} \circ u$, where $u'' \colon \mathcal{C}_{/y} \to \Dec_\top \mathcal{C}$ denotes the canonical map of Definition \ref{defislicegroupoid}.
\end{proof}

\begin{lem}\label{heredetaryterminalobject}
Let $\mathcal{C}$ be a Segal groupoid with a terminal object. Then for each object $x$ in $\mathcal{C}$, the coslice $\mathcal{C}_{x/}$ has a terminal object.
\end{lem}

\begin{proof}
The proof is analogous to that of Lemma \ref{hereditaryinitialobject}.
\end{proof}

Let $X$ be a decomposition groupoid. For $\lambda \colon \Delta^n \rightarrow X$, we denote by $\longt(\lambda)$ the $1$-simplex $\begin{tikzcd}[cramped, sep=small]
\Delta^1   \arrow[r, -act, ""]    &   \Delta^n  \arrow[r, "\lambda"]                      & X                   
\end{tikzcd}$.  Applying lower and upper decalage to $X$, we obtain a new decomposition groupoid $\Dec_{\top} \Dec_{\perp} X$ and a map $\epsilon \colon \Dec_{\top} \Dec_{\perp} X \rightarrow X$ which is culf by Proposition \ref{decareculfs}. Furthermore, we have a natural transformation $\pilong \colon \Dec_\top \Dec_\perp X \rightarrow \delta (X_1)$ defined as follows: the map $\pilong \colon \Dec_\top  \Dec_\perp X \rightarrow \delta (X_1)$ sends an $n$-simplex $\lambda$ in $\Dec_\top  \Dec_\perp X$ to $\longt(\lambda)$ in $X_1$ and an arrow $\alpha \colon \lambda \rightarrow \eta$ in $(\Dec_\top \Dec_\perp X)_n$ to $\longt(\alpha)$ in $X_1$.  The category $\simplexcategory^{\act}$ is the subcategory of $\simplexcategory$ whose objects are the nonempty finite ordinals and whose morphisms are the active maps. Since $[1]$ is terminal in $({\simplexcategory^{\act}})^{\op}$, the map  $\pilong$ is a simplicial map.

\begin{lem}\label{lemma: pi longt is cartesian}
The natural transformation $\pilong$ is cartesian on culf maps. That is, given a culf map  $F \colon X \rightarrow Y$ between decomposition groupoids, the square
\[\begin{tikzcd}
	{\Dec_\top \Dec_\perp X} & {\delta (X_1)} \\
	{\Dec_\top \Dec_\perp Y} & {\delta (Y_1)}
	\arrow["{\pilong}"', from=2-1, to=2-2]
	\arrow["{\delta (F_1)}", from=1-2, to=2-2]
	\arrow["{\Dec_\top \Dec_\perp F}"', from=1-1, to=2-1]
	\arrow["{\pilong}", from=1-1, to=1-2]
\end{tikzcd}\]
is a homotopy pullback.
\end{lem}

\begin{proof}
The pullback property can be checked levelwise. Note that $(\Dec_\top \Dec_\perp X)_n = X_{n+2}$. For $n \geq 0$, the square 
\[\begin{tikzcd}
	{X_{n+2}} & {X_1} \\
	{Y_{n+2}} & {Y_1}
	\arrow["{d_1^{n+1}}", from=1-1, to=1-2]
	\arrow["{F_{n+2}}"', from=1-1, to=2-1]
	\arrow["{d_1^{n+1}}"', from=2-1, to=2-2]
	\arrow["{F_1}", from=1-2, to=2-2]
\end{tikzcd}\]
is a homotopy pullback since $F$ is culf.
\end{proof}

\begin{defi}\label{definition: thenicalinterval}
Let $X$ be a decomposition groupoid and let $f$ be an object in $X_1$. The Segal groupoid $I_f$ is defined as the homotopy pullback, called the \emph{interval} of $f$,
\[\begin{tikzcd}
	{I_f} & 1 \\
	{\Dec_\top \Dec_\perp X} & {\delta (X_1)}.
	\arrow["{\pilong}"', from=2-1, to=2-2]
	\arrow["w"', from=1-1, to=2-1]
	\arrow[from=1-1, to=1-2]
	\arrow["{\ulcorner f \urcorner}", from=1-2, to=2-2]
	\arrow[draw=none, from=1-1, to=2-2]
	\arrow["\lrcorner"{anchor=center, pos=0.125}, draw=none, from=1-1, to=2-2]
\end{tikzcd}\]
\end{defi}
We write $w \colon I_f \rightarrow \Dec_{\top} \Dec_\bot X$ for the simplicial  map obtained in this way. From its construction as a pullback of a map between constant simplicial groupoids, it is clear that $w$ is culf. The double decalage construction induces a culf map $\mfunctor_f \colon I_f  \rightarrow X$, defined by the composition of  $w$ and the canonical map $\epsilon \colon \Dec_\top \Dec_\perp X \rightarrow X$.

\begin{rema}\label{defintervaldspace}
When $X$ is the ordinary nerve of a category, the description of $I_f$ is due to Lawvere~\cite{Law}: the objects of $I_f$ are two-step factorisations of $f$. The $1$-cells are arrows between such factorisations, or equivalently $3$-step factorisations, and so on. More generally, let $X$ be a decomposition set and $f \in X_1$. The Segal set $I_f$ is described as follows:
\begin{enumerate}
\item An object of $I_f$ is any $\sigma \in X_2$ such that $d_1(\sigma) = f$.
\item Given two objects $\sigma$ and ${\sigma'}$ in $I_f$, a morphism $\gamma \colon \sigma \rightarrow  {\sigma'}$ in $I_f$  is any object $\gamma \in X_3$, such that $d_2(\gamma) = \sigma$ and $d_1(\gamma) = {\sigma'}$. 
\item Given two morphisms $\gamma \colon \sigma \rightarrow \sigma'$ and ${\gamma'} \colon \sigma' \rightarrow {\sigma''}$ of $I_f$, the composition is defined by ${\gamma'} \circ \gamma \colon= d_2(\eta)$, where $\eta \in X_4$ satisfies that $d_1(\eta) = {\gamma'}$ and $d_3(\eta) = \gamma$.  
 The unique existence of  $\eta$ is a consequence of the decomposition-groupoid axioms in the form of Lemma \ref{lemma310}. Associativity also follows by Lemma \ref{lemma310}. 
\end{enumerate}
\end{rema}

Applying lower and upper decalage to $X$  generate two sections on $\Dec_\top \Dec_\perp X$. The first one is induced by $s_\perp \colon X \rightarrow \Dec_\perp X$ and the other is induced by $s_\top \colon X \rightarrow \Dec_\top X$.  We shall see later that $s_\perp(f)$ is an initial object and $s_\top(f)$ is a terminal object in $I_f$. 
Recall that we write $u \colon X_{/y} \rightarrow \Dec_{\top} X$ for the canonical map of Definition \ref{defislicegroupoid} and $v \colon X_{x/} \rightarrow \Dec_\perp X$ for the canonical map of Definition \ref{defilowercoslicegroupoid}. When further (co)slicing is used we decorate the $u$ or $v$ with a prime.

\begin{lem}\label{presentationlaw}
Let $X$ be a decomposition groupoid. For $f \in X_1$, put $x = d_1(f)$ and $d_0(f)= y$. There are canonical equivalences $(X_{x/})_{/f} \rightarrow I_f$ and $(X_{/y})_{f/} \rightarrow I_f$ such that the following diagram commutes up to isomorphism
\begin{center}
\begin{tikzcd}
(X_{x/})_{/f} \arrow[r, "\simeq"] \arrow[d, "d_\top \circ u"'] \arrow[white]{dr}[black, description]{(1)}  & I_f \arrow[white]{dr}[black, description]{(2)}  \arrow[d, "\mfunctor_f" description] & (X_{/y})_{f/} \arrow[l, "\simeq"'] \arrow[d, "d_\perp \circ v'"] \\
X_{x/} \arrow[r, "d_\perp \circ v"']                       & X                   & X_{/y}. \arrow[l, "d_\top \circ u'"]                       
\end{tikzcd}
\end{center}
\end{lem}

\begin{proof}
In the diagram
\[\begin{tikzcd}
	{(X_{x/})_{/f}} & 1 \\
	{\Dec_\top {X_{x/}}} & {\delta ((X_{x/})_0)} \\
	{\Dec_\top \Dec_\perp X} & {\delta (X_1)} \\
	{} & {}
	\arrow["{\ulcorner f\urcorner}", from=1-2, to=2-2]
	\arrow[from=1-1, to=1-2]
	\arrow["u"', from=1-1, to=2-1]
	\arrow["{\pil}"', from=2-1, to=2-2]
	\arrow["{\Dec_\top v}"', from=2-1, to=3-1]
	\arrow["{\pilong}"', from=3-1, to=3-2]
	\arrow["{\delta (v_0)}", from=2-2, to=3-2]
	\arrow["{(3)}"{description}, draw=none, from=1-1, to=2-2]
	\arrow["{(4)}"{description}, draw=none, from=2-1, to=3-2]
\end{tikzcd}\]
the square $(3)$ is a homotopy pullback by construction of $(X_{x/})_{/f}$ . Since $v$ is a right fibration the square $(4)$ is a homotopy pullback by Lemma \ref{lemma: pi lower is cartesian}. Therefore, the outer diagram is a homotopy pullback. Note that $\delta (v_0)(f) = y$. This implies that $(X_{x/})_{/f}$ is the homotopy pullback of $\pilong$ along $\ulcorner f \urcorner \colon 1 \rightarrow \delta (X_1)$. But this is precisely the definition of $I_f$. This gives us an equivalence $G \colon (X_{x/})_{/f} \rightarrow I_f$ such that $w \circ G \simeq \Dec_\top v \circ u$, which is the upper square in the diagram 
\[\begin{tikzcd}
	{(X_{x/})_{/f}} && {I_f} \\
	{\Dec_\top X_{x/}} & {\Dec_\top \Dec_\perp X} \\
	{X_{x/}} & {\Dec_\perp X} & X
	\arrow["{d_\top}"', from=2-1, to=3-1]
	\arrow["{\mfunctor_f}", from=1-3, to=3-3]
	\arrow["G", from=1-1, to=1-3]
	\arrow["u"', from=1-1, to=2-1]
	\arrow["\Dec_\top v", from=2-1, to=2-2]
	\arrow["\epsilon"{description}, from=2-2, to=3-3]
	\arrow["w"{description}, from=1-3, to=2-2]
	\arrow["{d_\top}"', from=2-2, to=3-2]
	\arrow["{d_\perp}"', from=3-2, to=3-3]
	\arrow["v"', from=3-1, to=3-2]
\end{tikzcd}\]
Since the other regions in the diagram commute strictly (by functoriality of upper decalage and by definition of $\epsilon$ and $\mfunctor_f$), we get a natural isomorphism for the outer square, which is precisely $(1)$. By analogous arguments, $(2)$ commutes up to isomorphism.
\end{proof}
When $X$ is the ordinary nerve of a category, Lemma \ref{presentationlaw}  is the same as Lemma 3.2 in~\cite{Law}. 

\begin{lem}\label{intervalisogrp}
Let $X$ be a Segal groupoid with an initial object $\bot$ and a terminal object $\top$. Let $h \colon \bot \rightarrow \top$ be a map from $\bot$ to $\top$, then $X \simeq I_{h}$.
\end{lem} 
\begin{proof}
Applying Lemma \ref{presentationlaw} to $h$, we have that $(X_{\perp/})_{/h} \simeq I_h$. Applying Lemma \ref{hereditaryinitialobject} to $\top$, it follows that $X_{/\top} \simeq (X_{\perp/})_{/h}$. Furthermore, $X_{/ \top} \simeq X$ since $\top$ is a terminal object. Combining these equivalences, we get that $X \simeq I_{h}$. 
\end{proof}

\noindent When $X$ is the ordinary nerve of a category, Lemma \ref{intervalisogrp} is the same as Lemma 3.3 in~\cite{Law}. 

\begin{propo}\label{completeintervalcategories}
Let $X$ be a complete decomposition groupoid. Then for each $f \in X_1$, the Segal groupoid $I_f$ is complete in the sense of decomposition groupoids, meaning that $s_0 \colon (I_f)_0 \rightarrow (I_f)_1$ is a monomorphism.
\end{propo}

\begin{proof}
By construction of $I_f$, we have the following diagram
\begin{center}
\begin{tikzcd}
(I_f)_0 \arrow[r, "s_0"] \arrow[d, "w_0"'] & (I_f)_1 \arrow[r, "d_0"] \arrow[d, "w_1"'] \drpullback & (I_f)_0 \arrow[d, "w_0"] \\
X_1 \arrow[r, "s_1"']                      & X_2 \arrow[r, "d_1"']                      & X_1.                     
\end{tikzcd}
\end{center}
Since $X$ is complete, the map $s_1 \colon X_1 \rightarrow X_2$ is a monomorphism, and therefore also its pullback $s_0 \colon (I_f)_0 \rightarrow (I_f)_1$ is a monomorphism,  which is to say that $I_f$ is complete.
\end{proof}

\section{Algebraic intervals and rigid decomposition groupoids} \label{sec: rigid ds}

To study the first case of the Gálvez--Kock--Tonks conjecture, the obvious level of generality would be discrete decomposition groupoids, but the proofs to be presented in this section work for locally discrete decomposition groupoids of the kind featured in the following definition: 

\begin{defi}\label{defi: discrete intervals dc}
A \emph{rigid decomposition groupoid} is a strict simplicial groupoid $X$ such that $d_1 \colon X_2 \rightarrow X_1$ is a discrete fibration, $s_0 \colon X_0 \rightarrow X_1$ is a monomorphism and the active-inert squares are strict pullbacks.
\end{defi}

The point, as we shall see, is for a rigid decomposition groupoid $X$, we have that for all $f \in X_1$, the Segal groupoid $I_f$ (\ref{definition: thenicalinterval}) is discrete. Note that every discrete decomposition groupoid is a rigid decomposition groupoid. This means that the rigid decomposition groupoids already cover locally finite posets, Cartier--Foata monoids and M\"{o}bius categories. The importance of locally discrete is to cover also strict (directed) restriction species as shown in the following example:

\begin{exa}\label{exa: DRS}
Let $\mathbb{I}$ be the category of finites sets and injections. Let $\mathbb{C}$ be the category of finite posets and convex monotone injections. A restriction species \cite{schmitt_1993} is a functor $R' \colon \mathbb{I}^{\op} \rightarrow \Set$ and a directed restriction species \cite{GKT:restr} is a functor $R \colon \mathbb{C}^{\op} \rightarrow \Set$. Note that any restriction species is a directed restriction species. An element of $R[P]$, where $P \in \mathbb{C}$, is called an $R$-structure on $P$. An $R$-structure on a poset $P$ thus restricts to any convex subposet $Q \subset P$. A directed restriction species $R$ induces a decomposition groupoid $\textbf{R}$ \cite{GKT:restr}, where $\textbf{R}_n$ is the groupoid of $R$ structures with an $n$-layering of the underlying poset $P$, that is a monotone map $P \rightarrow \underline{n}$, the linear order with $n$ elements. The map $d_1 \colon \textbf{R}_2 \rightarrow \textbf{R}_1$ forgets the layering and is clearly a discrete fibration. Altogether, the decomposition groupoid $\textbf{R}$ is rigid.
\end{exa}

\begin{exa}\label{intervalconstree}
Recall that for the decomposition groupoid of rooted trees \textbf{RT} of Example \ref{exatrees}, $\textbf{RT}_2$ is the groupoid of forests with an admissible cut, $\textbf{RT}_1$ is the groupoid of forests and the map $d_1 \colon \textbf{RT}_2 \rightarrow \textbf{RT}_1$ forgets the admissible cut. It is straightforward to see that $d_1$ is a discrete fibration. Therefore, \textbf{RT} is a rigid decomposition groupoid.

D\"{u}r~\cite{Dur} gave an incidence-coalgebra construction of the Butcher--Connes--Kreimer coalgebra by starting with the category of forests and root-preserving inclusions, generating a coalgebra and imposing the equivalence relation that identifies two root-preserving forest inclusions if their complement crowns are isomorphic forests. 

Consider the rigid decomposition groupoid of rooted trees $\textbf{RT}$. We can consider a tree $T$ as an object in $\textbf{RT}_1$. The interval $I_T$ can be described as follows: $(I_T)_0$ is the set of all isoclasses of admissible cuts of $T$, and $(I_T)_k$ is the set of isoclasses of all $k+1$ compatible admissible cuts of $T$.

We can relate the construction of D\"{u}r  with the interval construction of a rooted tree as follows: note that admissible cuts are essentially the same thing as root-preserving forest inclusions: then the cut is interpreted as the division between the included forest and the forest induced on the nodes in its complement. In this way we see that $(I_T)_k$ is the discrete groupoid of $k+1$ consecutive root-preserving inclusions ending in $T$.  
\end{exa}

\begin{rema}\label{tecnicalargument}

In a decomposition groupoid $X$, every active face map is a pullback of $d_1 \colon X_2 \rightarrow X_1$ \cite[Lemma 1.10]{GTK2}. Therefore, in the case where $X$ is rigid we have that $\pilong$ is a levelwise discrete fibration since in each level it is the long-edge map, which is a composition of active face maps, and these are all discrete fibrations. Therefore, the strict pullback of $\pilong$ along $\ulcorner f \urcorner \colon 1 \rightarrow \delta (X_1)$ is also a homotopy pullback. Furthermore, for every $f \in X_1$, the Segal groupoid $I_f$ is discrete.
\end{rema}

\begin{defi}\label{defi: choseterminalobject and initial}
Let $X$ be a discrete Segal groupoid.  A \emph{chosen terminal} object  is the choice of a terminal object $b$. A \emph{chosen initial} is the choice of an initial object $a$. 
\end{defi}

\begin{exa}\label{existencechosenterminalwithdecalage}
Let $X$ be a discrete decomposition groupoid. Let $y$ be an object in $X_0$. We already know from Proposition \ref{existenceterminalwithdecalage} that $s_0(y)$ is a terminal object in $X_{/y}$,  which we take as the chosen one. In this way $X_{/y}$ acquires a canonical chosen terminal. Similarly, $s_0(y)$ is an initial object in $X_{y/}$ (Proposition \ref{existenceinitialwithdecalage}), which we take as the chosen one. In this way $X_{x/}$ acquires a canonical chosen initial.
\end{exa}

\begin{defi}\label{defi: interval}
A \emph{discrete algebraic interval} is a discrete Segal groupoid $\mathcal{C}$ with a chosen initial object $\bot$ and a chosen terminal object $\top$. We denote the map from the chosen initial to the chosen terminal by $\ledge \colon \perp \rightarrow \top$.
\end{defi}

\begin{rema}
In the case of the nondiscrete algebraic intervals, further structure is required in the notion of chosen terminal, namely the choice of a section $s \colon X \rightarrow X_{/b}$ for the canonical map $d_\top u \colon X_{/b} \rightarrow X$. This will not be needed in the present paper, but the interested reader can find the theory of these worked out in Version 1 of this paper on arXiv.
\end{rema}

\begin{rema}
Batanin and Markl~\cite{BM} used the notion of a category
with chosen local terminal objects, meaning a category which in
each connected component is provided with a chosen terminal object.
This notion plays an important role in their theory of operadic
categories. Garner, Kock and Weber~\cite{KW} observed that the
structure of chosen local terminal objects is precisely to be a
coalgebra for the upper-Dec comonad. This in turn amounts to having
an extra top degeneracy map for the nerve of the category. When we
insist on having a chosen terminal object, it is inspired by this
decalage viewpoint on chosen terminals.  
Similarly of course, the notion of chosen local initial object amounts to
coalgebra structure for the lower-Dec comonad, via extra bottom
degeneracy maps, as the chosen initial object in our definition.
Finally, the main point here is the combination of the two
ideas. A discrete algebraic interval structure is in particular a coalgebra for the
two-sided-Dec comonad. This is very much in line with the notion of
flanking of G\'alvez--Kock--Tonks~\cite[\S 1]{GTK3}. 
\end{rema}

Definition \ref{definition: thenicalinterval} can be rewritten in terms of rigid decomposition groupoid as follows:

\begin{defi}\label{definition: strict interval}
Let $X$ be a rigid decomposition groupoid and let $f$ be an object in $X_1$. The Segal groupoid $I_f$ is defined as the strict pullback
\[\begin{tikzcd}
	{I_f} & 1 \\
	{\Dec_\top \Dec_\perp X} & {\delta (X_1)}.
	\arrow["{\pilong}"', from=2-1, to=2-2]
	\arrow["w"', from=1-1, to=2-1]
	\arrow[from=1-1, to=1-2]
	\arrow["{\ulcorner f \urcorner}", from=1-2, to=2-2]
	\arrow[draw=none, from=1-1, to=2-2]
	\arrow["\lrcorner"{anchor=center, pos=0.125}, draw=none, from=1-1, to=2-2]
\end{tikzcd}\]
In fact this strict pullback is also a homotopy pullback since $\pilong$ is a discrete fibration as a consequence of the fact that $d_1$ is a discrete fibration.
\end{defi}

We write $w \colon I_f \rightarrow \Dec_{\top} \Dec_\bot X$ for the simplicial  map obtained in this way. The double decalage construction induces a culf map $\mfunctor_f \colon I_f  \rightarrow X$, defined by the composition of $w$ and the canonical map $\epsilon \colon \Dec_\top \Dec_\perp X \rightarrow X$.

\begin{lem}\label{existenceterminalIF}\label{existenceinitialIF}
Let $X$ be a rigid decomposition groupoid and $f \in X_1$. The Segal groupoid $I_f$ has a canonical structure of an algebraic interval, where the chosen initial object is $s_0(f)$ and the chosen terminal object is $s_1(f)$.
\end{lem}

\begin{proof}
The object $s_1(f)$ is a terminal object in $I_f$ as a consequence of Lemma \ref{existenceterminalwithdecalage}, which we take as the chosen one. On the other hand, the object $s_0(f)$ is a initial object in $I_f$ as a consequence of Lemma \ref{existenceinitialwithdecalage}, which we take as the chosen initial. 
\end{proof} 

A simplicial map $F \colon X \rightarrow Y$ between rigid decomposition groupoids is called \emph{strict culf} if the naturality square for $F$ with respect any active map $[n] \rightarrowtail [k]$ in $\simplexcategory$ is a strict pullback. Since the active maps are fibrations in rigid decomposition groupoids, it follows that the strict pullbacks of a strict culf map are also homotopy pullbacks, so that strict culf map is in fact culf in the usual homotopy invariant sense. Lemma \ref{lemma: pi longt is cartesian} can be rewritten in terms of the strict condition as follows:

\begin{lem}\label{lemma: pi longt is strict cartesian}
The natural transformation $\pilong$ is cartesian on strict culf maps. That is, given a strict culf map  $F \colon X \rightarrow Y$ between rigid decomposition groupoids, the square
\[\begin{tikzcd}
	{\Dec_\top \Dec_\perp X} & {\delta (X_1)} \\
	{\Dec_\top \Dec_\perp Y} & {\delta (Y_1)}
	\arrow["{\pilong}"', from=2-1, to=2-2]
	\arrow["{\delta (F_1)}", from=1-2, to=2-2]
	\arrow["{\Dec_\top \Dec_\perp F}"', from=1-1, to=2-1]
	\arrow["{\pilong}", from=1-1, to=1-2]
\end{tikzcd}\]
is a strict pullback.
\end{lem}

\begin{proof}
The proof is analogous to that of Lemma \ref{lemma: pi longt is cartesian}. Furthermore, the strict pullback is also a homotopy pullback since $\pilong$ is a discrete fibration.
\end{proof}

The culf maps preserve the algebraic structure of a decomposition groupoid, but do not necessarily preserve the chosen initial and chosen terminal objects for a discrete algebraic interval. The maps that preserve this structure is the content of the following definition:  

\begin{defi}
A simplicial map  between discrete algebraic intervals  is termed \emph{stretched}, and written $\mathcal{C} \actto \mathcal{D}$, if it preserves the chosen initial object $\bot_\mathcal{C}$ and the chosen terminal object $\top_\mathcal{C}$.
\end{defi}

\begin{lem}\label{missing 3existencecell}
Let $X$ be a rigid decomposition groupoid and let $f$ be a $1$-simplex in $X$. The unique stretched map $\ledge \colon \Delta^1 \actto I_f$ is compatible with $\mfunctor_f$, 
  meaning that we have a commutative triangle
  \[
  \begin{tikzcd}
  \Delta^1 \ar[r, -act, "\ledge"] \ar[rd, "f"'] & I_f \ar[d, "\mfunctor_f"]  \\
   & X.
  \end{tikzcd}
  \]
\end{lem}

\begin{proof}
Put $x := d_\top(f)$ and $y := d_\perp(f)$ (the domain and codomain of $f$). Recall that the objects of $I_f$ are $2$-simplices with long edge
$f$.  The arrows in $I_f$ are $3$-simplices with long edge $f$.  We
know that the (chosen) initial object is $s_\bot(f)$ (which
can be thought of as the triangle with short sides $\identity_x$ and
$f$) and the (chosen) terminal object is $s_\top(f)$ (which
can be thought of as the triangle with short sides $f$ and
$\identity_y$).  The unique arrow $\ledge$ from the initial to the
terminal is the tetrahedron $s_\bot s_\top (f)$ (which we
can think of as the tetrahedron with short sides $\identity_x$, $f$,
and $\identity_y$). By definition $\mfunctor_f = d_\top d_\perp w$. Since $I_f$ is a discrete algebraic interval the map $w  \colon I_f \to \Dec_\top \Dec_\perp X$ is level-wise injective on objects. So what $\mfunctor_f$ does is that it applies $d_\top d_\bot$. In conclusion we have 
$\mfunctor_f(\ledge) =  d_\top d_\bot s_\top s_\bot(f) = f$, which is what we wanted to prove.
\end{proof}
 
\begin{exa}\label{example: Mstretched}
In the situation of Lemma \ref{missing 3existencecell}, if $X$ is already an interval and $f$ is its long edge, we see from the argument in the proof that $\mfunctor_f$ is stretched in this case we will see in \ref{lemma: mfunctor is invertible} that $\mfunctor_f$ is actually invertible in this case.
\end{exa}

\begin{rema}\label{rema: delta is stretched}
$\Delta^n$ is an algebraic interval for each $n$. The stretched maps $\Delta^m \actto \Delta^n$ are precisely the active maps. Every algebraic interval $A$ receives a unique stretched map from $\Delta^1$. A simplicial map between algebraic intervals $A \to B$ is stretched if and only if it commutes with the stretched maps from $\Delta^1$.
\end{rema}

Let $\mathcal{C}$ be a discrete algebraic interval. Let $\mathcal{C}_n$ denote the set of $n$-simplices in $\mathcal{C}$ and let $\mathcal{C}^{\strs}_n$ be the subset of stretched $(n+2)$-simplices in $\mathcal{C}$.

\begin{lem}\label{existssimplecesinterval}
The map $d_\top d_\perp \colon \mathcal{C}^{\strs}_{n+2} \rightarrow \mathcal{C}_n$ is a bijection.
\end{lem}

\begin{proof}
We will construct an inverse $t \colon \mathcal{C}_n \rightarrow \mathcal{C}^{\strs}_{n+2}$ of $d_\top d_\perp$ as follows. Let $\lambda$ be an object in $\mathcal{C}_n$, put $a = d_\top(\longt(\lambda))$ and $b = d_\bot(\longt(\lambda))$. Since $\mathcal{C}$ is a discrete algebraic interval, we have a chosen edge $f_\perp \colon \Delta^1 \rightarrow \mathcal{C}$ such that $d_\top (f_\perp) = \perp_\mathcal{C}$ and $d_\perp(f_\perp) = a$. By the same argument, we have a chosen edge $f_\top \colon \Delta^1 \rightarrow \mathcal{C}$ such that $d_\perp (f_\top) = \top_\mathcal{C}$ and $d_\top(f_\top) = b$.  In the diagram 
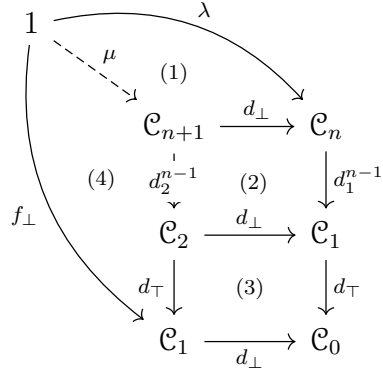
\begin{center}
\begin{tikzcd}
1 \arrow[rrd, "\lambda", bend left] \arrow[white]{rrd}[black, description]{(1)}     \arrow[rddd, "f_\perp"', bend right] \arrow[white]{rddd}[black, description]{(4)}  \arrow[rd, "\mu", dashed] &                                                      &                           \\
                                                                                                & \mathcal{C}_{n+1} \arrow[d, "d_2^{n-1}"' description] \arrow[r, "d_\perp"]  \arrow[white]{dr}[black, description]{(2)}     & \mathcal{C}_{n} \arrow[d, "d_1^{n-1}"] \\
                                                                                                & \mathcal{C}_2 \arrow[r, "d_\perp"] \arrow[d, "d_\top"']     \arrow[white]{dr}[black, description]{(3)}         & \mathcal{C}_1 \arrow[d, "d_\top"]   \\
                                                                                                & \mathcal{C}_1 \arrow[r, "d_\perp"']                            & \mathcal{C}_0                      
\end{tikzcd}
\end{center}
the squares $(2)$ and $(3)$ are strict pullbacks since $\mathcal{C}$ is a discrete Segal groupoid. Therefore, the outer rectangle is a strict pullback. By the pullback property of $\mathcal{C}_{n+1}$, there exists a unique map $\mu \colon 1 \rightarrow \mathcal{C}_{n+1}$ such that the diagram commutes. Since the square $(2)$ commutes and $d_\bot \mu = \lambda$, we have that $d_\bot \longt (\mu) = b$. Furthermore, in the diagram
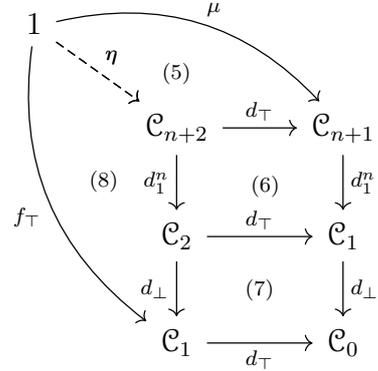
\begin{center}
\begin{tikzcd}
1 \arrow[rrd, "\mu", bend left] \arrow[rddd, "f_\top"', bend right] \arrow[white]{rddd}[black, description]{(8)}  \arrow[rd, "\eta", dashed]  \arrow[white]{rrd}[black, description]{(5)}  \arrow[rd, "\eta", dashed] &                                                     &                          \\
                                                                                                  & \mathcal{C}_{n+2}  \arrow[white]{dr}[black, description]{(6)}      \arrow[d, "d_1^{n}"'] \arrow[r, "d_\top"] & \mathcal{C}_{n+1} \arrow[d, "d_1^{n}"]    \\
                                                                                                  & \mathcal{C}_2  \arrow[white]{dr}[black, description]{(7)}     \arrow[r, "d_\top"] \arrow[d, "d_\perp"']       & \mathcal{C}_1 \arrow[d, "d_\perp"] \\
                                                                                                  & \mathcal{C}_1 \arrow[r, "d_\top"']                            & \mathcal{C}_0                     
\end{tikzcd}
\end{center}
the squares $(6)$ and $(7)$ are strict pullbacks since $\mathcal{C}$ is a discrete Segal groupoid. Therefore, the outer rectangle is a strict pullback. By the pullback property of $\mathcal{C}_{n+2}$, there exists a unique map $\ulcorner \eta \urcorner \colon 1 \rightarrow \mathcal{C}_{n+2}$ such that the diagram commutes. Note that $d_\top (\longt (\eta)) = \perp_\mathcal{C}$ since $(4)$ commutes and $d_\top (f_\perp) = \perp_\mathcal{C}$.
Since $(8)$ commutes and $d_\perp (f_\top) = \top_\mathcal{C}$, we have that $d_\perp (\longt (\eta)) = \top_\mathcal{C}$. This together with $d_\top (\longt (\eta)) = \perp_\mathcal{C}$ implies that $\longt(\eta) = \ledge_\mathcal{C}$ since $\mathcal{C}$ is a discrete algebraic interval. We define
$$t(\lambda) := \eta.$$ Combining $(1)$ and $(5)$, we have that $d_\bot d_\top \eta = \lambda$. This means that $d_\top d_\perp(t(\lambda)) = \lambda$. Now we will check that $t \circ d_\top d_\perp = \identity_{\mathcal{C}^{\strs}_{n+2}}$. Let $\psi$ be an object in $\mathcal{C}^{\strs}_{n+2}$.
Since $\mathcal{C}$ is a discrete algebraic interval, we have a chosen edge $f'_\perp \colon \Delta^1 \rightarrow \mathcal{C}$ such that $d_\top (f'_\perp) = \perp_\mathcal{C}$ and $d_\perp(f'_\perp) = d_\bot(\longt(d_\top d_\perp \psi))$. By the same argument, we have a chosen edge $f'_\top \colon \Delta^1 \rightarrow \mathcal{C}$ such that $d_\perp (f'_\top) = \top_\mathcal{C}$ and $d_\top(f'_\top) = d_\bot(\longt(d_\top d_\perp \psi))$.
 The commutative diagrams 
  
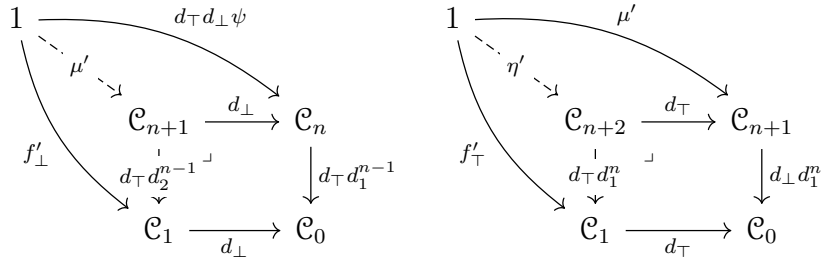
\[\begin{tikzcd}
	1 \\
	& {\mathcal{C}_{n+1}} & {\mathcal{C}_n} \\
	& {\mathcal{C}_1} & {\mathcal{C}_0}
	\arrow["{d_\top d_1^{n-1}}", from=2-3, to=3-3]
	\arrow["{d_\perp}"', from=3-2, to=3-3]
	\arrow["{d_\perp}", from=2-2, to=2-3]
	\arrow["{d_\top d_2^{n-1}}"' description, from=2-2, to=3-2]
	\arrow["{f'_\perp}"', bend right = 20, from=1-1, to=3-2]
	\arrow["{d_\top d_\perp \psi}", bend left =20, from=1-1, to=2-3]
	\arrow["{\mu'}"{description}, dashed, from=1-1, to=2-2]
	\arrow["\lrcorner"{anchor=center, pos=0.125}, draw=none, from=2-2, to=3-3]
\end{tikzcd} \; \; \; \begin{tikzcd}
	1 \\
	& {\mathcal{C}_{n+2}} & {\mathcal{C}_{n+1}} \\
	& {\mathcal{C}_1} & {\mathcal{C}_0}
	\arrow["{d_\perp d_1^{n}}", from=2-3, to=3-3]
	\arrow["{d_\top}"', from=3-2, to=3-3]
	\arrow["{d_\top}", from=2-2, to=2-3]
	\arrow["{d_\top d_1^{n}}"{description}, from=2-2, to=3-2]
	\arrow["{f'_\top}"', bend right =20, from=1-1, to=3-2]
	\arrow["{\mu'}", bend left = 20, from=1-1, to=2-3]
	\arrow["{\eta'}"{description}, dashed, from=1-1, to=2-2]
	\arrow["\lrcorner"{anchor=center, pos=0.125}, draw=none, from=2-2, to=3-3]
\end{tikzcd}\]
are given by the construction of $t$. Furthermore, $t(d_\top d_\perp (\psi)) = \eta'$. If we substitute $d_\top \psi$ by $\mu'$, we have that the left diagram commutes. By the uniqueness of $\mu'$, it follows that $\mu' = d_\top \psi$. This together with the stretched condition of $\psi$ implies that the right diagram commutes if we substitute $\psi$ by $\eta'$. Therefore, by the uniqueness of $\eta'$, we have that  
$\eta' = \psi$. This means that $t(d_\top d_\perp (\psi)) = \psi$.
\end{proof}

Suppose $X$ is a rigid decomposition groupoid. For an object $f \colon x \rightarrow y$, we have a canonical projection $\pi_m \colon (X_{x/})_{/f} \rightarrow X$ defined as the composite
\[\begin{tikzcd}
	{(X_{x/})_{/f}} & {\Dec_{\top} X_{x/}} & {\Dec_\top \Dec_\perp X} & X.
	\arrow["u", from=1-1, to=1-2]
	\arrow["{\Dec_\top v}", from=1-2, to=1-3]
	\arrow["\epsilon", from=1-3, to=1-4]
\end{tikzcd}\]
\begin{lem}\label{lemma: Lavw constr}
Let $\mathcal{C}$ be a discrete algebraic interval. The canonical projection $\pi_m \colon (\mathcal{C}_{\perp_{\mathcal{C}}/})_{/\ledge_{\mathcal{C}}} \rightarrow \mathcal{C}$ has an inverse $L \colon \mathcal{C} \rightarrow (\mathcal{C}_{\perp_{\mathcal{C}}/})_{/\ledge_{\mathcal{C}}}$.
\end{lem}

\begin{proof}
Since $\mathcal{C}$ has a chosen initial object, the projection $d_\perp \colon \mathcal{C}_{\perp_{\mathcal{C}}/} \rightarrow \mathcal{C}$ is an equivalence, and therefore an isomorphism since 
$\mathcal{C}$ and $\mathcal{C}_{\perp_{\mathcal{C}}/}$ are discrete. The map $p_{\perp_{\mathcal{C}}} \colon  \mathcal{C} \rightarrow \mathcal{C}_{\perp_{\mathcal{C}}/}$ denotes the inverse of $d_\perp$. The object $\ledge_\mathcal{C}$ is terminal in $\mathcal{C}_{\perp_{\mathcal{C}}/}$ since it is chosen terminal in $\mathcal{C}$. This implies that the projection $d_\top \colon (\mathcal{C}_{\perp_{\mathcal{C}}/})_{/\ledge_{\mathcal{C}}} \rightarrow \mathcal{C}_{\perp_{\mathcal{C}}/}$ 
has an inverse $p_{\ledge_{\mathcal{C}}} \colon  \mathcal{C}_{\perp_{\mathcal{C}}/} \rightarrow (\mathcal{C}_{\perp_{\mathcal{C}}/})_{/\ledge_{\mathcal{C}}}$ since $d_\top$ is an equivalence between discrete algebraic intervals. So we define $L$ as the composite
\[\begin{tikzcd}
	{\mathcal{C}} & {\mathcal{C}_{\perp_\mathcal{C}/}} & {(\mathcal{C}_{\perp_{\mathcal{C}}/})_{/ \ledge_{\mathcal{C}}}}.
	\arrow["{p_{\perp_\mathcal{C}}}", from=1-1, to=1-2]
	\arrow["{p_{\ledge_\mathcal{C}}}", from=1-2, to=1-3]
\end{tikzcd}\]
Since $\mathcal{C}$ is a discrete algebraic interval, $u$ and $v$ are level-wise injective on objects. This implies that $\pi_m$ is equal to $d_\top \circ d_\perp$. Since $p_{\ledge_{\mathcal{C}}}$ and $p_{\perp_{\mathcal{C}}}$ are inverse of $d_\top$ and $d_\perp$, it follows that $L \circ \pi_m = \identity_{(\mathcal{C}_{\perp_{\mathcal{C}}/})_{/\ledge_{\mathcal{C}}}}$ and $\pi_m \circ L = \identity_{\mathcal{C}}$.
\end{proof}

Recall that the class of culf maps can also be characterised as the class
right orthogonal to the active maps (between representables). That is, $X 
\to Y$ is culf if and only if for every active map $p \colon [m] \actto [n]$ 
and every commutative square
\[
\begin{tikzcd}
\Delta^m \ar[d, -act, "p"'] \ar[r] & X \ar[d]  \\
\Delta^n \ar[r] \ar[ru, dotted, "\exists!"]& Y
\end{tikzcd}
\]
there is a unique filler.
Usually this is about homotopy commutative squares and a contractible 
space of lifts, but for strict culf, it is actually about strictly 
commutative squares and truly unique lifts.

\begin{propo}\label{3existencecell}
  For any $n$-simplex $\lambda \colon \Delta^n \to X$ with long edge 
  $f$, there is a unique lift $\phi_\lambda$ for the square
\[
\begin{tikzcd}
\Delta^1 \ar[d, -act] \ar[r, -act] & I_f \ar[d, "\mfunctor_f"]  \\
\Delta^n \ar[r, "\lambda"'] \ar[ru, dotted, "\phi_\lambda"]& X .
\end{tikzcd}
\]
\end{propo}

\begin{proof}
The square commutes by Lemma \ref{missing 3existencecell}. Indeed the
composite $\Delta^1 \to \Delta^n \to X$ is equal to $f$ since
$f$ was defined as the long edge, and $\Delta^1 \to I_f \to X$
is equal to $f$ by Lemma~\ref{missing 3existencecell}. Furthermore, since $\Delta^1 \to \Delta^n$ is active and $\mfunctor_f$ is strict culf, we have a unique filler which we denote as $\phi_\lambda$.
\end{proof}

\begin{lem}\label{isoimageculf}
Let $G \colon X \rightarrow Y$ be a simplicial map between rigid decomposition groupoids. For $f \in X_1$, there is a unique stretched map $G_f$ fitting into the commutative diagram 
\[\begin{tikzcd}
	{I_f} & {I_{Gf}} \\
	X & Y.
	\arrow["{\mfunctor_f}"', from=1-1, to=2-1]
	\arrow["{\mfunctor_{Gf}}", from=1-2, to=2-2]
	\arrow["G"', from=2-1, to=2-2]
	\arrow[-act, "{G_f}", from=1-1, to=1-2]
	\arrow["{(1)}"{description}, draw=none, from=1-1, to=2-2]
\end{tikzcd}\]
If $G$ is strict culf then $G_f$ is an isomorphism.
\end{lem}

\begin{proof}
In the diagram
\[\begin{tikzcd}
	{I_f} & 1 \\
	{\Dec_{\top} \Dec_\perp X} & {\delta (X_1)} \\
	{\Dec_{\top} \Dec_\perp Y} & {\delta (Y_1)}
	\arrow["{\ulcorner f \urcorner}", from=1-2, to=2-2]
	\arrow[from=1-1, to=1-2]
	\arrow["w"', from=1-1, to=2-1]
	\arrow["{\pilong}"', from=2-1, to=2-2]
	\arrow["{(2)}"{description}, draw=none, from=1-1, to=2-2]
	\arrow["{\Dec_\top \Dec_\perp G}"', from=2-1, to=3-1]
	\arrow["{\delta (G_1)}", from=2-2, to=3-2]
	\arrow["{\pilong}"', from=3-1, to=3-2]
	\arrow["{(3)}"{description}, draw=none, from=2-1, to=3-2]
\end{tikzcd}\]
the square $(2)$ is a strict pullback by construction of $I_f$, and $(3)$ commutes since $\pilong$ is a natural transformation. Combining $(2)$ and $(3)$, we have that the outer diagram commutes. By the pullback property of $I_{Gf}$, we have a unique map $G_f \colon I_f \rightarrow I_{Gf}$ fitting into a commutative diagram
\[\begin{tikzcd}
	{I_f} \\
	{\Dec_\top \Dec_\perp X} & {I_{Gf}} & 1 \\
	X & {\Dec_\top \Dec_\perp Y} & {\delta(Y_1)} \\
	& Y
	\arrow["{\ulcorner Gf \urcorner}", from=2-3, to=3-3]
	\arrow["{w'}"', from=2-2, to=3-2]
	\arrow[from=2-2, to=2-3]
	\arrow["{\pilong}"', from=3-2, to=3-3]
	\arrow["w"', from=1-1, to=2-1]
	\arrow["{\Dec_\top \Dec_\perp G}"{description}, from=2-1, to=3-2]
	\arrow[bend left=15, from=1-1, to=2-3]
	\arrow["{G_f}"', from=1-1, to=2-2]
	\arrow["\epsilon"', from=2-1, to=3-1]
	\arrow["{\epsilon'}", from=3-2, to=4-2]
	\arrow["G"', from=3-1, to=4-2]
	\arrow["{(4)}"{description}, draw=none, from=1-1, to=3-2]
	\arrow["{(5)}"{description}, draw=none, from=2-1, to=4-2]
\end{tikzcd}\]
Combining $(4)$ and $(5)$, we get that $(1)$ commutes. This implies that $\mfunctor_{Gf}G_f(\ledge_{f}) = Gf$. Since $X$ and $Y$ are rigid, the maps $w$ and $w'$ are level-wise injective on objects. The functor $G_f$ is described as follows: for an $n$-simplex $\lambda$ in $I_f$, we have that $(G_f)_n(\lambda) = G_{n+2}(\lambda)$. This description is possible since we work with strict pullbacks, $w$ is level-wise injective on objects and $(1)$ commutes. This implies that $w(\lambda)$ is the same $\lambda$ but interpreted as an $(n+2)$-simplex in $X$. Using this description of $I_f$ it is immediate to see that $G_f(s_0(f)) = s_0(Gf)$ and $G_f(s_1(f)) = s_1(Gf)$, this means that $G_f$ sends the chosen initial and terminal objects of $I_f$ to the chosen initial and terminal objects of $I_{Gf}$. In other words, $G_f$ is stretched.
 
Recall that the chosen edge $\ledge_{Gf} \colon \perp_{I_{Gf}} \rightarrow \top_{I_{Gf}}$ of $I_f$, satisfying that $\mfunctor_{Gf} \ledge_{Gf} = Gf$.  
Applying Proposition \ref{3existencecell} to the map $Gf$, we have a unique stretched map $\phi_{Gf}$ satisfies $\mfunctor_{Gf} \phi_{Gf} = Gf$. But as shown above, $G_f(\ledge_{f})$ and $\ledge_{Gf}$ also satisfy this condition. This implies that $G_f(\ledge_{f}) = \ledge_{Gf}$. 
Furthermore, if $G$ is strict culf, $(3)$ is a strict pullback by Lemma \ref{lemma: pi longt is strict cartesian}. Therefore, combining $(2)$ and $(3)$, we have that $I_f$ is the strict pullback of $\pilong \colon \Dec_\top \Dec_\perp Y \rightarrow \delta (Y_1)$ along $\ulcorner Gf \urcorner \colon 1 \rightarrow \delta (Y_1)$. But this is precisely the definition of $I_{Gf}$. Since $I_f$ and $I_{Gf}$ are pullbacks over the same diagram, it follows that $I_f \cong I_{Gf}$.  Furthermore, this isomorphism is given by $G_f$ since the squares $(3)$ and $(4)$ commute.
\end{proof}

\begin{rema}\label{rema:transitivity 3.20}
The uniqueness of Lemma~\ref{isoimageculf} immediately implies the following
  `transitivity' property of the construction $G \mapsto G_f$: Given
  \[
  \begin{tikzcd}[column sep={5em,between origins}]
  I_f \ar[d, "\mfunctor_f"] \ar[r, "G_f"] & I_{Gf} \ar[d, "\mfunctor_{Gf}"] \ar[r, 
  "H_{Gf}"] & I_{HGf} \ar[d, "\mfunctor_{HGf}"]  \\
  X \ar[r, "G"'] & Y \ar[r, "H"'] & Z,
  \end{tikzcd}
  \]
  we have
  $$
  (H\circ G)_f = H_{Gf} \circ G_f .
  $$
\end{rema}

\section{Stretched-culf factorisation system}\label{sec:Factorisation systems}


\noindent A \emph{factorisation system} in a category $\mathcal{D}$ consists of two classes $E$ and $F$ of maps, that we shall depict as $\twoheadrightarrow$ and $\rightarrowtail$, such that 

\begin{enumerate}
\item The class $F$ is closed under isomorphisms.
\item The classes $E$ and $F$ are orthogonal, $E \perp F$. That is, given $e \in E$ and $f \in F$, for every solid square

\begin{center}
\begin{tikzcd}
. \arrow[r] \arrow[d, "e"', two heads] & . \arrow[d, "f", tail] \\
. \arrow[r] \arrow[ru, dashed]         & .                     
\end{tikzcd}
\end{center}

there is a unique filler.
\item Every map $h$ admits a factorisation
\begin{center}
\begin{tikzcd}
. \arrow[rr, "h"] \arrow[rd, "e"', two heads] &                          & . \\
                                              & . \arrow[ru, "f"', tail] &  
\end{tikzcd}
\end{center}
with $e \in E$ and $f \in F$.
\end{enumerate}

\begin{rema}
The classical notion of orthogonal factorisation system requires that $E$ be closed under isomorphism. In our case it is not required. In case $E$ is not closed under isomorphism we can always saturate it.
\end{rema}

\noindent Let $\Ar^E(\mathcal{D}) \subset \Ar(\mathcal{D})$ denote the full subcategory spanned by the arrows in the left-hand class $E$.

\begin{lem}\cite[Lemma 1.3]{GTK3} \label{cartesiansquares}
The domain projection $\Ar^{E}(\mathcal{D}) \rightarrow \mathcal{D}$ is a cartesian fibration. The cartesian arrows in $\Ar^{E}(\mathcal{D})$ are given by squares of the form
\begin{center}
\begin{tikzcd}
. \arrow[r] \arrow[d, two heads] & . \arrow[d, two heads] \\
. \arrow[r,tail]         & .                     
\end{tikzcd}
\end{center}
\end{lem}

Let $\Interval$ be the category whose objects are discrete algebraic intervals and whose morphisms are functors. We need some preliminary results to prove that the stretched functors as left-hand class and the culf functors as right-hand class form a factorisation system in $\Interval$.

\begin{lem}\label{stretchedtriangle}
Consider the following commutative diagram of simplicial maps
\[\begin{tikzcd}
	& {\mathcal{B}} \\
	{\mathcal{A}} && {\mathcal{C}}.
	\arrow["F"', from=2-1, to=2-3]
	\arrow["S", -act, from=2-1, to=1-2]
	\arrow["G", from=1-2, to=2-3]
\end{tikzcd}\]
where $\mathcal{A}$, $\mathcal{B}$ and $\mathcal{C}$ are discrete algebraic intervals, and $S$ is stretched. Then $F$ is stretched if and only if $G$ is stretched. 
\end{lem}

\begin{proof}
Let $\ledge_{\mathcal{A}} \colon \perp_\mathcal{A} \rightarrow \top_{\mathcal{A}}$ be the unique map from $\perp_\mathcal{A}$ to $\top_\mathcal{A}$.
Let  $\ledge_{\mathcal{B}} \colon \perp_\mathcal{B} \rightarrow \top_{\mathcal{B}}$ and $\ledge_{\mathcal{C}} \colon \perp_\mathcal{C} \rightarrow \top_{\mathcal{C}}$ be the unique maps in $\mathcal{B}$ and $\mathcal{C}$. It is obvious that $F$ is stretched when $G$ is stretched. For the other direction, suppose $F$ stretched. We have that 
\begin{align*}
G(\ledge_{\mathcal{B}})& = G(S(\ledge_{\mathcal{A}}))& \tag{since $S$ is stretched} \\
& =  F(\ledge_{\mathcal{A}}) & \tag{since $F = GS$} \\
& = \ledge_{\mathcal{C}}. & \tag{since $F$ is stretched}
\end{align*}
This means that $G$ is stretched.
\end{proof}

\begin{lem}\label{lemma: mfunctor is invertible}
Let $\mathcal{C}$ be a discrete algebraic interval with long edge $\ledge$. The simplicial map $\mfunctor_{\ledge_{\mathcal{C}}} \colon I_{\ledge} \rightarrow \mathcal{C}$ has an inverse $W \colon \mathcal{C} \rightarrow I_{\ledge}$.
\end{lem}

\begin{proof}
Since $\mathcal{C}$ is a discrete algebraic interval, we have a map $L \colon \mathcal{C} \rightarrow (\mathcal{C}_{\perp_\mathcal{C}/})_{/ \ledge}$ by Lemma \ref{lemma: Lavw constr}. Recall that for an $n$-simplex $\lambda$ in $\mathcal{C}_n$, the $n$-simplex $L(\lambda)$ satisfies that $\longt(L(\lambda)) = s_\top s_\perp \ledge_{\mathcal{C}}$ and $d_\top d_\perp L(\lambda) = \lambda$. Consider the canonical projections $u \colon (\mathcal{C}_{\perp_\mathcal{C}/})_{/ \ledge} \rightarrow \Dec_\top \mathcal{C}_{\perp_\mathcal{C}/}$ and $v \colon \mathcal{C}_{\perp_\mathcal{C}/} \rightarrow \Dec_\perp \mathcal{C}$. Since $\mathcal{C}$ is a discrete algebraic interval, the canonical projections are level-wise injective on objects. So it is straightforward to check that $\pilong( \Dec_\top v \circ u \circ L(\lambda)) = \ledge_\mathcal{C}$. Therefore, the outer diagram
\[\begin{tikzcd}
	\mathcal{C} \\
	& {I_{\ledge}} & 1 \\
	& {\Dec_\top \Dec_\perp \mathcal{C}} & {\delta (\mathcal{C}_1)}
	\arrow["{\ulcorner \ledge\urcorner }", from=2-3, to=3-3]
	\arrow["w", from=2-2, to=3-2]
	\arrow[from=2-2, to=2-3]
	\arrow["{\pilong}"', from=3-2, to=3-3]
	\arrow[from=1-1, to=2-3]
	\arrow["W"', dashed, from=1-1, to=2-2]
	\arrow["{\Dec_{\top} v\circ u \circ L}"{description}, bend right=20, from=1-1, to=3-2]
\end{tikzcd}\]
commutes. By the pullback property of $I_{\ledge}$, we have a unique map $W \colon \mathcal{C} \rightarrow I_{\ledge}$ such that the diagram commutes. Informally, for an $n$-simplex $\lambda$ in $\mathcal{C}_n$, the map $W$ only adds to $\lambda$ the chosen initial edge $\perp_{\mathcal{C}} \rightarrow d_\top(\longt(\lambda))$ by precomposing and the chosen terminal edge $d_\perp(\longt(\lambda)) \rightarrow \top_{\mathcal{C}}$ by postcomposing. Since $w \circ W = {\Dec_{\top} v\circ u \circ L}$ and $d_\top d_\perp L(\lambda) = \lambda$, we have that $\mfunctor_{\ledge} \circ W (\lambda) = \lambda$. By analogous arguments we have that  $W \circ \mfunctor_{\ledge} = \identity_{I_{\ledge}}$.   
\end{proof}

\setcounter{equation}{0}
\begin{lem}\label{lemma: culf are like mono}
 Let $\mathcal{A}$ and $\mathcal{B}$ be discrete algebraic intervals, and let $X$ be a rigid decomposition groupoid. Suppose we have a fork diagram
  $$
  \begin{tikzcd}
  \mathcal{A} \ar[r, shift left, "V"] \ar[r, shift right, "W"'] & {\mathcal{B}} \ar[r, "F"] & X  
  \end{tikzcd}
  $$
(meaning $FV=FW$) where $V$ and $W$ are stretched and $F$ is strict culf. Then already $V=W$.
\end{lem}

\begin{proof}
Let $\varpi_{\mathcal{A}}$ denote the long edge of $\mathcal{A}$ and $\varpi_{\mathcal{B}}$ the 
  long edge of ${\mathcal{B}}$, as usual. Since $V$ and $W$ are stretched, we have
  $V(\varpi_{\mathcal{A}}) = \varpi_{\mathcal{B}} = W(\varpi_{\mathcal{A}})$, so the following diagram is well 
  formed from applying the construction of Lemma~\ref{isoimageculf} (for 
  composable maps as in Remark~\ref{rema:transitivity 3.20}):
  \[
  \begin{tikzcd}[column sep={6em,between origins}, row sep={4em,between 
	origins}]
  I_{\varpi_{\mathcal{A}}} 
	\ar[d, "\mfunctor_{\varpi_{\mathcal{A}}}"'] 
	\ar[r, shift left, "V_{\varpi_{\mathcal{A}}}"] 
	\ar[r, shift right, "W_{\varpi_{\mathcal{A}}}"']
  & I_{\varpi_{\mathcal{B}}} 
	\ar[d, "\mfunctor_{\varpi_{\mathcal{B}}}"] 
	\ar[r, "F_{\varpi_{\mathcal{B}}}"] 
  & I_{F\varpi_{\mathcal{B}}} 
	\ar[d, "\mfunctor_{F\varpi_{\mathcal{B}}}"]  \\
  \mathcal{A} 
	\ar[r, shift left, "V"] 
	\ar[r, shift right, "W"'] 
  & {\mathcal{B}} \ar[r, "F"'] 
  & X .
  \end{tikzcd}
  \]
  Since $FV=FW$, we also have $F_{\varpi_{\mathcal{B}}} V_{\varpi_{\mathcal{A}}} = F_{\varpi_{\mathcal{B}}} 
  W_{\varpi_{\mathcal{A}}}$. This is a consequence of the uniqueness statement in 
  Lemma~\ref{isoimageculf} as in Remark~\ref{rema:transitivity 3.20}. But since $F$ is strict culf,
  the map $F_{\varpi_{\mathcal{B}}}$ is an isomorphism by Lemma~\ref{isoimageculf}. It follows 
  that $V_{\varpi_{\mathcal{A}}} = W_{\varpi_{\mathcal{A}}}$. Finally, since ${\mathcal{A}}$ and ${\mathcal{B}}$ are 
  discrete algebraic intervals and $\varpi_{\mathcal{A}}$ and $\varpi_{\mathcal{B}}$ are their long
  edges, it follows from Lemma~\ref{lemma: mfunctor is invertible} that the two vertical maps 
  $\mfunctor_{\varpi_{\mathcal{A}}}$ and $\mfunctor_{\varpi_{\mathcal{B}}}$ are isomorphisms, and this implies 
  that $V=W$.
\end{proof}

\setcounter{equation}{0}

\begin{lem}\label{stretculffactarrow}
Let $\mathcal{C}$ and $\mathcal{D}$ be discrete algebraic intervals. Let $F \colon \mathcal{C} \rightarrow \mathcal{D}$ be a simplicial map. Then $F$ admits a stretched-culf factorisation. 
\end{lem}

\begin{proof}
Let $\ledge_\mathcal{C}$ be the long $1$-simplex of the interval $\mathcal{C}$. By Lemma \ref{isoimageculf}, we have a stretched map $F_{\ledge_\mathcal{C}} \colon I_{\ledge_\mathcal{C}} \rightarrow I_{F{\ledge_\mathcal{C}}}$ fitting into the commutative diagram  
\begin{center}
\begin{tikzcd}
I_{\ledge_\mathcal{C}} \arrow[d, "\mfunctor_{\ledge_\mathcal{C}}"'] \arrow[r, -act, "F_{\ledge_\mathcal{C}}"]                    &  I_{F{\ledge_\mathcal{C}}} \arrow[d, "\mfunctor_{F{\ledge_\mathcal{C}}}"]   \\
\mathcal{C} \arrow[r, "F"']  & \mathcal{D}.               
\end{tikzcd}
\end{center}
Recall that the vertical arrows are strict culf. The map $\mfunctor_{\ledge_\mathcal{C}}$ is stretched by Example \ref{example: Mstretched}. Since $\mathcal{C}$ is a discrete algebraic interval, we have that $\mfunctor_{\ledge_\mathcal{C}}$ is invertible by Lemma \ref{lemma: mfunctor is invertible}. From the diagram
\[\begin{tikzcd}
	{I_{\ledge_\mathcal{C}}} && {I_{\ledge_\mathcal{C}}} \\
	& {\mathcal{C},}
	\arrow["{\identity_{I_{\ledge_\mathcal{C}}}}", -act, from=1-1, to=1-3]
	\arrow["\mfunctor_{\ledge_\mathcal{C}}"', -act,  from=1-1, to=2-2]
	\arrow["{\mfunctor_{\ledge_\mathcal{C}}^{-1}}"', from=2-2, to=1-3]
\end{tikzcd}\]
it follows that $\mfunctor^{-1}_{\ledge_\mathcal{C}}$ is also stretched, by Lemma \ref{stretchedtriangle}. So altogether, the diagram
\[\begin{tikzcd}
	{\mathcal{C}} && {\mathcal{D}} \\
	& {I_{F{\ledge_\mathcal{C}}}}
	\arrow["F", from=1-1, to=1-3]
	\arrow[-act, "{F_{\ledge_\mathcal{C}} \circ \mfunctor^{-1}_{\ledge_\mathcal{C}}}"', from=1-1, to=2-2]
	\arrow["{\mfunctor_{F{\ledge_\mathcal{C}}}}"', from=2-2, to=1-3]
\end{tikzcd}\]
commutes, where the map $F_{\ledge_\mathcal{C}} \circ \mfunctor^{-1}_{\ledge_\mathcal{C}}$ is stretched and $\mfunctor_{F{\ledge_\mathcal{C}}}$ is culf.
\end{proof}

\setcounter{equation}{0}

\begin{lem}\label{ortoghonalfactsystem}
Let $\mathcal{E}, \mathcal{E'}$ and $\mathcal{C}$  be discrete algebraic intervals. Let $X$ be a rigid decomposition groupoid. For the commutative square of simplicial maps
\begin{center}
\begin{tikzcd}
\mathcal{E} \arrow[r, "G"] \arrow[d, -act, "S"'] & \mathcal{C} \arrow[d, "F"] \\
\mathcal{E'} \arrow[r, "H"'] \arrow[ru, dashed]         & X                    
\end{tikzcd}
\end{center}
where $S\colon \mathcal{E} \actto \mathcal{E}'$ is stretched and $F \colon \mathcal{C} \rightarrow X$ is strict culf, there is a unique filler.
\end{lem}

\begin{proof}
We will first construct a filler $L \colon \mathcal{E}' \rightarrow {\mathcal{C}}$ and then prove it is unique.
For each $n$-simplex $\lambda \colon \Delta^n \rightarrow \mathcal{E}'$, Lemma \ref{existssimplecesinterval} gives an $(n+2)$-simplex  $\eta_\lambda \colon \Delta^{n+2} \actto \mathcal{E}'$
such that
\begin{equation}\label{eq2propofac}
d_\bot d_\top (\eta_\lambda) = \lambda.
\end{equation}
and
\begin{equation}\label{eq1propofac}
\longt(\eta_\lambda) = \ledge_{\mathcal{E}'}
\end{equation} 
We assumed that $S$ is stretched, so $S(\ledge_{\mathcal{E}}) = \ledge_{\mathcal{E}'}$. 
This together with the equation $HS = FG$ and the stretched condition of $S$ are used in the following calculation:
$$\longt(H(\eta_\lambda)) = H (\longt(\eta_\lambda)) = H(\ledge_{\mathcal{E}'}) = H(S(\ledge_{\mathcal{E}})) = F(G(\ledge_{\mathcal{E}})).$$  
In other words, the outer diagram 
\begin{center}
\begin{tikzcd}
1 \arrow[rdd, "H(\eta_\lambda)"', bend right] \arrow[rrd, "\identity", bend left] \arrow[white]{rdd}[black, description]{(3)} \arrow[dr, dotted, "\overline{H\eta_{\lambda}}" description]&                                                                                                                          &                                      \\
                                                                               & (I_{FG\ledge_{\mathcal{E}}})_n \arrow[d, "w_n'"] \arrow[r]  \drpullback & 1 \arrow[d, "FG\ledge_{\mathcal{E}}"] \\
                                                                               & X_{n+2} \arrow[r, "\longt"']                                                                                     & X_1                       
\end{tikzcd}
\end{center}
commutes. The pullback property of $I_{FG\ledge_{\mathcal{E}}}$ gives the dotted map $\overline{H\eta_\lambda} \colon \Delta^n \actto I_{FG\ledge_{\mathcal{E}}}$ such that the diagram commutes. We define the map $V \colon \mathcal{E}' \rightarrow I_{FG\ledge_\mathcal{E}}$ by $V(\lambda) = \overline{H(\eta_\lambda)}$, for each $n$-simplex $\lambda \colon \Delta^n \rightarrow \mathcal{E}'$. It is straightforward to check that $V$ is a simplicial map. Furthermore,
\begin{align*}
\mfunctor_{FG\ledge_{\mathcal{E}}} V (\lambda)  & =   d_\bot d_\top w_n' \overline{H(\eta_\lambda)} & \tag{\mbox{by def. of $\mfunctor_{FG\ledge_{\mathcal{E}}}$  and $V$}} \\
& =  d_\bot d_\top H(\eta_\lambda) & \tag{\mbox{by triangle (3)}}  \\
& =  H d_\bot d_\top(\eta_\lambda)  & \tag{\mbox{since $H$ is a sim. map}} \\
& =  H (\lambda). & \tag{\mbox{by Eq.~(\ref{eq2propofac})}}
\end{align*}
This means that the following diagram commutes
\[\begin{tikzcd}
	& {I_{FG\ledge_{\mathcal{E}}}} \\
	{\mathcal{E}'} && {X.}
	\arrow[""{name=0, anchor=center, inner sep=0}, "H"', from=2-1, to=2-3]
	\arrow["V", from=2-1, to=1-2]
	\arrow["{\mfunctor_{FG\ledge_{\mathcal{E}}}}", from=1-2, to=2-3]
	\arrow["{(4)}"{description}, Rightarrow, draw=none, from=1-2, to=0]
\end{tikzcd}\]
Since $F$ is strict culf, Lemma \ref{isoimageculf} gives an isomorphism $K \colon I_{G\ledge_\mathcal{E}} \rightarrow I_{FG\ledge_\mathcal{E}}$ fitting into the commutative diagram
\[\begin{tikzcd}
	{I_{G\ledge_{\mathcal{E}}}} & {\mathcal{C}} \\
	{I_{FG\ledge_{\mathcal{E}}}} & {X.}
	\arrow["F", from=1-2, to=2-2]
	\arrow["{\mfunctor_{FG\ledge_{\mathcal{E}}}}"', from=2-1, to=2-2]
	\arrow["{\mfunctor_{G\ledge_{\mathcal{E}}}}", from=1-1, to=1-2]
	\arrow["{K^{-1}}", from=2-1, to=1-1]
	\arrow["{(5)}"{description}, draw=none, from=1-1, to=2-2]
\end{tikzcd}\]
Combining the commutativity of $(4)$ and $(5)$ gives $H = F \circ \mfunctor_{G\ledge_{\mathcal{E}}} \circ K^{-1} \circ V$, and the hypothesis that $H \circ S = F \circ G$, we have that
\begin{equation*}\tag{6}
F \circ \mfunctor_{G\ledge_{\mathcal{E}}} \circ K^{-1} \circ V \circ S = H \circ S = F \circ G.
\end{equation*}
Lemma \ref{lemma: culf are like mono} says that it is possible to cancel $F$ in Eq.~$(6)$ if $F$ is strict culf, which is given by hypothesis. Therefore, the diagram
\[\begin{tikzcd}
	& {\mathcal{E}'} & {I_{FG\ledge_\mathcal{E}}} & {I_{G\ledge_{\mathcal{E}}}} \\
	{\mathcal{E}} &&& {\mathcal{C}}
	\arrow["V", from=1-2, to=1-3]
	\arrow["G"', from=2-1, to=2-4]
	\arrow["S",-act, from=2-1, to=1-2]
	\arrow["{\mfunctor_{G\ledge_{\mathcal{E}}}}", from=1-4, to=2-4]
	\arrow["{K^{-1}}", from=1-3, to=1-4]
	\arrow["{(7)}"{description}, draw=none, from=1-2, to=2-4]
\end{tikzcd}\]
commutes. The pieces now fit together to form the commutative diagram
\begin{center}
\begin{tikzcd}
\mathcal{E} \arrow[ddd,-act, "S"'] \arrow[rr, "G"]  \arrow[white]{rdd}[black, description]{(7)}   &                                   & \mathcal{C} \arrow[ddd, "F"] \\
                                               & I_{G\ledge_\mathcal{E}}  \arrow[ru, "\mfunctor_{G\ledge_{\mathcal{E}}}"]     \arrow[white]{rd}[black, description]{(5)}            &           {}                   \\
                                               & I_{FG\ledge_\mathcal{E}}  \arrow[rd, "\mfunctor_{FG\ledge_{\mathcal{E}}}"] \arrow[u, "K^{-1}"] \arrow[white]{d}[black, description]{(4)} &      {}                        \\
\mathcal{E}' \arrow[rr, "H"'] \arrow[ru, "V"] &                   {}                & X.                 
\end{tikzcd}
\end{center}
We define $L \colon \mathcal{E}' \rightarrow \mathcal{C} $ as $L := \mfunctor_{G\ledge_{\mathcal{E}}} \circ K^{-1} \circ V$. 
Finally we establish uniqueness, exploiting that we already have 
  existence given by the functor $L$. Suppose we have two fillers 
  \[
  \begin{tikzcd}[sep={5em,between origins}]
  \mathcal{E} \ar[r, "G"] \ar[d, -act, "S"'] & \mathcal{C} \ar[d, "F"]   \\
  \mathcal{E}' 
  \ar[r, "H"'] 
  \ar[ru, shift left, "L_1"]
  \ar[ru, shift right, "L_2"'] 
  & X.
  \end{tikzcd}
  \]
  Now factor $G$ as a 
  stretched map $G'$ followed by a strict culf map $C$,
  \[
  \begin{tikzcd}
  \mathcal{E} \ar[r, -act, "G'"] \ar[rd, "G"']& \mathcal{C}' \ar[d, "C"] \\
  & \mathcal{C}   
\end{tikzcd}
\]  
  which is possible by 
  Lemma~\ref{stretculffactarrow}. Now we can invoke existence of lifts to the situation
  \[
  \begin{tikzcd}[sep={6em,between origins}]
  \mathcal{E} \ar[r, -act, "G'"] \ar[d, -act, "S"'] & \mathcal{C}' \ar[d, "C"]   \\
  \mathcal{E}'
  \ar[ru, dotted, shift left, "L'_1"]
  \ar[ru, dotted, shift right, "L'_2"'] 
  \ar[r, shift left, "L_1"]
  \ar[r, shift right, "L_2"'] 
  & \mathcal{C}
  \end{tikzcd}
  \]
  since $S$ is stretched and $C$ is culf. This gives the existence of
  $L'_1$ and $L'_2$ as indicated, and they are stretched by Lemma~\ref{stretchedtriangle}
  since both $S$ and $G'$ are stretched. But now we are in position to
  apply Lemma~\ref{lemma: culf are like mono}: Since we have $FL_1=FL_2$ (as both are equal to
  $H$), we also have $FCL'_1 = FCL'_2$. Furthermore, since $FC$ is culf and $L'_1$
  and $L'_2$ are stretched, we conclude by Lemma~\ref{lemma: culf are like mono} that already
  $L'_1= L'_2$, and therefore also $L_1=L_2$. 
\end{proof}

\begin{rema}\label{culfconditiondiagonalortoghonal}
In Lemma \ref{ortoghonalfactsystem}, we have that the diagram
\begin{center}
\begin{tikzcd}
\mathcal{E} \arrow[r, "G"] \arrow[d, -act, "S"'] & \mathcal{C} \arrow[d, "F"] \\
\mathcal{E'} \arrow[r, "H"'] \arrow[ru, "L" description]     & X                     
\end{tikzcd}
\end{center}
commutes. By hypothesis $F$ is culf. Therefore, $L$ is a culf if and only if $H$ is culf. On the other hand, by hypothesis $S$ is stretched and applying Lemma \ref{stretchedtriangle}, we have that $L$ is  stretched if and only if $G$ is stretched.
\end{rema}

\begin{rema}\label{remaortoghonal}
If we had required $X$ to be a discrete algebraic interval, then Lemma \ref{ortoghonalfactsystem} would say that the stretched and strict culf maps are orthogonal classes of maps in the category of discrete algebraic intervals and simplicial maps, as exploited in the following proposition. It will be important later in \ref{subsec:The complete decomposition groupoid $U_X$} that we allow $X$ to be more general than just a discrete algebraic interval.
\end{rema}

\begin{propo}\label{propofactsystemInt}
The stretched maps as left-hand class and the strict culf functors as right-hand class form a factorisation system in $\Interval$.
\end{propo}

\begin{proof}
The strict culf maps are closed under isomorphism. We have that every simplicial map $F$ in $\Interval$ admits a stretched-culf factorisation by Lemma \ref{stretculffactarrow}. Therefore, we only have to prove that the classes are orthogonal, which follows from Lemma \ref{ortoghonalfactsystem}.
\end{proof}

\section{The decomposition groupoid $\mathcal{U}$}
\label{sec:U}

\noindent  In Section \ref{sec:Factorisation systems}, the stretched-culf factorisation system was defined in $\Interval$, which we can use to define a fibration that encodes the pseudo-simplicial groupoid of discrete algebraic intervals.

Let $\Ar^{s}(\Interval) \subset \Ar(\Interval)$ denote the full subcategory spanned by the stretched functors. $\Ar^{s}(\Interval)$ is a cartesian fibration over $\Interval$ via the domain projection by Lemma \ref{cartesiansquares}. We now restrict this cartesian fibration to $\simplexcategory \subset \Interval$

\begin{center}
\begin{tikzcd}
\Ar^s(\Interval)_{\vert_\simplexcategory} \drpullback \arrow[r, "\fffunctor"] \arrow[d, "\dominio"'] & \Ar^s(\Interval) \arrow[d, "\dominio"] \\
\simplexcategory  \arrow[r, "\fffunctor"']                                      & \Interval.                             
\end{tikzcd}
\end{center}
We put
$$\mathcal{U} := \Ar^s(\Interval)_{\vert_\simplexcategory}.$$
$\mathcal{U} \rightarrow \simplexcategory$ is the cartesian fibration of subdivided algebraic discrete  intervals. By Lemma \ref{cartesiansquares}, the cartesian maps in $\mathcal{U}$ are squares
\begin{center}
\begin{tikzcd}
\Delta^k \arrow[r] \arrow[d, -act] & \Delta^n \arrow[d, -act] \\
\mathcal{C} \arrow[r, "\CULF"']                                      & \mathcal{D}.                           
\end{tikzcd}
\end{center}

\noindent The cartesian fibration $\mathcal{U} \rightarrow \simplexcategory$ determines a right fibration $\mathcal{U}^{\cart} \rightarrow \simplexcategory$, and hence by straightening \cite[Theorem 8.3.1]{Bor2} a simplicial groupoid
$$U \colon \simplexcategory^{\op} \rightarrow \widehat{\Grpd}$$
where $\widehat{\Grpd}$ is the 2-category of large groupoids, functors and natural transformations. 

The following result is due to Gálvez--Kock--Tonks \cite[Theorem 4.8]{GTK1}, who prove it in the more general setting of $\infty$-groupoids.

\begin{teo}\label{theoremUdecompositiongroupoid}
The simplicial groupoid $U \colon \simplexcategory^{\op} \rightarrow \widehat{\Grpd}$ is a complete decomposition groupoid.
\end{teo}

\begin{blanko}
{The complete decomposition groupoid $U_X$}
\label{subsec:The complete decomposition groupoid $U_X$}
\end{blanko}

The decomposition groupoid $U \colon \simplexcategory^{\op} \rightarrow \widehat{\Grpd}$ is not a strict
  simplicial object but only a pseudo-simplicial object. In a
  famous paper~\cite{Jardine}, Jardine figured out all the $2$-cell data
  and 17 coherence laws for pseudo-simplicial objects in terms of
  face and degeneracy maps. We overcome the difficulty of working with these coherence laws by building a local strict model,  a kind of neighbourhood $U_X \subset U$ around the discrete algebraic intervals of a given rigid decomposition groupoid $X$. 
  
From the viewpoint of cartesian fibrations, the problem with $\mathcal{U} \to \simplexcategory$ 
is that it is not split. It is not possible to define a coherent global 
choice of cartesian lifts of arrows in $\simplexcategory$. To fix this, we
restrict to a full subcategory $\mathcal{U}_X$ consisting only of the 
(subdivided) intervals of $X$ (and not even including isomorphic 
intervals). 

\begin{defi}
Let $\mathcal{U}_X$ denote the category whose objects are pairs $(f, \phi_{\lambda} \colon \Delta^n \actto I_f)$, with $f \in X_1$ and $\phi_{\lambda} \colon \Delta^n \actto
I_f$ an $n$-subdivision of the interval $I_f$. Here $\lambda \in X_n$ and $\longt(\lambda) = f$. The morphisms $F \colon (\phi_{\lambda}, f) \to (\phi_{\lambda'}, f')$ are culf maps $F \colon I_f \to I_{f'}$ such that $F(\phi_{\lambda}) = \phi_{\lambda'}$. To simplify the notation, we will write the objects of $\mathcal{U}_X$ as $(\phi_{\lambda} ,I_f)$.
\end{defi}  
  The benefit is that when everything is inside $X$, we can make 
canonical choices of cartesian lifts. They are given by the following lemma.


\setcounter{equation}{0}

\begin{lem}\label{precomposeinterval} 
  Let $X$ be a rigid decomposition groupoid, and let $p \colon \Delta^{n'} \to \Delta^n$
  be a map in $\simplexcategory$. For any $n$-simplex $\lambda \colon \Delta^n 
  \to X$, the commutative triangle
    \[
  \begin{tikzcd}[column sep={2.5em,between origins}]
  \Delta^{n'} \ar[rr, "p"] \ar[rd, "\lambda'"'] && \Delta^n \ar[ld, "\lambda"]  \\
   & X
  \end{tikzcd}
\]
gives the standard factorisations (\ref{3existencecell}) as in the solid square
\begin{equation*}
\begin{tikzcd}[column sep={2.5em,between origins}]
\Delta^{n'} \ar[rr, "p"] \ar[d, -act, "\phi_{\lambda'}"'] && \Delta^n 
\ar[d, -act, "\phi_{\lambda}"]  \\
I_{f'} \ar[rr, dotted, "c^p_\lambda"] \ar[rd, "\mfunctor_{f'}"'] && I_f \ar[ld, "\mfunctor_f"] \\
& X. &
\end{tikzcd}
\end{equation*}
The statement is that there is a unique filler $c^p_\lambda$ as indicated 
with the dotted arrow, and this map is strict culf.
\end{lem}

\begin{proof}
Since $\phi_{\lambda'}$ is stretched and $\mfunctor_f$ is strict culf, the required map $c^p_\lambda$ is given by Lemma \ref{ortoghonalfactsystem}, and it is strict culf by Remark
\ref{culfconditiondiagonalortoghonal}.
\end{proof}

\setcounter{equation}{0}

Notice how the ambient $X$ is crucially exploited to characterise the lift uniquely. We also spell out how this choice of lifts act on isomorphisms:

\begin{lem}\label{definitiondiinmorphismUX}
  Let $X$ be a rigid decomposition groupoid, and consider an isomorphism $F
  \colon (I_f,\phi_\lambda) \isopil (I_g,\phi_\mu)$ in $(U_X)_n$, as on the
  right in the following diagram. For any map $p \colon \Delta^{n'} \to \Delta^n$
  in $\simplexcategory$, there is induced an isomorphism $F' \colon
  (I_{f'},\phi_{\lambda'}) \isopil (I_{g'},\phi_{\mu'})$ in $(U_X)_{n'}$,
  as indicated with the dotted arrow:
  
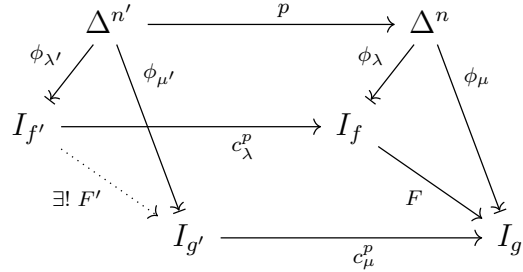
\[
\begin{tikzcd}[column sep={2.5em,between origins}]
& \Delta^{n'} 
\ar[rrrr, pos=(0.55), "p"] 
\ar[ld, -act, "\phi_{\lambda'}"'] 
\ar[rdd, -act, pos=(0.3), "\phi_{\mu'}"] &&&& 
\Delta^n 
\ar[ld, -act, "\phi_{\lambda}"'] 
\ar[rdd, -act, pos=(0.3), "\phi_{\mu}"] &  \\
I_{f'} \ar[rrrr, pos=(0.7), "c^p_\lambda"'] \ar[rrd, dotted, "\exists!\; F'"'] &&&& I_f 
\ar[rrd, "F"'] && \\
&& I_{g'} \ar[rrrr, pos=(0.55), "c^p_\mu"'] &&&& I_{g} 
\end{tikzcd}
\]
This $F'$ is characterised as the unique isomorphism in 
$(U_X)_{n'}$
compatible with the canonical interval inclusions $c^p_\lambda$ and $c^p_\mu$ (that 
is, unique making the whole diagram commute).
\end{lem}

Let us explain the notation. The domain and codomain of $F$ are objects in 
$(U_X)_n$:
  as usual, the 
  notation refers to an $n$-simplex $\lambda \colon \Delta^n \to X$ with long 
  edge $f:=\operatorname{long}(\lambda)$ and another $n$-simplex
  $\mu \colon \Delta^n \to X$ with long 
  edge $g:=\operatorname{long}(\mu)$, and $F \colon I_f \isopil I_g$ is
  an isomorphism of intervals compatible with the subdivisions 
  $\phi_\lambda \colon \Delta^n \actto I_f$ and $\phi_\mu \colon \Delta^n \actto I_g$
  provided by Proposition \ref{3existencecell}. 
  
  The map $p \colon \Delta^{n'} \to \Delta^n$ 
  gives rise to $n'$-simplices $\lambda'$ and $\mu'$ in $X$:
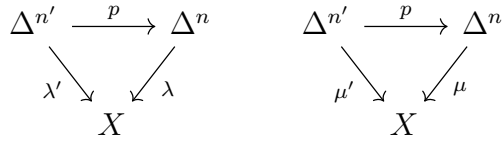
  \[
  \begin{tikzcd}[column sep={2.5em,between origins}]
  \Delta^{n'} \ar[rr, "p"] \ar[rd, "\lambda'"'] && \Delta^n \ar[ld, "\lambda"]  \\
   & X
  \end{tikzcd}
  \qquad
  \begin{tikzcd}[column sep={2.5em,between origins}]
  \Delta^{n'} \ar[rr, "p"] \ar[rd, "\mu'"'] && \Delta^n \ar[ld, "\mu"]  \\
   & X
  \end{tikzcd}
  \]
  and induced interval inclusions (strict culf maps)
  \[
  I_{f'} \stackrel{c^p_\lambda}\longrightarrow I_f
  \qquad\qquad\quad
    I_{g'} \stackrel{c^p_\mu}\longrightarrow I_g
  \]
  as in Lemma~\ref{precomposeinterval}. 

\begin{proof}
  Rearranging the bottom and left part of the diagram as
  
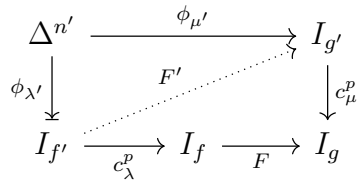
\[
  \begin{tikzcd}
  \Delta^{n'} \ar[d, -act, "\phi_{\lambda'}"'] \ar[rr, "\phi_{\mu'}"] && 
  I_{g'} \ar[d, "c^p_\mu"]  \\
  I_{f'} \ar[rru, dotted, "F'"] \ar[r, "c^p_\lambda"']  & I_f \ar[r, "F"'] & I_g
  \end{tikzcd}
  \]
  we see that $F'$ is the unique lift existing by Lemma \ref{ortoghonalfactsystem} since $\phi_{\lambda'}$ is stretched and $c^p_\mu$ is strict culf.
\end{proof}

\begin{rema}\label{rem: exp active pF}
In Lemma \ref{definitiondiinmorphismUX}, the diagram 
\begin{center}
\begin{tikzcd}
\Delta^{n'} \arrow[d, -act, "\phi_{\lambda'}"'] \arrow[r, -act, "\phi_{\mu'}"]                   & I_{{g'}} \arrow[d, "{c^p_\mu}"] \\
I_{{f'}} \arrow[r, "F c^p_\lambda"'] \arrow[ru, "{F'}" description] & I_g                                                  
\end{tikzcd}
\end{center}
commutes. When $p$ is active, we have that ${f'} = f$ and ${g'} = g$. Furthermore, if we substitute ${F'}$ by $F$, the diagram also commutes. By Lemma \ref{ortoghonalfactsystem}, we have that ${F'} = F$.  Therefore, when we work with an active map, we will use $F$ instead of $F'$.
\end{rema}

\setcounter{equation}{0}

Whit these preparations, we can establish that $\mathcal{U}_X$ is split:

\begin{propo}\label{proofsimplicialsetux}
  The cartesian fibration $\mathcal{U}_X \to \simplexcategory$ is split.
  The splitting is given by the cartesian arrows chosen in 
  Lemma~\ref{precomposeinterval}.
\end{propo}

\begin{proof}
  That this choice of lifts constitutes a splitting means that it is 
  functorial: composites of chosen lifts are lifts of composites, and lift 
  of identity arrows are identity arrows.
  For composition: given the solid diagram
\begin{equation*}
\begin{tikzcd}[column sep={5.5em,between origins}]
  \Delta^{n''} \ar[r, "q"] \ar[d, -act, "\phi_{\lambda''}"'] &
  \Delta^{n'} \ar[r, "p"] \ar[d, -act, "\phi_{\lambda'}"'] & \Delta^n 
\ar[d, -act, "\phi_{\lambda}"]  \\
I_{f''} \ar[r, dotted, "c^q_{\lambda'}"] \ar[rd, "\mfunctor_{f''}"'] &
I_{f'} \ar[r, dotted, "c^p_\lambda"] \ar[d, "\mfunctor_{f'}"'] & I_f \ar[ld, "\mfunctor_f"] \\
& X &
\end{tikzcd}
\end{equation*}
there are induced $c^q_{\lambda'}$ and $c^p_\lambda$ making the whole 
diagram commute. Now by the uniqueness characterisation of $c$-maps,
the composite $c^p_{\lambda} \circ c^q_{\lambda'}$ must be equal to 
$c^{pq}_{\lambda}$, as required.
\end{proof}

Knowing that the $c$-maps provide a splitting for $\mathcal{U}_X \to 
\simplexcategory$, there is now induced a strict functor
$$
U_X \colon  \simplexcategory^{\op} \to \Grpd
$$
(groupoid-valued functor corresponding to the associated right fibration). We can now simply spell out explicitly what this simplicial groupoid is. On objects, we simply have to describe the fibres:
$(U_X)_n$ is thus the groupoid whose objects are subdivided intervals of 
$X$, say $\phi_\lambda: \Delta^n \actto I_f$ (for some $\lambda\in X_n$ with long edge 
$f$), and whose arrows are the vertical arrows in $\mathcal{U}_X$, namely
strictly commutative triangles
\begin{equation*}
\begin{tikzcd}[column sep={2.5em,between origins}]
 & \Delta^n \ar[ld, -act, "\phi_{\lambda}"'] \ar[rd, -act, "\phi_{\mu}"] &  \\
I_f \ar[rr, "\sim"] && I_g.
\end{tikzcd}
\end{equation*}
Note that since $\mathcal{U}_X$ was defined as full inside $\mathcal{U}$, there are no
compatibility requirement with the `inclusions' $\mfunctor_f \colon  I_f \to X$ and $\mfunctor_g \colon 
I_g \to X$.

The simplicial operators act via cartesian lifts: the formula for $p: 
\Delta^{n'} \to \Delta^n$ is
$$
p^\ast \big( \Delta^n \stackrel{\phi_\lambda}\actto I_f \big)
=
( \Delta^{n'} \stackrel{\phi_{\lambda'}}\actto I_{f'})
$$
with reference to the chosen cartesian arrow 
\begin{equation}
\begin{tikzcd}[column sep={2.5em,between origins}]
\Delta^{n'} \ar[rr, "p"] \ar[d, -act, "\phi_{\lambda'}"'] && \Delta^n 
\ar[d, -act, "\phi_{\lambda}"] \\
I_{f'} \ar[rr, "c^p_\lambda"'] 
&& I_f. 
\end{tikzcd}
\end{equation}

The action of the simplicial operator on an isomorphism in $(U_X)_n$, say
$F \colon (I_f,\phi_\lambda) \isopil (I_g,\phi_\mu)$, is given by
Lemma~\ref{definitiondiinmorphismUX}. Indeed, this lemma is nothing but the standard
description of how a vertical isomorphism is transported along a cartesian
lift.

(Note that the construction of the isomorphism $F'$, which in Lemma~\ref{definitiondiinmorphismUX} 
was given using the stretched-culf factorisation system, can also be 
regarded as the argument why general arrows in a cartesian fibration
factor uniquely through cartesian arrows. Indeed the stretched-culf 
factorisation system is the abstract reason why we have a cartesian 
fibration.)

\begin{lem}\label{pisofibrationUX}\label{lemmadiscreteiso}
Let $p \colon [n] \actto [m]$ be an active map in $\simplexcategory$. Then $p^* \colon (U_X)_m \rightarrow (U_X)_n$ is a discrete fibration.
\end{lem}

\begin{proof}
Let $(I_g, \phi_\mu)$ be an object in $(U_X)_m$ and let $F \colon (I_f, \phi_\lambda) \rightarrow p^*(I_g, \phi_\mu)$ be a morphism in $(U_X)_n$. To provide a lift is to use the same underlying $F$ (by \ref{rem: exp active pF}, since $p$ is active), but the compatibility which characterises morphisms in $(U_X)_m$ is now with the $\phi$ maps from $\Delta^m$ instead of from $\Delta^n$. In other words, we need to find  the dashed arrow in the diagram 
\begin{center}
\begin{tikzcd}
\Delta^n \arrow[rr, "\phi_\lambda"] \arrow[rd, "p"'] &                                                                      & I_f \arrow[d, "F"] \\
                                             & \Delta^m \arrow[r, "\phi_\mu"'] \arrow[ru, "\phi_{\overline{\eta}}", dashed] & I_g,               
\end{tikzcd}
\end{center}
which is possible since $F$ is invertible, in fact $\overline{\eta} = \mfunctor_f F^{-1}(\phi_\lambda)$. Therefore $p^*$ is a discrete fibration.
\end{proof}

\begin{exa}
In general, the image of an inert map of $\simplexcategory^{\op}$ under $U_X$ is not a discrete fibration. Let $\mathcal{C}$ be the category pictured by the following commutative diagram
\[\begin{tikzcd}
	& y \\
	x && z & w \\
	& {y'}
	\arrow["a", from=2-1, to=1-2]
	\arrow["{a'}"', from=2-1, to=3-2]
	\arrow["b", from=1-2, to=2-3]
	\arrow["{b'}"', from=3-2, to=2-3]
	\arrow["f"{description}, from=2-1, to=2-3]
	\arrow["g"', from=2-3, to=2-4]
\end{tikzcd}\]

Since $x$ is an initial object and $w$ is a terminal object, we have that $N(\mathcal{C}) \simeq I_{gf}$ by Lemma \ref{lemma: mfunctor is invertible}. Let $\phi_{gf}$ be the $2$-simplex induced by the morphisms $f$ and $g$ in $(N({\mathcal{C}}))_2$.
Let $(N(\mathcal{C}), \phi_{gf}) \in (U_{N(\mathcal{C})})_2$ be the interval construction of $\phi_{gf}$. Applying $d_0$ to $\phi_{gf}$, we have that $d_0(N(\mathcal{C}), \phi_{gf})  = (I_g, \phi_g)$. Let $\identity_{I_g} \colon (I_g, \phi_g) \rightarrow (I_g, \phi_g)$ be the identity morphism in $(U_{N(\mathcal{C})})_1$.

We can construct two lifts of $\identity_{I_g}$ in $(U_{N(\mathcal{C})})_2$. Let $F \colon (N(\mathcal{C}), \phi_{gf}) \rightarrow (N(\mathcal{C}), \phi_{gf})$ be the functor that fixes all the objects in $\mathcal{C}$ except $y$ and $y'$. It is easy to check that $d_0 F = \identity_{I_{g}}$. On the other hand it is straightforward to see that the identity morphism $\identity_{I_{gf}}$ satisfies $d_0 \identity_{I_{gf}} = \identity_{I_{g}}$. Therefore, $F$ and $\identity_{I_{gf}}$ are two different lifts of $\identity_{I_{g}}$.

\end{exa} 

When $S$ is a simplicial groupoid, we have a simplicial set induced by the \textit{object functor} $\Obj \colon \Grpd \rightarrow \Set$, which is defined as forgetting the morphisms. We denote $\Obj \circ S$ as $S^0$.

\begin{propo}\label{propo: U0 is iso X}
Let $X \colon \simplexcategory^{\op} \rightarrow \Grpd$ be a rigid decomposition groupoid. Then $U_X^0 \cong X^0$.
\end{propo}  

\begin{proof}
The proof is easily deduced from the fact that every object $(I_f, \phi_\lambda)$ in $(U_X^0)_n$ corresponds to some $\lambda$ in $X_n^0$ by definition of $U_X$.
\end{proof}

\begin{lem}\label{proofdspaceux}
Let $X$ be a rigid decomposition groupoid. The simplicial groupoid $U_X \colon \simplexcategory^{\op} \rightarrow \Grpd$ is a decomposition groupoid.
\end{lem}

\begin{proof} 
We need to show that for an active-inert pullback square in $\simplexcategory^{\op}$, the image under $U_X$ is a homotopy pullback 
\begin{center}
\begin{tikzcd}
(U_{X})_{m} \arrow[r, -act,  "g"] \arrow[d, tail, "h"']& (U_{X})_{n} \arrow[d, tail, "{\overline{h}}"]\\
(U_{X})_{k} \arrow[r, -act ,"\overline{g}"'] & (U_{X})_{s}. 
\end{tikzcd}
\end{center}
Here $g$ and $\overline{g}$ are active maps,  $h$ and $\overline{h}$ are inert maps. Since $g$ and $\overline{g}$ are active maps, they are discrete fibrations by Lemma \ref{lemmadiscreteiso}. Therefore,  we can work with strict fibres. By Lemma \ref{lemmapullbackfibres}, the previous square  is a homotopy pullback if and only if for each object $(I_f, \phi_\lambda)$ in $(U_X)_n$, corresponding to some $\lambda \in X_n$, the morphism $h' \colon \Fib_{(I_f, \phi_\lambda)}(g) \rightarrow \Fib_{\overline{h}(I_f, \phi_\lambda)}(\overline{g})$, induced by the morphism $h$, is an equivalence. Here $\Fib_{(I_f, \phi_\lambda)}(g)$ is the strict fibre of $g$ over $(I_f, \lambda)$ and 
$\Fib_{\overline{h}(I_f, \phi_\lambda)}(\overline{g})$ is the strict fibre of $\overline{g}$ over $\overline{h}(I_f, \phi_\lambda)$. 

The fibres $\Fib_{(I_f, \phi_\lambda)}(g)$ and $\Fib_{\overline{h}(I_f, \phi_\lambda)}(\overline{g})$ are discrete groupoids since $g$ and $g'$ are discrete fibrations. Furthermore, as a consequence of Proposition \ref{propo: U0 is iso X}, we have a bijection between the objects of $X$ and $U_X$. This implies that the diagram 
\[\begin{tikzcd}
	{\Fib_{(\lambda)}(g)} & {\Fib_{\overline{h}\lambda}(\overline{g})} \\
	{\Fib_{(I_f, \phi_\lambda)}(g)} & {\Fib_{\overline{h}(I_f, \phi_\lambda)}(\overline{g})}
	\arrow["{h''}", from=1-1, to=1-2]
	\arrow["{h'}", from=2-1, to=2-2]
	\arrow[ from=1-1, to=2-1]
	\arrow[ from=1-2, to=2-2]
\end{tikzcd}\]
commutes. Here $h'' \colon \Fib_{\lambda}(g) \rightarrow \Fib_{\overline{h}\lambda}(\overline{g})$ is the morphism induced by $h \colon X_m \rightarrow X_k$. Since $X$ is rigid, the morphism $h''$  is an equivalence. Since $\Fib_{(I_f, \phi_\lambda)}(g)$ and $\Fib_{\overline{h}(I_f, \phi_\lambda)}(\overline{g})$ are discrete groupoids, the vertical maps are equivalences by Proposition \ref{propo: U0 is iso X}. Hence, the map $h' \colon \Fib_{(I_f, \phi_\lambda)}(g) \rightarrow \Fib_{\overline{h}(I_f, \phi_\lambda)}$ is an equivalence.
\end{proof}

\setcounter{equation}{0}

\begin{lem}\label{UXcomplete}
Let $X$ be a rigid decomposition groupoid. Then the decomposition groupoid $U_X$ is a complete.
\end{lem}

\begin{proof}
To establish that $U_X$ is complete, we need to check that the map $s_0 \colon (U_X)_0 \rightarrow (U_X)_1$ is a monomorphism. This means that we need to show that the fibre is either empty or singleton.
Remember that the objects in $(U_X)_0$ are given by $0$-simplices of $X$. Combining this with the fact that the long edge of a $0$-simplex is $s_0(x)$, we have that the objects in $(U_X)_0$ are of the form $(I_{s_0(x)}, \phi_x)$.
Since $s_0$ is active, we have that $s_0$ is a discrete fibration by Lemma \ref{lemmadiscreteiso}. Therefore, we will consider strict fibres. For $f \in X_1$ denote by $\phi_f \colon \Delta^1 \actto I_f$ the unique stretched map. The strict fibre over $(I_f, \phi_f) \in (U_X)_1$ is given by the strict pullback
\begin{center}
\begin{tikzcd}
{\Fib_{(I_f, \phi_f)}(s_0)} \drpullback \arrow[r] \arrow[d] & (U_X)_0 \arrow[d, "s_0"] \\
1 \arrow[r, "{\ulcorner (I_f, \phi_f) \urcorner}"']                  & (U_X)_1.                
\end{tikzcd}
\end{center}
Unless $f$ is degenerate, the strict fibre is empty. In the degenerate case, consider $(I_{s_0 (x)}, \phi_{s_0(x)})$ and $(I_{s_0 (y)}, \phi_{s_0(y)})$ two objects in ${\Fib_{(I_f, \phi_f)}(s_0)}$ such that $(I_{s_0 (x)}, \phi_{s_0(x)}) = (I_{s_0 (y)}, \phi_{s_0(y)})$. This means that $\phi_{s_0(x)} = \phi_{s_0(y)}$. This together with the rigid condition of $X$ (the map $s_0 \colon X_0 \rightarrow X_1$ is a monomorphism) implies that $x = y$.
\end{proof}

To construct a map from $X$ to $U_X$, the following result is necessary:

\begin{lem}\label{lemma: functorial of I in arrows}
  Given an isomorphism $\alpha \colon \lambda \to \mu$ in $X_n$, there is induced an isomorphism
  \[
  \begin{tikzcd}[column sep={2.5em,between origins}]
 & \Delta^n \ar[ld, -act, "\phi_{\lambda}"'] \ar[rd, -act, "\phi_{\mu}"] &  \\
I_f \ar[rr, "F_\alpha"'] && I_g.
\end{tikzcd}
\]
As usual, $f = \longt(\lambda)$ and $g = \longt(\mu)$.
\end{lem}
\begin{proof}
The main point is to prove it just for $1$-simplices: given $\longt(\alpha) \colon f \to g$ in 
$X_1$, the interval $I_f$ is the fibre over $f \in \delta(X_1)$ of the whole
simplicial groupoid $\Dec_\top\Dec_\bot X \to \delta(X_1)$, and $I_g$ is 
the fibre over $g$. This is a level-wise fibration over $\delta(X_1)$, since it is formed entirely of
active maps. But in a fibration, any isomorphism $f \cong g$ between  two objects in the base induces an isomorphism $F_\alpha \colon I_f \to I_g$ between the fibres. {Recall the objects of $I_f$ are the $2$-simplices with long edge $f$, and the objects of $I_g$ are the $2$-simplices
with long edge $g$. So the isomorphism $F_\alpha$ sends a $2$-simplex with long edge $f$ to a $2$-simplex with long edge $g$. This forces $F_\alpha s_0(f) = s_0(g)$ and $F_\alpha s_1(f) = s_1(g)$, which is equivalent to saying that $F_\alpha$ is stretched and therefore $F_\alpha \ledge_{I_f} = \ledge_{I_g}$}.

Coming back to the general case, $\lambda \cong \mu$: we have the solid outer square:
\[
\begin{tikzcd}
\Delta^1 \ar[d, -act] \ar[r, -act, "\ledge_{I_f}"] \ar[rr, -act, bend left, "\ledge_{I_g}"] & I_f \ar[r, "F_\alpha"] & I_g \ar[d, "\mfunctor_g"]  \\
\Delta^n \ar[ru, dotted, "\phi_\lambda"] \ar[rru, dotted, "\phi_\mu"'] \ar[rr, "\mu"'] && X.
\end{tikzcd}
\]
The curved triangle commutes by the $1$-simplex case already treated. The dotted arrows then exist individually by Proposition \ref{3existencecell}. The triangle that these two 
dotted arrows form with $I_f \simeq I_g$ is now forced to commute, since 
$\Delta^1 \to \Delta^n$ is active and $\mfunctor_g$ is strict culf. 
\end{proof}

\setcounter{equation}{0}
We define a simplicial map $I \colon X \rightarrow U_X$, using the interval construction:
\begin{itemize}
\item the map $I$ sends an object $\lambda \in X_n$ to the pair $(I_f, \phi_\lambda)$ where  $f = \longt(\lambda)$ and $\phi_\lambda$ is the  $n$-simplex induced by $\lambda$ of Proposition \ref{3existencecell}.

\item the map $I$ sends an arrow $\alpha \colon \lambda \rightarrow \mu$ in $X_n$ to the isomorphism $F_\alpha \colon (I_f, \phi_\lambda) \rightarrow (I_g, \phi_\mu)$ induced  by $\alpha$ of Lemma \ref{lemma: functorial of I in arrows}. (As usual, $f = \longt(\lambda)$ and $g = \longt(\mu)$.)

\end{itemize}

\begin{propo}\label{propoculfix}
Let $X$ be a rigid decomposition groupoid. The simplicial map $I \colon X \rightarrow U_X$ is strict culf.
\end{propo}

\begin{proof}
Since $d_1$ is active, we have that $d_1$ is a discrete fibration by Lemma \ref{lemmadiscreteiso}. Therefore, as a consequence of Proposition \ref{culfcondition}, to prove that $I$ is culf it is enough to prove that the following diagram is a strict pullback
\begin{center}
\begin{tikzcd}
X_2 \arrow[r, "d_1"] \arrow[d, "I_2"']& X_1 \arrow[d, "I_1"] \\
(U_X)_2 \arrow[r, "d_1"']& (U_X)_1,
\end{tikzcd}
\end{center} 
which is equivalent to proving that the functor $G \colon X_2 \rightarrow (U_X)_2 \times_{(U_X)_1} X_1$ induced by the pullback property is an isomorphism. For each $\sigma \in X_2$, the object $G(\sigma)$ is equal to $((I_{d_1(\sigma)}, \phi_\sigma), d_1(\sigma))$ where $\phi_\sigma$ is given by Proposition \ref{3existencecell}. For a morphism $\alpha \colon \sigma \rightarrow \sigma'$ put $f = d_1(\sigma)$ and $g = d_1(\sigma')$, the morphism $G(\sigma)$ is equal to $(H_\alpha, d_1(\alpha))$. Here $H_\alpha \colon (I_f, \phi_\sigma) \rightarrow (I_g, \phi_{\sigma'})$ is the isomorphism given by Lemma \ref{lemma: functorial of I in arrows}.

Recall that $d_1$ is a discrete fibration, this together with Proposition \ref{3existencecell} allows to construct a functor $R \colon (U_X)_2 \times_{(U_X)_1} X_1 \rightarrow X_2$. For an object $(\phi_\sigma \colon \Delta^2 \actto I_f, f)$, the object $R(\phi_\sigma, f)$ is defined as $\mfunctor_f \phi_\sigma$ in $X_2$. For a morphism $(H, \overline{\alpha})$, where $H \colon (I_f, \phi_\sigma) \rightarrow (I_g, \phi_{\sigma'})$ and $\overline{\alpha} \colon f \rightarrow g$, the morphism $R(H, \overline{\alpha})$ is defined as the morphism $\alpha \colon \mfunctor_f \phi_\sigma \rightarrow \mfunctor_g \phi_{\sigma'}$ which is the lifting of the arrow $\overline{\alpha} \colon f \rightarrow g$ with respect to $\mfunctor_f \phi_\sigma$ and $\mfunctor_g \phi_{\sigma'}$. The lift is unique since $d_1$ is a discrete fibration. It is straightforward to verify that $R$ is the inverse of the functor $G$. Note that the diagram is also a homotopy pullback since it is a strict pullback and $d_1$ is a discrete fibration.
\end{proof}

\setcounter{equation}{0}

\begin{blanko}
{Compatibility of $\mfunctor$-maps and subdivided intervals}
\label{subsec: interval cons of interval}
\end{blanko}

Given a simplicial map 
$G \colon X\to Y$, there will be natural relationships between intervals in
$X$ and intervals in $Y$, but to compare them we need to step out to the 
global $U$,
leaving the realms of $U_X$ and $U_Y$. 

Lemma \ref{isoimageculf} can be proved in an alternative way as follows:

\begin{lem}\label{lem:3.18}
  For any simplicial map between rigid decomposition groupoids $G \colon X \to Y$, there is a 
  unique
  stretched map $G_f \colon I_f \actto I_{Gf}$, compatible with
  $\mfunctor$-maps. This means that the diagram 
\[\begin{tikzcd}
	{I_f} & {I_{Gf}} \\
	X & Y
	\arrow["{\mfunctor_f}"', from=1-1, to=2-1]
	\arrow["{\mfunctor_{Gf}}", from=1-2, to=2-2]
	\arrow["G"', from=2-1, to=2-2]
	\arrow[-act, "{G_f}", from=1-1, to=1-2]
	\arrow["{}"{description}, draw=none, from=1-1, to=2-2]
\end{tikzcd}\]  
commutes. If $G$ is culf then $G_f$ is invertible.
\end{lem}
\begin{proof}
  The unique stretched map is given by Lemma \ref{ortoghonalfactsystem}:
  \[
\begin{tikzcd}
\Delta^1 \ar[d, -act] \ar[rr, -act] && I_{Gf} \ar[d, "\mfunctor_{Gf}"]  \\
I_f \ar[rru, dotted] \ar[r, "\mfunctor_f"'] & X \ar[r, "G"'] & Y.
\end{tikzcd}
\]
If $G$ is culf, then the dotted arrow is culf too by Remark \ref{culfconditiondiagonalortoghonal}, and since it is both
culf and stretched, it is invertible as a consequence of Proposition \ref{propofactsystemInt}.
\end{proof}

\begin{lem}\label{lem:3.18-subdiv}
  For any simplicial map between rigid decomposition groupoids $G \colon X \to Y$, there is a 
  unique
  stretched map $G_f \colon I_f \actto I_{Gf}$, compatible with subdivision: if we start with
  $\lambda : \Delta^n \to X$ (with long edge $f$), then the triangle
\begin{equation}\label{G}
\begin{tikzcd}[column sep={2.5em,between origins}]
 & \Delta^n \ar[ld, -act, "\phi_\lambda"'] \ar[rd, -act, "\phi_{G\lambda}"] &  \\
I_f \ar[rr, -act, "G_f"'] \ar[d, "\mfunctor_f"'] && I_{Gf} \ar[d, "\mfunctor_{Gf}"] \\
X \ar[rr, "G"'] && Y
\end{tikzcd}
\end{equation}
commutes. 
\end{lem}

\begin{proof}
  Orthogonality (Lemma \ref{ortoghonalfactsystem}) for the square
  \[
\begin{tikzcd}[column sep={22pt,between origins}, row sep={22pt,between 
  origins}]
  &\Delta^1 \ar[rd] \ar[rddd, -act, bend right=18] \ar[rrrrrd, -act, bend left=15] 
  &&&&& \\
&&\Delta^n \ar[dd, -act] \ar[rrrr, -act] &&&& I_{Gf} \ar[dd, "\mfunctor_{Gf}"]  \\
\\
&&I_f \ar[rrrruu, dotted] \ar[rr, "\mfunctor_f"'] && X \ar[rr, "G"'] && Y
\end{tikzcd}
\]
gives a unique filler, which has to be $G_f$, since it is also a filler 
for the square starting at $\Delta^1$.
\end{proof}

Finally we need to establish also the corresponding result for isomorphisms 
in $X_n$:  
given $\lambda \cong \mu$ in $X_n$, Lemma~\ref{lemma: functorial of I in arrows} gives isomorphisms of 
(subdivided) intervals 
\begin{equation}\label{eq:5.12triang}
\begin{tikzcd}[column sep={2.5em,between origins}]
 & \Delta^n \ar[ld, -act, "\phi_{\lambda}"'] \ar[rd, -act, "\phi_{\mu}"] &  \\
I_f \ar[rr, "\sim"] && I_g
\end{tikzcd}
\qquad
\begin{tikzcd}[column sep={2.5em,between origins}]
 & \Delta^n \ar[ld, -act, "\phi_{G\lambda}"'] \ar[rd, -act, "\phi_{G\mu}"] &  \\
I_{Gf} \ar[rr, "\sim"] && I_{Gg}.
\end{tikzcd}
\end{equation}
\begin{lem}\label{lem:Giso}
Let $G \colon X \rightarrow Y$ be a simplicial map between rigid decomposition groupoids. For any $f \cong g$
in $X_1$, the diagram
  \begin{equation}\label{IfIg}
  \begin{tikzcd}
  I_{f} \ar[r, -act, "G_f"] \ar[d, "\sim"'] & I_{Gf} \ar[d, "\sim"]  \\
  I_{g} \ar[r, -act, "G_g"'] & I_{Gg}
  \end{tikzcd}
  \end{equation}
commutes. Here the horizontal arrows are given by Lemma \ref{lem:3.18} and the vertical arrows by Lemma \ref{lemma: functorial of I in arrows}. 
\end{lem}

\begin{proof}
  The diagram
  \[
  \begin{tikzcd}[column sep={4.5em,between origins}, row sep={4em,between origins}]
  \Delta^{n} \ar[rrr, bend right=15, "\phi_{G\mu}"']
  \ar[d, -act, "\phi_\lambda"'] \ar[rr, "\phi_{G\lambda}"]  
  \ar[rd, "\phi_{\mu}"] 
  && I_{Gf} \ar[r, "\sim"]  &I_{Gg}  \ar[d, "\mfunctor_{Gg}"] \\
  I_f \ar[r, "\sim"] & I_g \ar[r, "\mfunctor_g"'] & X \ar[r, "G"'] & Y
  \end{tikzcd}
  \]
  commutes: the middle pentagon region is \eqref{G}, and the triangles are 
  \eqref{eq:5.12triang}. Inside
  the outer square we have the following two dotted maps:
  \[
  \begin{tikzcd}[column sep={4.5em,between origins}, row sep={4em,between origins}]
  \Delta^{n} \ar[d, -act, "\phi_\lambda"'] \ar[rr, "\phi_{G\lambda}"]  && 
  I_{Gf} \ar[r, "\sim"] & 
  I_{Gg} \ar[d, "\mfunctor_{Gg}"]  \\
 I_f \ar[rru, dotted] \ar[r, "\sim"] & I_g \ar[rru, dotted] \ar[r, "\mfunctor_g"'] & X \ar[r, "G"'] & Y.
  \end{tikzcd}
  \]
  The two triangle-shaped regions with dotted arrows also commute: the
  leftmost triangle is the triangle part of \eqref{G} for $\lambda$, and
  the rightmost `triangle' is the square part of \eqref{G} for $\mu$. The
  dotted parallelogram is now forced to commute, since both composites in
  it are fillers for the outer square, and by orthogonality (Lemma \ref{ortoghonalfactsystem}) only one filler
  can exist as $\phi_\lambda$ is stretched and $\mfunctor_{Gg}$ is strict culf.
\end{proof}

So now we completely control the $G$-maps in each simplicial degree 
individually. 

We shall also establish the naturality in simplicial operators:
We have seen (in Lemma~\ref{precomposeinterval}) 
that for any $p \colon \Delta^{n'} \to \Delta^n$, there is 
induced a canonical culf map $c^p_\lambda \colon I_{f'} \to I_f$ compatible like this:
\begin{equation}
\begin{tikzcd}[column sep={2.5em,between origins}]
\Delta^{n'} \ar[rr, "p"] \ar[d, -act, "\phi_{\lambda'}"'] && \Delta^n 
\ar[d, -act, "\phi_{\lambda}"] \\
I_{f'} \ar[rr, "c^p_\lambda"'] && I_f. 
\end{tikzcd}
\end{equation}

The following lemma shows that these functorialities are compatible.

\begin{lem}\label{lem:GXY}
Let $G \colon X \rightarrow Y$ be a simplicial map between rigid decomposition groupoids. From the situation
  \[
  \begin{tikzcd}[column sep={2.5em,between origins}]
	\Delta^{n'} \ar[rr, "p"] \ar[rd, "\lambda'"'] && \Delta^n \ar[ld, 
	"\lambda"] \\
   & X \ar[d, "G"] & \\
   & Y, &  
  \end{tikzcd}
  \]
  we get a commutative square
  \[
  \begin{tikzcd}
  I_{f'} \ar[r, "c^p_\lambda"] \ar[d, "G_{f'}"'] & I_f \ar[d, "G_{f}"]  \\
  I_{Gf'} \ar[r, "c^p_{G\lambda}"'] & I_{Gf}
  \end{tikzcd}
  \]
  involving the maps from the previous functorialities.
\end{lem}

\begin{proof}
  The diagram
  \[
  \begin{tikzcd}[column sep={4.5em,between origins}, row sep={3em,between origins}]
  \Delta^{n'} \ar[dd, -act, "\phi_{\lambda'}"'] \ar[r, "p"]  & \Delta^n 
  \ar[r, -act, "\phi_{\lambda}"] &I_f \ar[d] \ar[r, "G_f"] & 
  I_{Gf} \ar[dd, "\mfunctor_{Gf}"]  \\
  && X \ar[rd] & \\
  I_{f'}  \ar[rru] \ar[r, "G_{f'}"'] & I_{Gf'} \ar[rr, 
  "\mfunctor_{Gf'}"'] && Y
  \end{tikzcd}
  \]
  commutes: the pentagon by construction of the $\phi$-maps (Proposition \ref{3existencecell}),
  and the two squares by Equation~\eqref{G}.
  
  The outer square has the following two dotted $c$-maps:
  \[
  \begin{tikzcd}[column sep={4.5em,between origins}, row sep={4em,between origins}]
  \Delta^{n'} \ar[d, -act, "\phi_{\lambda'}"'] \ar[r, "p"]  & \Delta^n 
  \ar[r, -act, "\phi_{\lambda}"] &I_f \ar[r, "G_f"] & 
  I_{Gf} \ar[d, "\mfunctor_{Gf}"]  \\
  I_{f'} \ar[rru, dotted] \ar[r, "G_{f'}"'] & I_{Gf'} \ar[rru, dotted]\ar[rr, 
  "\mfunctor_{Gf'}"'] && Y.
  \end{tikzcd}
  \]
  The two triangle-shaped 
  regions with dotted arrows commute by construction of the $c$-maps 
  (Lemma~\ref{precomposeinterval}).
  The dotted parallelogram is now forced to commute, since both 
  composites in it are fillers for the outer square, and only one filler 
  can exist, as $\phi_{\lambda'}$ is stretched and $\mfunctor_{Gf}$ is strict 
  culf.
\end{proof}

\begin{blanko}
{Interval construction of an interval}
\label{subsec: interval cons of interval}
\end{blanko}
\setcounter{equation}{0}

Let $A$ be a discrete algebraic interval (simplicial set), and consider a subdivision of it,
$a: \Delta^n \actto A$. This whole data describes an $n$-simplex in $U$,
which we denote $\underline a : \Delta^n \to U$. Note that the long edge 
of $\underline a$ is $A$ itself.

We can now apply Proposition \ref{3existencecell} to $\underline a$ (as an $n$-simplex in $U$) to get
\[
\begin{tikzcd}
\Delta^n \ar[r, -act, "\phi_{\underline a}"] \ar[rd, "\underline a"']& 
I^U_A \ar[d, "\mfunctor_A"]  \\
 & U.
\end{tikzcd}
\]

\begin{lem}\label{lem:AIA}
  There is a canonical isomorphism $A \simeq I^U_A$ compatible
  with the subdivision:
\begin{equation}\label{AIA}
\begin{tikzcd}[column sep={2.5em,between origins}]
 & \Delta^n \ar[ld, -act, "a"'] \ar[rd, -act, "\phi_{\underline a}"] &  \\
A \ar[rr, "\sim"] && I^U_A
\end{tikzcd}
\end{equation}
\end{lem}

\begin{proof}
  There is a canonical simplicial map $A \to \Dec_\bot\Dec_\top U$, given 
  by sending an $n$-simplex $\lambda \colon \Delta^n \to A$ to the corresponding 
  stretched $(n{+}2)$-simplex $\overline \lambda \colon \Delta^{n+2} \actto A$,
  interpreted as an $(n{+}2)$-simplex in $U$. This simplicial map clearly 
  factors through $I_A^U \to \Dec_\bot\Dec_\top U$. We claim that the 
  induced simplicial map $A \to I_A^U$ is an isomorphism. Indeed, 
  $(I^U_A)_n$ is by definition the strict pullback
  \[
  \begin{tikzcd}
  (I^U_A)_n \drpullback \ar[d] \ar[r] & 1 \ar[d, "\name{A}"]  \\
  U_{n+2} \ar[r, -act] & U_1
  \end{tikzcd}
  \]
  which is to say that it is the groupoid of 
  stretched maps $\Delta^{n+2}\actto A$, in turn isomorphic to the
  groupoid of general maps $\Delta^n \to A$, which is the groupoid $A_n$.
\end{proof}

\begin{lem}\label{lem:AIAnat}
For any stretched isomorphism of intervals $A \cong B$, we 
get from
Lemma~\ref{lemma: functorial of I in arrows} an isomorphism $I_A \cong I_B$. The statement is that
these isos are compatible, meaning that the diagram
\[
\begin{tikzcd}
A \ar[d] \ar[r] & I_A \ar[d]  \\
B \ar[r] & I_B
\end{tikzcd}
\]
commutes.
\end{lem}

\begin{proof}
  This follows since the isomorphisms involved
  $$
  (I^U_A)_n \simeq \Map^{\operatorname{str}}(\Delta^{n+2},A) \simeq 
  \Map(\Delta^n,A) = A_n
  $$
  are all natural in stretched isomorphisms $A \simeq B$.
\end{proof}

A simplicial map $G \colon X \rightarrow Y$ between decomposition groupoids is \emph{full and faithful} if for all objects $x, y \in X$ it induces an equivalence on the mapping groupoids
$$G_{x,y} \colon \Map_X(x,y) \rightarrow \Map_Y(Gx, Gy).$$

By the way it was defined $U_X$, we have a canonical simplicial map $\jmath \colon U_X \rightarrow {U}$, defined by $\jmath(I_f, \phi_\lambda) = (I_f, \phi_\lambda)$ for $(I_f, \phi_\lambda) \in (U_X)_n$ and $\jmath F = F$ for $F \in ({U_X})_n$. It is straightforward to prove the following result.

\begin{lem}\label{comparasionfunctor}
Let $X$ be a rigid decomposition groupoid. Then the simplicial map $\jmath \colon {U_X} \rightarrow {U}$ is full and faithful.
\end{lem}

\begin{rema}\label{rema: explication what is I}
G\'alvez-Carrillo, Kock and Tonks~\cite{GTK3} defined the culf \emph{classifying map} $I' \colon \simplexcategory_{/X} \rightarrow \mathcal{U}$. It takes an $n$-simplex $\lambda \colon \Delta^n \rightarrow X$ to an $n$-subdivided interval $\phi_\lambda \colon \Delta^n \actto I_f$ in ${\mathcal{U}}$ (or to the pair $(I_f, \phi_\lambda)$ in $U_n$). Here $f = \longt(\lambda)$ and $\simplexcategory_{/X}$  denotes the Grothendieck construction of $X$. 
In the present paper $I '$ is the map $(\jmath \circ I) \colon X \rightarrow U$,  since  for each $\lambda \in X_n$, we have that $(\jmath \circ I)(\lambda) = (I_f, \phi_\lambda)$ which is the same as $I'(\lambda)$. We will abuse notation and denote $I'$ as $I$ in Section \ref{sec:GKTconjecture}. 
\end{rema}

\begin{blanko}
{Comparison with a strictification of $U$} 
\label{subsec:strictification}
\end{blanko}

In this section, we briefly compare our local strict model $U_X$ with a global strictification $\tildeb{U}$ of $U$, proposed by the referee. 

There is a well-known construction that replaces a pseudo-functor into $\Grpd$ with a strict functor (see for example \cite[\S 6.4]{gambino_2008}). In the present case there is a very explicit combinatorial description of such a strictification. An inert map from $\Delta^k$ to $\Delta^n$ is completely  determined by the values of $0$ and $k$. So we will denote an inert map $\Delta^k \rightarrowtail \Delta^n$ as a pair $(i, j) \colon \Delta^k \rightarrowtail \Delta^n$ such that $0 \mapsto i$ and $k \mapsto j$. We denote by $P_n$ the poset of inert faces of $\Delta^n$. We define $\tildeb{U}_n$ to be the groupoid of liftings
\[\begin{tikzcd}
	& {\mathcal{U}} \\
	{P_n} & \simplexcategory.
	\arrow[from=2-1, to=2-2]
	\arrow["\dominio", from=1-2, to=2-2]
	\arrow[dashed, from=2-1, to=1-2]
\end{tikzcd}\]
This gives a whole family of stretched maps $\Delta^k \actto \mathcal{C}$, one for each $(i,j) \in P_n$, and squares
\[\begin{tikzcd}
	{\Delta^k} & {\Delta^{k'}} \\
	{\mathcal{C}_{ij}} & {\mathcal{C}_{i'j'}}
	\arrow[-act, from=1-1, to=2-1]
	\arrow[tail, from=1-1, to=1-2]
	\arrow["\CULF"', from=2-1, to=2-2]
	\arrow[-act, from=1-2, to=2-2]
\end{tikzcd}\]
for each map $(i,j)  \rightarrowtail (i',j')$. Here $\mathcal{C}_{ij}$ and $\mathcal{C}_{i'j'}$ are discrete algebraic intervals. For example an object in $\tildeb{U}_2$ is pictured as follows
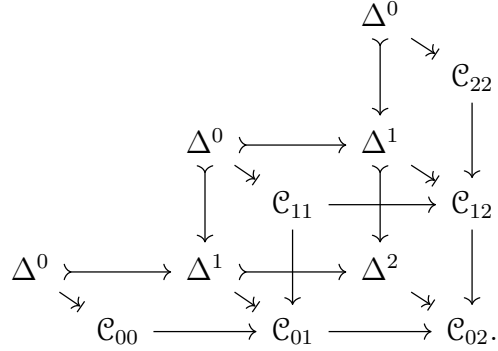
\[\begin{tikzcd}[row sep=tiny,column sep=tiny]
	&&&& {\Delta^0} \\
	&&&&& {\mathcal{C}_{22}} \\
	&& {\Delta^0} && {\Delta^1} \\
	&&& {\mathcal{C}_{11}} && {\mathcal{C}_{12}} \\
	{\Delta^0} && {\Delta^1} && {\Delta^2} \\
	& {\mathcal{C}_{00}} && {\mathcal{C}_{01}} && {\mathcal{C}_{02}}.
	\arrow[from=2-6, to=4-6]
	\arrow[from=4-6, to=6-6]
	\arrow[-act, from=3-5, to=4-6]
	\arrow[-act, from=1-5, to=2-6]
	\arrow[tail, from=1-5, to=3-5]
	\arrow[tail, from=3-5, to=5-5]
	\arrow[-act, from=5-5, to=6-6]
	\arrow[-act, from=5-3, to=6-4]
	\arrow[tail, from=5-3, to=5-5]
	\arrow[from=6-4, to=6-6]
	\arrow[-act, from=5-1, to=6-2]
	\arrow[tail, from=5-1, to=5-3]
	\arrow[from=6-2, to=6-4]
	\arrow[-act, from=3-3, to=4-4]
	\arrow[tail, from=3-3, to=5-3]
	\arrow[from=4-4, to=6-4]
	\arrow[tail, from=3-3, to=3-5]
	\arrow[from=4-4, to=4-6]
\end{tikzcd}\]
It is possible to define face and degeneracy maps between the groupoids $\tildeb{U}_n$ to assemble them into a strict simplicial groupoid $\tildeb{U}$. Informally, the face map $d_i$ acts by `erasing' all stretched maps containing an index $i$. The degeneracy maps $s_i$ repeat the $i$th row and the $i$th column. We have a canonical equivalence $\piend \colon \tildeb{U} \rightarrow U$ that on objects erases all the stretched maps except the last one. In case we consider only intervals that come from a fixed strict decomposition groupoid $X$, we have a strict simplicial groupoid $\tildeb{U}_X$ and a canonical equivalence $\piend' \colon \tildeb{U}_X \rightarrow U_X$. 

The interval construction $I \colon X \rightarrow U_X$ from \cite{GTK3} can easily be factored through $\tildeb{U}_X$ to give a refined interval construction $\tildeb{I} \colon X \rightarrow \tildeb{U}_X$ that sends an $n$-simplex $\lambda \colon \Delta^n \rightarrow X$ to $(I_{\lambda c}, \phi_{\lambda c})$ for each $c \colon \Delta^k \rightarrowtail \Delta^n \in P_n$. Here $\phi_{\lambda c}$ is given by Proposition \ref{3existencecell}. For example, for a $1$-simplex $f \colon \Delta^1 \rightarrow X$, the object $\tildeb{I}(f)$ in $(\tildeb{U}_X)_1$ is given by the following diagram
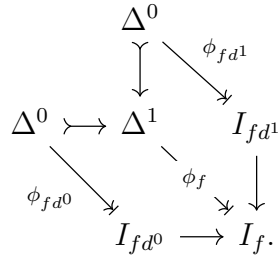
\[\begin{tikzcd}[column sep=small]
	& {\Delta^0} \\
	{\Delta^0} & {\Delta^1} & {I_{fd^1}} \\
	& {I_{fd^0}} & {I_f}.
	\arrow[-act, "{\phi_{fd^0}}"', from=2-1, to=3-2]
	\arrow[tail, from=2-1, to=2-2]
	\arrow[tail, from=1-2, to=2-2]
	\arrow[-act, "{\phi_{fd^1}}", from=1-2, to=2-3]
	\arrow[-act, "{\phi_{f}}"'{description}, from=2-2, to=3-3]
	\arrow[from=2-3, to=3-3]
	\arrow[from=3-2, to=3-3]
\end{tikzcd}\]
Note that since $U_X$ is already strict, all this refined data is redundant.

The four versions of $U$ and the four interval constructions are compatible, as indicated in the commutative diagram
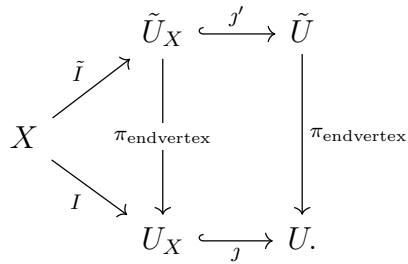
\[\begin{tikzcd}
	& {\tilde{U}_X} & {\tilde{U}} \\
	X \\
	& {U_X} & U.
	\arrow["I"', from=2-1, to=3-2]
	\arrow["\jmath"', hook, from=3-2, to=3-3]
	\arrow["{\jmath'}", hook, from=1-2, to=1-3]
	\arrow["{\tilde{I}}", from=2-1, to=1-2]
	\arrow["\piend"{description}, from=1-2, to=3-2]
	\arrow["\piend", from=1-3, to=3-3]
\end{tikzcd}\]
The original $U$ is hard to work with, as it is pseudo-simplicial instead of strict simplicial. Both
$\tildeb{U}_X$ and $U_X$ are practical because they are strict. ($\tildeb{U}_X$ is strict but is too redundant.) In this paper we prefer to work with $U_X$ since in any case most of the arguments are carried out locally at $X$, and in this situation it is the most direct approach.
 
\setcounter{equation}{0}

\section{G\'alvez--Kock--Tonks Conjecture}\label{sectionGKT}
\label{sec:GKTconjecture}
Let $\cDcmp$ denote the $\infty$-category of complete decomposition spaces and culf maps.   The construction of the complete decomposition groupoid $U$ was motivated by the following statement:  \\

\noindent \textbf{G\'alvez--Kock--Tonks Conjecture} \cite[\S 5.4]{GTK3}
For each decomposition space $X$, the space $\Map_{\cDcmp}(X,U)$ is contractible. \\

\noindent A partial result states that $\Map(X,U)$ is connected.
An $\infty$-version of this result is Theorem~5.5 in \cite{GTK3}. We 
include a proof here for two reasons. Firstly, we need to be more precise
regarding strictness conditions, and secondly, the proof in \cite{GTK3} 
does not actually give any argument for naturality in inert maps. As we 
shall see, this is a subtle issue, and the lack of argument in \cite{GTK3} 
may be considered a gap in that proof.

In our setting of rigid decomposition spaces, the relevant maps are
the strict culf maps. We are now concerned with culf maps to $U$.
Recall that $U \colon \simplexcategory^{\op} \to \Grpd$ is only pseudo-simplicial, 
but that it is actually strict on active maps. Furthermore,
for active $[n']\actto [n]$, the corresponding $U_n \to U_{n'}$ is a fibration.
The notion of strict culf map $J \colon X \to U$ is therefore still meaningful:
we do allow pseudo-simplicial maps, but they are still
required to be strict on the active part, and the naturality squares on 
active maps are required to be strict pullbacks. 
This implies in particular that for the unique
active map $\operatorname{long} \colon [1]\actto [n]$, and for every $n$-simplex 
$\lambda\in X_n$ with long edge $f = \operatorname{long}(\lambda)$, 
we have a strict equality
\begin{equation}\label{longJn=J1}
\operatorname{long}( J_n(\lambda) ) = J_1(f) .
\end{equation}
For general $p \colon [n'] \to [n]$ in $\simplexcategory$ (not necessarily 
active) it follows that
$J$ takes a strict triangle
  \[
  \begin{tikzcd}[column sep={2.5em,between origins}]
  \Delta^{n'} \ar[rr, "p"] \ar[rd, "\lambda'"'] && \Delta^n \ar[ld, "\lambda"]  \\
   & X
  \end{tikzcd}
\]
to
a commutative square
of the form
\begin{equation}\label{J-cart}
\begin{tikzcd}[column sep={2.5em,between origins}]
\Delta^{n'} \ar[rr, "p"] \ar[d, -act, "J_{n'}(\lambda')"'] && \Delta^n 
\ar[d, -act, "J_n(\lambda)"]  \\
J_1(f') \ar[rr, "e"']  && J_1(f)
\end{tikzcd}
\end{equation}
with $e$ culf. (This is to say, it is a
cartesian morphism for the right fibration 
$\mathcal{U}^{\operatorname{cart}} \to \simplexcategory$.)

\begin{teo} \label{conecctedconjecture}
 For any rigid decomposition groupoid $X$, the groupoid
  $\Map(X,U)$ of strict culf maps from $X$ to $U$ is connected. More
  precisely, for any strict culf map $J \colon X \to U$, there is a natural
  transformation (actually a modification) $\Gamma \colon J \isopil I$.
\end{teo}

\begin{proof}
Recall that $I$ was defined in Remark \ref{rema: explication what is I}.  There are three steps in the proof: Step 1 is to establish a canonical
  isomorphism $J_1(f) \cong I_1(f)$ for each $f\in X_1$, and show that
  this is natural in arrows in $X_1$. Step 2 is to exploit culfness to
  extend this isomorphism canonically to $J_n(\lambda) \cong I_n(\lambda)$
  for each $\lambda\in X_n$ (again naturally in $\lambda$). The
  construction in Step 2 actually shows that these isomorphisms are natural
  in active maps in $\simplexcategory$. But in any case, Step 3 consists in
  showing that the isomorphisms are natural in {\em all} maps in
  $\simplexcategory$, meaning that for any $p \colon [n'] \to [n]$ in
  $\simplexcategory$, the naturality square
  \[
  \begin{tikzcd}
\Delta^{n'} \ar[r, "p"] \ar[d, -act, "a'"'] & \Delta^n 
\ar[d, -act, "a"]  \\  J_1(f') \ar[r] \ar[d, "\sim"] & J_1(f) \ar[d, "\sim"]  \\
  I_1(f') \ar[r] & I_1(f)
  \end{tikzcd}
  \]
  commutes (as will be detailed).
  
{\bf Step 1.} Given $f\in X_1$,
we construct isomorphisms
$$
I^X_f \cong  I^U_{J_1(f)} \cong J_1(f) .
$$
Here the first isomorphism is an instance of Lemma~\ref{precomposeinterval}, where $J \colon X\to 
U$ plays the role of $G \colon X \to Y$. The second isomorphism is an instance of
Lemma~\ref{lem:AIA}, where $J_1(f)$ plays the role of $A$.

We should now argue why these isomorphisms are natural in arrows in 
$X_1$: given $f \cong g$, we need to check that this
naturality square commutes:
\[
\begin{tikzcd}
J_1(f) \ar[d] \ar[r] & I_1(f) \ar[d]  \\
J_1(g) \ar[r] & I_1(g)
\end{tikzcd}
\]
Since the vertical isomorphisms are composites of isos from 
Lemma~\ref{lem:3.18} and from Lemma~\ref{lem:AIA}, the naturality in maps 
inside $X_1$ follows from the naturality expressed by Lemma~\ref{lem:Giso} and
Lemma~\ref{lem:AIAnat}.

{\bf Step 2.} We now show that these isomorphisms $J_1(f)\cong I_1(f)$
extend to isomorphisms $J_n(\lambda) \cong I_n(\lambda)$ for each $n$,
using that both $I$ and $J$ are strict culf. We have
\[
\begin{tikzcd}
X_1 \ar[d] & X_n \dlpullback \ar[l, -act] \ar[d]  \\
(U_X)_1 & (U_X)_n \ar[l, -act]
\end{tikzcd}
\]
Since these horizontal maps are fibrations, we can describe the objects in 
$(U_X)_n$ as follows.
To give an object $J_n(\lambda)$ in $(U_X)_n$ is to give the underlying interval
$J_1(f)$ and an object in the fibre over $J_1(f)$. Since the square is a 
strict pullback, to give an object in the fibre of the bottom horizontal 
map is the same as giving an object in the fibre over $f$ of the top 
horizontal maps, i.e.~a subdivision, i.e.~an object $\lambda\in X_n$.
This same description holds for $I$.
So to give, for a fixed $\lambda\in X_n$, an isomorphism 
$J_n(\lambda) \isopil I_n(\lambda)$ is to give an isomorphism
$J_1(f) \isopil I_1(f)$, and keep the $\lambda$ in the fibres fixed.

As in Step 1, we should now argue why these isomorphisms are natural in arrows in 
$X_n$: given $\lambda\cong \mu$ in $X_n$, we need to check that this 
naturality square commutes:
\[
\begin{tikzcd}
J_n(\lambda) \ar[d] \ar[r] & I_n(\lambda) \ar[d]  \\
J_n(\mu) \ar[r] & I_n(\mu)
\end{tikzcd}
\]
The argument is the same as that given in degree $1$, but invoking now 
Lemma~\ref{lem:3.18-subdiv} instead of Lemma~\ref{lem:3.18}.

Note that the isomorphisms are natural in all 
active maps $[n] \actto [1]$ by construction, and therefore, by the 
standard prism-lemma argument, are also natural in all active maps.

{\bf Step 3.}
The final step is to show that the 
isomorphisms are also natural in inert maps, and in fact we prove 
uniformly that they are natural in all maps $p : [n'] \to [n]$ in 
$\simplexcategory$.
Given $\lambda: \Delta^n \to X$ (with long edge $f$) and a map $p: 
\Delta^{n'} \to \Delta^n$, put $\lambda':=\lambda\circ p$ (with long 
edge $f'$), so that we have
  \[
  \begin{tikzcd}[column sep={2.5em,between origins}]
  \Delta^{n'} \ar[rr, "p"] \ar[rd, "\lambda'"'] && \Delta^n \ar[ld, "\lambda"]  \\
   & X
  \end{tikzcd}
\]
which is sent by $J$ to
\begin{equation}\label{e=culf}
\begin{tikzcd}[column sep={3em,between origins}]
\Delta^{n'} \ar[rr, "p"] \ar[d, -act, "a'"'] && \Delta^n 
\ar[d, -act, "a"]  \\
J_1(f') \ar[rr, "{e \ \text{\rm culf}}"']  && J_1(f) .
\end{tikzcd}
\end{equation}

By Step 1 we have isomorphisms in each simplicial degree, 
which are strictly compatible with the 
subdivisions by Step 2, to give commutative triangles
\begin{equation}
\begin{tikzcd}[column sep={5em,between origins}]
 & \Delta^n \ar[ld, -act] \ar[d, -act] \ar[rd, -act] &  \\
J_1(f') \ar[r, "\sim"] & I^U_{J_1(f')} \ar[r, "\sim"] & I_1(f')
\end{tikzcd}
\qquad
\begin{tikzcd}[column sep={5em,between origins}]
 & \Delta^n \ar[ld, -act] \ar[d, -act] \ar[rd, -act] &  \\
J_1(f) \ar[r, "\sim"] & I^U_{J_1(f)} \ar[r, "\sim"] & I_1(f) .
\end{tikzcd}
\end{equation}

These diagrams together with Lemma~\ref{precomposeinterval} ensure that the following outer 
rectangle commutes:
  \[
  \begin{tikzcd}[column sep={5.5em,between origins}, row sep={4em,between origins}]
  \Delta^{n'} \ar[d, -act, "a'"'] \ar[r, "p"]  & \Delta^n 
  \ar[r, -act, "a"] &J_1(f) \ar[r, "\sim"] & I^U_{J_1(f)} \ar[r, "\sim"] & 
  I_1(f) \ar[d, "\mfunctor_f"]  \\
  J_1(f') \ar[rru, dotted, "e"] \ar[r, "\sim"] &I^U_{J_1(f')} 
  \ar[r, "\sim"] & I_1(f') \ar[rru, dotted, "c^p_\lambda"']\ar[rr, 
  "\mfunctor_{f'}"'] && X
  \end{tikzcd}
  \]
   We furthermore have the two diagonal dotted arrows indicated.
The leftmost triangle-shaped region is \eqref{e=culf}, and the right-most
triangle is given in Lemma \ref{precomposeinterval}.
  The composed dotted parallelogram is now forced to commute, since the 
  composites in it are fillers for the outer square, and only one filler 
  can exist since $a'$ is stretched and $\mfunctor_f$ is strict 
  culf.

  The dotted arrows are the cartesian lifts of $p$ to $J_1(f)$ and 
  $I_1(f)$, and the commutativity (by Lemma \ref{lem:GXY})  of 
  \[
  \begin{tikzcd}
\Delta^{n'} \ar[r, "p"] \ar[d, -act, "a'"'] & \Delta^n 
\ar[d, -act, "a"]  \\  J_1(f') \ar[r, "e"] \ar[d, "\sim"] & J_1(f) \ar[d, "\sim"]  \\
  I_1(f') \ar[r, "c^p_\lambda"'] & I_1(f)
  \end{tikzcd}
  \]
  now shows that the isomorphisms $J_n\isopil I_n$ are natural in
  $p$ (and thereby with the whole simplicial structure).
\end{proof}


\begin{blanko}
{Modifications}
\label{subsec:modif}
\end{blanko}

Theorem \ref{conecctedconjecture} implies that every natural transformation from $X$ to $U$ is isomorphic to $I$. Therefore, to prove the conjecture, we only need to prove that $I$ does not admit other self-modifications than the identity. Thus, we will introduce the notion of modification in the context in which we need it.

\noindent A modification between two natural transformations is a family of $2$-cells in the $2$-category of (small) categories that satisfies some coherence conditions as indicated in  the following definition:

\begin{defi}\cite[Definition 7.3.1]{Bor1} \label{definitionmodif}
Let $\mathcal{C}$ and $\mathcal{D}$ be two $2$-categories. Let $F,G \colon \mathcal{C} \rightarrow \mathcal{D}$ be two functors and $\alpha, \beta \colon F \Rightarrow G$ be two natural transformations from $F$ to $G$. A modification $\Gamma \colon \alpha \Rightarrow \beta$ assigns to each object $x$ in  $\mathcal{C}$ a $2$-cell $\Gamma_x \colon \alpha_x \rightarrow \beta_x$ of $\mathcal{D}$ compatibly with the 2-cell components of $F$ and $G$ in the sense of the equation
\begin{center}
\begin{tikzcd}
F(x) \arrow[r,shift left, bend left=20, ""{name=U}] \arrow[r,shift right, bend right=20, ""{name =D}] \arrow[d, "F(f)"']& G(x) \arrow[d, "G(f)"]
 \arrow[Rightarrow, from=U, to =D,shorten <=2pt, "\; \Gamma_x"] \\
 F(y) \arrow[r,""{name=D2}] & G(y)  \arrow[Rightarrow, from=D, to =D2, shorten <=4pt, "\; \beta_f"]
\end{tikzcd} = \begin{tikzcd}
F(x) \arrow[r,""{name=D2}] \arrow[d, "F(f)"']& G(x) \arrow[d, "G(f)"]
 \\
 F(y) \arrow[r,shift left, bend left=20, ""{name=U}] \arrow[r,shift right, bend right=20, ""{name =D}]  & G(y).  \arrow[Rightarrow, from=U, to =D, shorten <=2pt, "\; \Gamma_y"]
  \arrow[Rightarrow, from=D2, to =U, shorten <=4pt, "\; \alpha_f"]
\end{tikzcd}

\end{center}
\end{defi}
We are interested in the case where $\mathcal{C} = \simplexcategory$ and $\mathcal{D} = \Grpd$, and where $F = X$ and $G = U_X$, and where $\alpha$ and $\beta$ are both equal to $I$. In this case, Definition \ref{definitionmodif} can be written as follows.

\begin{defi}\label{defimodififorconj}
A modification $\Gamma \colon I \rightarrow I$ assigns to each $[n]$ in $\simplexcategory$ a natural transformation $\Gamma_n \colon I_n \rightarrow I_n$ in $\Grpd$ such that for each $n \geq 1$ the following equations hold for each $0 \leq i \leq n$ and $0 \leq j < n$
\begin{equation*}\tag{1}
 \begin{tikzcd}
X_n \arrow[rr,shift left, bend left=20, ""{name=U}] \arrow[rr,shift right, bend right=20, ""{name =D}] \arrow[d, "d_i"']& {} & (U_X)_n \arrow[d, "d_i"]
 \arrow[Rightarrow, from=U, to =D, shorten <=2pt, "\Gamma_n" description] \\
 X_{n-1} \arrow[rr,""{name=D2}] & {} & (U_X)_{n-1} 
\end{tikzcd} = \begin{tikzcd}
X_n \arrow[rr,""{name=D2}] \arrow[d, "d_i"']& &  (U_X)_n \arrow[d, "d_i"]
 \\
 X_{n-1} \arrow[rr,shift left, bend left=20, ""{name=U}] \arrow[rr,shift right, bend right=20, ""{name =D}]  & & (U_X)_{n-1}  \arrow[Rightarrow, from=U, to =D, shorten <=2pt, "\; \Gamma_{n-1}" description]
\end{tikzcd} 
\end{equation*}
\begin{equation*}\tag{2}
 \begin{tikzcd}
X_{n-1} \arrow[rr,shift left, bend left=20, ""{name=U}] \arrow[rr,shift right, bend right=20, ""{name =D}] \arrow[d, "s_j"'] & &  (U_X)_{n-1} \arrow[d, "s_j"]
 \arrow[Rightarrow, from=U, to =D, shorten <=2pt,  "\Gamma_{n-1}" description] \\
 X_{n} \arrow[rr,""{name=D2}] & & (U_X)_{n} 
\end{tikzcd} = \begin{tikzcd}
X_{n-1} \arrow[rr,""{name=D2}] \arrow[d, "s_j"']& & (U_X)_{n-1} \arrow[d, "s_j"]
 \\
 X_{n} \arrow[rr,shift left, bend left=20, ""{name=U}] \arrow[rr,shift right, bend right=20, ""{name =D}]  & & (U_X)_{n}.  \arrow[Rightarrow, from=U, to =D, shorten <=2pt, "\Gamma_n" description]
\end{tikzcd}
\end{equation*}
\end{defi}

\begin{rema}\label{rema:notationmod}
 We can define a modification $\Gamma \colon I \rightarrow I$ level by level, so let $\Gamma_n \colon I_n \rightarrow I_n$ be a component of the modification $\Gamma$. Given $\lambda \in X_n$, let $\phi_\lambda$ be the  $n$-simplex induced by $\lambda$ constructed in Proposition \ref{3existencecell} and $f = \longt(\lambda)$. The modification $\Gamma$ assigns to $\lambda$ an invertible morphism $\Gamma_n^\lambda \colon (I_f, \phi_\lambda) \rightarrow (I_f, \phi_\lambda)$ in $(U_X)_n$. The morphism $\Gamma_n^\lambda$ has associated an underlying map $\overline{\Gamma_n^\lambda} \colon I_f \rightarrow I_f$.
 
Let $p \colon [m] \actto [n]$ be an active map. By Remark \ref{rem: exp active pF}, we have that $p^*\overline{\Gamma_n^\lambda} = \overline{\Gamma_n^\lambda}$. This implies that 
\begin{equation*}
\overline{\Gamma_m^{\lambda p}} = \overline{\Gamma_n^\lambda}
\end{equation*}
where $\overline{\Gamma_m^{\lambda p}} \colon I_f \rightarrow I_f$ is the underlying map of $\Gamma_m^{\lambda p}$. The difference between $\Gamma_n^\lambda$ and $\Gamma_m^{\lambda p}$ is that the first one respects the $n$-subdivision $\phi_\lambda$ and the other respects the $m$-subdivision $\phi_{\lambda p}$.
\end{rema}

\begin{lem}\label{teocontractiblefordesset}
Let $X$ be a rigid decomposition groupoid. The mapping groupoid $\Map_{\cDcmp}(X, U_X)$ is contractible.
\end{lem} 

\begin{proof}
Theorem \ref{conecctedconjecture} shows that we only have to prove that $I$ does not admit other self-modifications $\Gamma$ than the identity. Let
$\lambda$ be an $n$-simplex in $X$ and put $f = \longt(\lambda)$. Let $\Gamma$ a modification, with components $\Gamma_n \colon I_n \rightarrow I_n$ and let $\overline{\Gamma_n^\lambda} \colon I_f \rightarrow I_f$ be the underlying map of  $\Gamma_n^\lambda \colon (I_f, \phi_\lambda) \rightarrow (I_f, \phi_\lambda)$ of Remark \ref{rema:notationmod}.

Since $\longt \colon [1] \actto [n]$ is an active map in $\simplexcategory$, by Remark \ref{rema:notationmod}, we have that 
\begin{equation}\label{tag1lemacontractible}\tag{1}
\overline{\Gamma_n^\lambda} = \overline{\Gamma_1^f}
\end{equation}
where $\overline{\Gamma_1^f} \colon I_f \rightarrow I_f$ is the underlying map of $\Gamma_1^f \colon (I_f, \phi_f) \rightarrow (I_f, \phi_f)$. On the other hand, given a morphism $\alpha \colon \sigma \rightarrow \overline{\sigma}$ in $I_f$, Lemma \ref{existssimplecesinterval} gives a stretched 3-simplex $\eta_\alpha \colon \Delta^3 \actto I_f$ such that 
\begin{equation*}\label{tag2lemacontractible}\tag{2}
d_\perp d_\top \eta_\alpha = \alpha.
\end{equation*}
The modification $\Gamma$ assigns to $\mfunctor_f \eta_\alpha$ an invertible map $\Gamma_3^{\eta_\alpha} \colon (I_f, \eta_\alpha) \rightarrow (I_f, \eta_\alpha)$ such that $\overline{\Gamma_3^{\eta_\alpha}} \eta_\alpha = \eta_\alpha$. Furthermore,
\begin{align*}
\overline{\Gamma_3^{\eta_\alpha}} (\alpha) & = \overline{\Gamma_3^{\eta_\alpha}}(d_\perp d_\top \eta_\alpha) & \tag{by Eq.~\ref{tag2lemacontractible}} \\
& =  d_\perp d_\top \overline{\Gamma_3^{\eta_\alpha}}(\eta_\alpha) & \tag{since $\overline{\Gamma_3^{\eta_\alpha}}$ is a sim. map} \\
& = d_\perp d_\top (\eta_\alpha) & \\
& = \alpha. 
\end{align*}
By Definition \ref{defimodififorconj}, we have the equality
\begin{center}
\begin{tikzcd}
X_3 \arrow[r,shift left, bend left=18, ""{name =U1}, "I_3"]\arrow[r,shift right, bend right=20, ""{name =D1}, "I_3"'] & (U_X)_3 \arrow[r, "d_1d_1"]& (U_X)_1  \arrow[Rightarrow,shift left, bend left=20, from=U1, to=D1, shorten <=5pt, "\Gamma_3"'] 
\end{tikzcd}
 $=$
\begin{tikzcd}
X_3 \arrow[r, "d_1d_1"]& X_1 \arrow[r,shift left, bend left=20, ""{name =U2}, "I_1"]\arrow[r,shift right, bend right=20, ""{name =D2}, "I_1"'] & (U_X)_1.
\arrow[Rightarrow,shift left, bend left=20, from=U2, to=D2, shorten <=5pt, "\Gamma_1"']
\end{tikzcd}
\end{center}
This equation implies that $d_1d_1(\Gamma_3^{\eta_\alpha}) = \Gamma_1^{f}$. Since $d^1d^1$ is active, we have that $\overline{\Gamma_3^{\eta_\alpha}} = \overline{\Gamma_1^{f}}$ by Remark \ref{rema:notationmod}.  Hence altogether, for each $\alpha \in I_f$
\begin{align*}
\overline{\Gamma_n^\lambda}(\alpha) & = \overline{\Gamma_1^f}(\alpha) & \tag{by Eq.~(\ref{tag1lemacontractible})} \\
& = \overline{\Gamma_3^{\eta_\alpha}}(\alpha) & \\
& = \alpha.
\end{align*}
Since $\Gamma_n^\lambda$ is the identity arrow for each $\lambda \in X_n$, we have that $\Gamma$ is the identity modification.
\end{proof}

\begin{teo}\label{corocontractibleU}
Let $X$ be a rigid decomposition groupoid. The mapping groupoid $\Map_{\cDcmp}(X, U)$ is contractible.
\end{teo}

\begin{proof}
Each natural transformation from $X$ to $U$ is isomorphic to $I$ by Theorem \ref{conecctedconjecture}.  We can factor  $I \colon X  \rightarrow U$ as 
\begin{center}
\begin{tikzcd}
X \arrow[rr, "I"] \arrow[rd, "I"'] &                      & U \\
                                   & U_X. \arrow[ru, hook, "\jmath"'] &  
\end{tikzcd}
\end{center}
Since $\jmath \colon U_X \rightarrow U$ is full and faithful (\ref{comparasionfunctor}), we have that $\jmath_{!} \colon \Map_{\cDcmp}(X, U_X) \rightarrow \Map_{\cDcmp}(X, U)$ is also full and faithful. Since $\Map_{\cDcmp}(X, U_X)$ is contractible (\ref{teocontractiblefordesset}) and $\Map_{\cDcmp}(X, U)$ is connected (\ref{conecctedconjecture}), it follows that $\Map_{\cDcmp}(X, U)$ is contractible.
\end{proof}

\addcontentsline{toc}{section}{References}

\begin{footnotesize}

\end{footnotesize}
\bibliographystyle{scplain}

  \begin{tabular}{@{}l@{}}%
    \scriptsize{DEPARTAMENT DE MAT\`{E}MATIQUES, UNIVERSITAT AUT\`{O}NOMA DE BARCELONA}\\
    \textit{E-mail address}: \texttt{wforero@mat.uab.cat}
  \end{tabular}

\end{document}